\documentclass{amsart}
\usepackage{latexsym,amsmath,amsthm,amssymb,amsfonts}
\usepackage[dvipsnames]{xcolor}
\usepackage{tikz}
\usepackage{color}
\usepackage{comment}
\usepackage{url}
\usepackage[colorlinks=true,
            linkcolor=violet,
            urlcolor=Aquamarine,
            citecolor=YellowOrange]{hyperref}

\newcommand{\abs}[1]{\left\vert#1\right\vert}
\newcommand{\norm}[1]{\left\Vert#1\right\Vert}
\newcommand{\Real}{\mathbb R}

\numberwithin{equation}{section}
\allowdisplaybreaks

\def\F{\mathcal{F}}
\def\G{\mathcal{G}}
\def\n{\mathfrak{n}}

\def\s{\,\,\,\,}
\def\R{\mathbb{R}}

\def\C{\mathbb{C}}

\def\so{\mathfrak{so}(\R^n)}
\newcommand{\set}[1]{\left\{#1\right\}}

\numberwithin{equation}{section}
\newtheorem{theorem}{Theorem}[section]
\newtheorem{lem}[theorem]{Lemma}
\newtheorem{thm}[theorem]{Theorem}
\newtheorem{pro}[theorem]{Proposition}
\newtheorem{cor}[theorem]{Corollary}
\newtheorem{defi}[theorem]{Definition}
\newtheorem{rem}[theorem]{Remark}

\def\lan{\langle}
\def\ran{\rangle}

\newcounter{Cnumber}

\makeatletter

\newcommand{\Rmnum}[1]{\expandafter\@slowromancap\romannumeral #1@}
\makeatother
\title[3-circle Theorem for Willmore surfaces II]{\bf 3-circle Theorem for Willmore surfaces II\\--{\tiny degeneration of the complex structure}}

\author{Yuxiang Li, Hao Yin, Jie Zhou}
\address{Department of Mathematical Sciences, Tsinghua University, People's Republic of China}
\email{liyuxiang@tsinghua.edu.cn}

\address{School of Mathematical Sciences, Unviersity of Science and Technology of China,People's Republic of China}
\email{haoyin@ustc.edu.cn}

\address{School of Mathematical Sciences, Capital Normal University, People's Republic of China}
\email{zhoujiemath@cnu.edu.cn}

\thanks{Y.Li is supported by National Key R\&D Program of China 2022YFA 1005400; H. Yin is supported by NSFC 12141105; J. Zhou is supported by NSFC 12301077.}
\date{}

\begin{document}

\begin{abstract}
    We study the compactness of Willmore surfaces without assuming the convergence of the induced complex structures. In particular, we compute the energy loss in the neck in terms of the residue and we prove that the limit of the image of the Gauss map is a geodesic in the Grassmannian $G(2,n)$ whose length can also be computed in terms of the residue. Moreover, we provide a family of explicit Willmore surfaces in $\R^3$ that illustrate the denegeration phenomenon involved in the above results.
\end{abstract}
\maketitle

\section{Introduction}
In the previous paper \cite{Li-Yin}, the authors used the three-circle theorem to study the blow-up of Willmore surfaces. As an application, they provided new proofs to the removable singularity theorem by Kuwert-Sch\"atzle and Rivi\`ere \cite{K-S1,K-S2,R}, as well as the energy identity by Bernard-Rivi\`ere \cite{B-R} when the induced complex structures converge. 

The convergence of the induced complex structure is very important for the above mentioned papers, because it implies the vanishing of some residue, which is the key to the proofs in \cite{Li-Yin, B-R, K-S2}. In spite of the difficulty, Laurain and Rivi\`ere  \cite{L-R} was able to show the validity of the energy identity under the assumption that the product of this residue with the length of the neck approaches zero. Very recently, Martino proved that energy identity is true when $n=3$ and the index is bounded \cite{M}. 

In this paper, we continue to investigate the neck part with possible non-zero residue. It is by now well known that the problem can be reduced to studying a sequence of Willmore immersions satisfying the following assumptions. (For the convenience of the readers, we include a brief outline of the reduction in the appendix.)

{\it
For any $k\in \mathbb N$, let
$f_k:[0,T_k]\times S^1\rightarrow \R^n$  be a conformal and Willmore immersion with $g_k=f_k^*(g_{\R^n})=e^{2u_k}(dt^2+d\theta^2)$, which satisfies
\begin{itemize}
\item[A1)]  $u_k=-mt+v_k$ for some $m\in \mathbb N$, with
$$
\|v_k(0,\theta)\|_{L^{\infty}(S^1)}\le 1, \s\|\nabla v_k\|_{L^\infty([0,T_k]\times S^1)}\rightarrow 0,
$$
\item[A2)]
$$\Theta_k=\sup_{t\in [0,T_k-1]}\int_{[t,t+1]\times S^1}|A_k|^2{d}V_{g_k}\rightarrow 0,
$$
\item[A3)] $\int_0^{2\pi}f_k(T_k-1,\theta) d\theta =0$.
\end{itemize}}

Under assumptions A1), A2) and A3), we would like to study the following questions:
{\it 
\begin{itemize}
\item[Q1)] What is the limit
$$
\lim_{m\rightarrow+\infty}\lim_{k\rightarrow+\infty}W(f_k,[0,T_k]\times S^1)?
$$
\item[Q2)] What is the limit of the imagine of Gauss map?\\
\end{itemize}
}

Before we state our main theorems, we need to define the residue mentioned above. 
 Given $c\in \mathbb{R}^n, S\in \so$(the set of $n \times n$ skew-symmetric matrices) and a conformal and Willmore immersion $f:[0,T]\times S^1\rightarrow \R^n$, we define
\begin{eqnarray*}
\tau_1(f,c)&=&\big(-2\int_{\{t\}\times S^1}\partial_tHd\theta-4\int_{\{t\}\times S^1}(H\cdot A_{ti})g^{ij}(\partial_jf)d\theta+\int_{\{t\}\times S^1}|H|^2\partial_tfd\theta\big)\cdot c\\
 \tau_2(f,S)&=&2\int_{\set{t}\times S^1} (H\cdot \partial_t (Sf) - \partial_t H\cdot Sf) d\theta - 4\int_{\set{t}\times S^1} (H\cdot A_{ti}) g^{ij} (\partial_j f\cdot Sf) d\theta\\ 
		&& + \int_{\set{t}\times S^1} \abs{H}^2 \partial_t f\cdot Sf d\theta.
\end{eqnarray*}
$\tau_1,\tau_2$ are  independent of $t$ (see \cite{Li-Yin}). Please note that $\tau_1,\tau_2$ are two of a series of quantities that are known either as residues, or conservation laws. They have long been known to be important in the study of Willmore surfaces \cite{K-S1, B-R, B, B-R0}.

For our purpose, we need to measure the size of $\tau_2$. Indeed, $\tau_2$ can be considered as a linear function over
$\so$. We take $\so$ as a subspace of $M(n)$, the linear space consisting of 
$n\times n$ matrices, with the given inner product 
$$
\lan A,B\ran=tr(AB^T),
$$
and put
$$
\|\tau_2(f,\cdot)\|=\sup_{\lan S,S\ran=1}\tau_2(f,S).
$$

Now, we can state our main theorem:
\begin{thm}
	\label{thm:main}
Let $f_k$ be a conformal and Willmore immersion which satisfies A1)-A3). Then
$$
\lim_{k\rightarrow+\infty}W(f_k,[0,T_k]\times S^1)=
\frac{1}{8\pi}\cdot \lim_{k\rightarrow+\infty}\|\tau_2(f_k,\cdot)\|^2T_k.
$$
Moreover, the  limit of the image of the Gauss map is a geodesic in $G(2,n)$
of length
$$
\frac{1}{4\sqrt{2}\pi}\cdot \lim_{k\rightarrow+\infty}\|\tau_2(f_k,\cdot)\|T_k.
$$
\end{thm}
\begin{rem}
In \cite{L-R},  Laurain and Rivi\`ere have shown that if $\lim_{k\rightarrow+\infty}\|\tau_2(f_k,\cdot)\|T_k=0$, then there is no energy loss. In \cite{M}, Martino proved that the the image of the conformal Gauss map converges to a geodesic when $n=3$. 
\end{rem}

\begin{rem}
In the field of harmonic maps, the corresponding results can be found in \cite{C-T, Chen-Li-Wang,Zhu}.
\end{rem}

Our main tool is the three-circle theorem suggested in the title of this paper. There are several versions of it, and in this paper, we will apply it several times to the mean curvature $H$, the second fundamental form $A$ and the immersion $f$. The basic idea is to prove decay for certain quantities of an almost harmonic functions. 

For some constant $L>0$ to be determined in later proofs, we assume that there is $m_k\in \mathbb N$ such that
\[
	T_k=m_k L.
\]
We define
\[
	Q_i=[(i-1)L,iL]\times S^1.
\]

In \cite{Li-Yin}, the authors have verified that when $u$ is a harmonic function defined on the cylinder, and if the Fourier expansion of
$u$ does not include the $e^{mt}$ terms, then $\int_{Q_i}e^{-2mt}u^2 dtd\theta$
 satisfies the inequality that we call the three-circle property. In \cite{Li-Yin}, $\tau_1=0$ and $\tau_2=0$ so that the terms involving $e^{mt}$ in the Fourier expansion of $H_k$ are almost absent.
In this paper, we do not have $\tau_2=0$ and hence to capture the obstruction to the three circle theorem, we make the following definitions for a general function $u(t,\theta)$
\begin{equation}\label{Fourier1}
\varphi(t)=\frac{1}{\pi}(\int_0^{2\pi}u\cos m\theta d\theta, \int_0^{2\pi}u\sin m\theta d\theta),
\end{equation}
\begin{equation}\label{g+}
g^+(u)=\varphi'+m\varphi,\s g^-(u)=m\varphi-\varphi',
\end{equation}
\begin{equation}\label{Fi}
\F_i(u)=\frac{g^+({(2i-1)L/2})}{2m}e^{mt-m(2i-1)L/2}\cdot(\cos m\theta,\sin m\theta),
\end{equation}
and
\begin{equation}\label{Gi}
\G_i(u)=\frac{g^-({(2i-1)L/2})}{2m}e^{-mt+m(2i-1)L/2}\cdot(\cos m\theta,\sin m\theta).
\end{equation}
When $u$ is harmonic, we have the following expansion
\begin{equation}\label{Fourier.harmonic}
u=at+b+\sum_{k=1}^\infty \underbrace{e^{kt}(a_k\cos k\theta+b_k\sin k\theta)}_{=:u_k}+\sum_{k=1}^\infty\underbrace{ e^{-kt}(a_k'\cos k\theta+b_k'\sin k\theta)}_{=:u_{-k}}.
\end{equation}
It is clear that
$$
\varphi(t)=(a_m,b_m)e^{mt}+(a_m',b_m')e^{-mt}=:\varphi_m+\varphi_{-m},
$$
$$
g^-(u)=2m(a_m',b_m')e^{-mt}, g^+(u)=2m(a_m, b_m) e^{mt},
$$
and hence
$$
\F_i(u)=e^{mt}(a_m\cos m\theta+b_m\sin m\theta)=u_m,
$$
$$
\G_i(u)=e^{-mt}(a_m'\cos m\theta+b_m'\sin m\theta)=u_{-m}.
$$
So, the operators $\F_i$ and $\G_i$ are designed to be the projections to the eigenspaces when restricted to harmonic functions. 

When $u$ is only almost harmonic functions, $\F_i(u)$ is then the obstruction of having the three circle lemma for $\int_{Q_i} \abs{u}^2 e^{-2mt}dtd\theta$. In other words, as long as $\F_i(u)$ is small compared with $\int_{Q_i}\abs{u}^2 e^{-2mt}dtd\theta$, the three circle argument still works.

Following the above discussion, we define
$$
\lambda_i(H_k)=\sqrt{\frac{\int_{Q_i}e^{-2mt}|\mathcal{F}_i(H_k)|^2 dtd\theta}
{\int_{Q_i}e^{-2mt}|H_k|^2 dtd\theta}}.
$$
When $\lambda_i$ is smaller than some universal constant, we will be able to use the 3-circle for the equation of mean curvature (see Theorem 
 \ref{3-circle.nonlinear1}). We also define
\[
\mu_i(f_k)=\frac{W(f_k,Q_i)}{\int_{Q_i}|A_k|^2dV_{g_k}},
\]
such that when $\mu_i$ is smaller than another universal constant, the 3-circle works for the equation of the Gauss map (see Lemma \ref{3circle A}). We will also need
$$
\nu_i(H_k)=\frac{\int_{Q_i}e^{-2mt}|H_k-\mathcal{F}_i(H_k)|^2 dtd\theta}
{\int_{Q_i}e^{-2mt}|\F_1(H_k)|^2 dtd\theta}.
$$
When $\lambda_i$ is bounded away from zero and $\nu_i$ is small, we may use 3-circle for the equation of $H$ again to see the decay of $\int_{Q_i} \abs{H_k- \F_i(H_k)}^2 e^{-2mt} dtd\theta$ (see Theorem \ref{2nd.asymp.}).

Using the ideas above, we are able to prove
\begin{thm}\label{main1}
For any $q\in(0,2)$, we can find $L$ such that the following holds.
Let $f_k$ be a conformal and Willmore immersion from $[0,m_k L]\times S^1$
into $\R^n$, which satisfies A1) and A2). Then for any $p>0$, either we have
\begin{equation}
	\label{eqn:trivialenergy}
\lim_{k\rightarrow+\infty}\sum_{i=1}^{m_k}\|A_k\|_{L^2(Q_i)}^p=0,
\end{equation}
or we can find integers $0<\mathfrak a_k<\mathfrak b_k<m_k$ with $\mathfrak{b}_k-\mathfrak{a}_k\rightarrow+\infty$ such that 
\begin{equation}
	\label{eqn:trivialends}
 \lim_{k\rightarrow+\infty} (\sum_{i=1}^{\mathfrak{a}_k}+\sum_{i=\mathfrak b_k}^{m_k})\|A_k\|_{L^2(Q_i)}^p=0
\end{equation}
and
\begin{equation}
	\label{eqn:goodmiddle}
\lim_{l\rightarrow+\infty}\lim_{k\rightarrow+\infty}\max_{\mathfrak a_k+l<i<\mathfrak b_k-l}(\nu_i+|\lambda_i-1|+|1-\mu_i|)=0.
\end{equation}
\end{thm}

Theorem \ref{main1} divides the cylinder into three parts and \eqref{eqn:trivialends} implies that the two parts outside  $[\mathfrak{a}_kL,\mathfrak{b}_kL]$ can be ignored. In fact, the contribution of energy of $f_k$ restricted to $[0, \mathfrak a_kL]\times S^1$ and $[\mathfrak b_kL, m_kL]\times S^1$ converges to zero as $k\to \infty$. The same is true when we compute the image of the Gauss map $\n_k$ in these regions. Our next goal is to study the behavior of $f_k$ inside this middle part.

To do so, we need to define the following notations. Let $f$ be any (parametrization of) surface defined on $Q_1\cup \cdots \cup Q_{m}$ and let $e^{2u}(dt^2+ d\theta^2)$ be the pullback metric. For any $i=1,\cdots, m$, we set
\begin{equation}
    \label{eqn:pci}
P_i=((i-1)L,0),\s c^i=u(P_i),
\end{equation}
\begin{equation}
    \label{eqn:fki}
{f}^{i}=e^{-c^i}\left(f((i-1)L+t,\theta)-\frac{1}{2\pi} \int_0^{2\pi}f((i-1)L,\theta) d\theta \right)
\end{equation}
and
$$
{f}^{i*}=e^{-c^i}\frac{1}{2\pi}\int_{0}^{2\pi}f((i-1)L+t,\theta)d\theta.
$$
The domain of $f^i$ depends on $i$ and it is simply a translation and scaling of $f$ focused on $Q_i$.

Let $f_k$ be given as in Theorem \ref{main1}. For any sequence $i_k$ with $i_k-\mathfrak{a}_k, \mathfrak{b}_k-i_k\rightarrow+\infty$, we set 
\begin{equation}
    \label{eqn:hatf}
\hat f_k=f_k^{i_k},
\end{equation}
and 
\begin{eqnarray}\label{def. hat u}
d\hat f_k\otimes d\hat f_k&=&e^{2\hat u_k}(dt\otimes dt+d\theta\otimes d\theta),\\
\label{def.hat.n}
\hat \n_k=\n(\hat f_k)&=&e^{-2\hat u_k}\partial_t\hat f_k\wedge \partial_\theta \hat f_k,\s \bar{\n}_k=\frac{1}{2\pi}\int_0^{2\pi}\hat{\n}_k(0,\theta)d\theta,\\
\label{def.hat.H.A}
\hat{H}_k&=&H(\hat{f}_k), \s \hat{A}_k=A(\hat{f}_k),\\%, \s d\hat{f}_k\otimes d\hat{f}_k=e^{2\hat{u}_k}(dt^2+d\theta^2)\\
\label{def.hat.epsilon}
\hat \epsilon_k&=&\sqrt{ W(\hat f_k,Q_1)/\pi L}=\sqrt{ W(f_k,Q_{i_k})/\pi L}.
\end{eqnarray}
\begin{rem}
    The notation $\hat f_k$ depends on a choice of the sequence $i_k$. For simplicity of notations, we simply write $\hat f_k$.
\end{rem}
By A1) and A2), it is known (for example \cite{Li-Yin}) that there is no interesting geometry in the limit of $\hat f_k$ because the limit is flat. More precisely, after choosing a new orthonormal basis in $\R^n$ and taking a subsequence, we may assume that $\hat f_k$ converges to
\begin{equation}
    \label{eqn:finfinity}
f_\infty=\frac{1}{m}e^{-mt}(\cos m\theta,\sin m\theta,0,0,\cdots,0).
\end{equation}
And therefore, the limits of $\hat H_k$ and $\hat A_k$ are zero. In the next theorem, we scale them up by $\hat \epsilon_k^{-1}$ and it turns out that the scaled sequence has nontrivial limit that can be explicitly written down.

Before the statement of the theorem, we need one more definition.
\begin{defi}
\label{defi:l}
    Let $(e_i)$ be a natural basis of $\R^n$ and assume $v_1,v_2$ are in ${\rm span} \set{e_3,\cdots,e_n}$. Then $\mathbb L_{v_1,v_2}$ is defined to be the linear map from $\so$ to $\R$, given by 
$$
\mathbb L=\mathbb L_{v_1,v_2}(S)=v_1\cdot(0,0,a_{31},a_{41},\cdots,a_{n1})+v_2\cdot(0,0,a_{32},a_{42},\cdots,a_{n2})
$$
for any $S=(a_{ij})\in \so$. Moreover, the norm of $\mathbb L_{v_1,v_2}$ is defined by
$$
\norm{\mathbb L_{v_1,v_2}}=\sup_{\sum\limits_{ij}|a_{ij}|^2=1} \mathbb L_{v_1,v_2}((a_{ij}))=\sqrt{\frac{|v_1|^2+|v_2|^2}{2}}.
$$ 
\end{defi}

\begin{thm}\label{neck.first.order}
	Let $f_k$ and $\mathfrak a_k, \mathfrak b_k$ be as in Theorem \ref{main1}. For any $i_k$ with $i_k-\mathfrak a_k, \mathfrak b_k-i_k\to \infty $, using the notations defined in \eqref{eqn:hatf}-\eqref{def.hat.epsilon}, there exists an orthonormal basis $(e_i)$ of $\R^n$ and some $v_1,v_2 \in \mbox{span}\set{e_3,\cdots ,e_n}$ such that
\begin{itemize}
\item[1)] $\hat H_k/\hat \epsilon_k$ converges to 
\begin{equation}
    \label{eqn:limith}
h= e^{mt}(v_1\cos{m\theta}+ v_2\sin{m\theta}).
\end{equation}
Moreover, we have $\abs{v_1}^2 + \abs{v_2}^2=1$.

\item[2)] $\frac{\hat \n_k-\bar{\n}_k}{\hat \epsilon_k}$ converges to
\begin{equation}
   \label{eqn:limitv} 
   \begin{split}
v=&\frac{1}{2}(e_1\wedge v_2-e_2\wedge v_1)t-\frac{1}{4m}\big(\cos 2m\theta(v_1\wedge e_2+v_2\wedge e_1)-\sin 2m\theta (v_1\wedge e_1-v_2\wedge e_2)\big). 
   \end{split}
\end{equation}
In particular,
\[
\frac{1}{\hat \epsilon_k}\abs{\frac{1}{2\pi}\int_0^{2\pi} \partial_t \hat{\n}_k d\theta} \qquad \text{converges to} \quad \frac{1}{2}.
\]

\item[3)] $\frac{\hat A_{k,tt}}{\hat \epsilon_k}$, $\frac{\hat A_{k,\theta\theta}}{\hat \epsilon_k}$ converges to
$$
\frac{1}{2}e^{-mt}(v_1\cos m\theta+v_2\sin m\theta),
$$
and $\frac{\hat A_{k,\theta}}{\hat \epsilon_k}$ converges to
$$
\frac{1}{2}e^{-mt}(v_1\sin m\theta-v_2\cos m\theta).
$$
\item[4)] $\hat K_k/\hat \epsilon_k^2$ converges to 
$$
\frac{e^{2mt}}{4}\big( (|v_1|^2-|v_2|^2)\cos 2m\theta+2\langle v_1, v_2\rangle \sin 2m\theta\big)
$$ 
and the normal curvature 
$\frac{\hat K^{\bot}_k}{\hat\epsilon_k^2}$ converges to zero.

\item[5)] If $f_k$ also satisfies A3), then $\frac{\tau_2(f_k,S)}{\hat \epsilon_k}$ converges to $-4\pi \mathbb L_{v_1,v_2}(S)$
and
$$
\lim_{k\rightarrow \infty}\frac{\|\tau_2(f_k,\cdot)\|}{\sqrt{W(f_k,Q_{i_k})}}=2\sqrt{\frac{2\pi }{L}},
$$
where the norm of $\tau_2(f_k,\cdot)$ is the natural one as a linear function from $\so$ to $\R$ (compare with Definition \ref{defi:l}). 
\end{itemize}
Here all above convergences (except the last one) are in $C^\infty_{loc}(\Real \times S^1)$ and we pass to subsequences if necessary.
\end{thm}

As an immediate corollary, we can use an argument by contradiction to show
\begin{cor}
For $f_k$ and $i_k$ as in Theorem \ref{main1}, we have
\[
\lim_{l\rightarrow+\infty}\lim_{k\rightarrow \infty}\max_{\mathfrak a_k+l<i<\mathfrak b_k-l}\left|\frac{\|\tau_2(f_k,\cdot)\|}{\sqrt{W(f_k,Q_{i})}}-2\sqrt{\frac{2\pi}{L}}\right| + \abs{\frac{1}{\epsilon_k} \abs{\frac{1}{2\pi}\int_0^{2\pi} \partial_t \hat \n_k d\theta} -\frac{1}{2} }=0.
\]
\end{cor}
This result is enough for the proof of the part of Theorem \ref{thm:main} on the loss of energy and the length of the neck. However, to study the limit of the image of the Gauss maps, we need to investigate the higher-order asymptotic properties of $f_k$.

\begin{thm}\label{main2}
Assume that $f_k$ satisfies A1)-A3) and Let $\mathfrak a_k, \mathfrak b_k$ be as given in Theorem \ref{main1}.
Then there exist $\mathfrak a_k'$ and $\mathfrak b_k'$ satisfying $\mathfrak a_k<\mathfrak a_k'<\mathfrak b_k'<\mathfrak b_k$, such that
\begin{itemize}
\item[1)] For any $p>0$,
$$
\lim_{k\to \infty} (\sum_{i=0}^{\mathfrak{a}_k'}+\sum_{i=\mathfrak b_k'}^{m_k})\|A_k\|_{L^2(Q_i)}^p =0.
$$
\item[2)]When  $\mathfrak a_k'<i<\mathfrak b_k'$, we have
$$
\int_{Q_i}e^{-2u_k}|f_k-\G_i(f_k)|^2dtd\theta\leq CW(f_k,Q_i)
$$
and
\begin{equation}\label{decay.H-FH}
\int_{Q_i}e^{2u_k}|H_k-\F_i(H_k)|^2dtd\theta\leq C(W(f_k,Q_i))^2.
\end{equation}
\end{itemize}
\end{thm}

The estimates proved in Theorem \ref{main2} allow us to prove more about the asymptotic limit of $f_k$ in the neck.
\begin{thm}\label{2nd.asymp.}
Let $f_k$ and $\mathfrak a_k', \mathfrak b_k'$ be as in Theorem \ref{main2}. For any $i_k$ with $i_k-\mathfrak a_k',\mathfrak b_k'-i_k \to \infty$, using the notations defined in \eqref{eqn:hatf}-\eqref{def.hat.epsilon}, we have 
\begin{itemize}
\item[1)] $\frac{\hat f_k-\G_1(\hat f_k)}{\hat\epsilon_k}$ converges in $C^\infty_{loc}(\R \times S^1)$ (up to the addition of a constant) to
$$
\phi=(-\frac{1}{2m}(t-\frac{L}{2})-\frac{1}{4m^2})e^{-mt}(v_1\cos m\theta +v_2\sin m\theta).
$$
\item[2)] $\frac{\hat H_k-\F_1(\hat H_k)}{\hat\epsilon_k^2}$ 
converges in $C^\infty(\R \times S^1)$ to
a function with image in $span\{e_1,e_2\}$.
\end{itemize}
Here $e_1,e_2$ and $v_1,v_2$ are given by Theorem \ref{neck.first.order} for the same subsequence.
\end{thm}

Theorem \ref{main2} and Theorem \ref{2nd.asymp.} are used in the proof of the claim about geodesic in Theorem \ref{thm:main}.

The paper is organized as follows. In Section \ref{sec:pde}, we prove some 3-circle lemmas on general elliptic equations that may be of independent interest. In Section \ref{sec:limit}, we prove Theorem \ref{main1} and Theorem \ref{neck.first.order}. In Section \ref{sec:higher}, we prove Theorem \ref{main2} and Theorem \ref{2nd.asymp.}. We put these together in Section \ref{sec:proof} to prove Theorem \ref{thm:main}. Finally, we present a family of explicitly defined Willmore surfaces and use them to construct a sequence $f_k$ satisfying A1)-A3). Indeed, they are the model case for the asymptotic behaviors that are studied in this paper.

\section{3-circle lemma for elliptic equation}
\label{sec:pde}
In this section, we first define what we mean by 3-circle lemma and then prove various versions of it. 

\begin{defi}
\label{def:3circle}
Let $Q_i=[(i-1)L,iL]\times S^1$ and $q>0$.
For each $Q_i$, we attach some quantity $\Phi_i\geq 0$. We say $\Phi_i$ satisfies the 3-circle 
lemma on $Q_1\cup Q_2\cup Q_3$ for $(q,L)$, if 
\begin{equation}\label{eqn:qphi}
	\Phi_2\leq e^{-qL} (\Phi_{1}+\Phi_{3}).
\end{equation}
\end{defi}

Throughout this paper, $L$ depends on $q$. Hence, we may assume $e^{qL}>2$. Once $q$ and $L$ are fixed in some context, we set
\begin{equation}
    \label{eqn:qprime}
    q'=q-\frac{\log 2}{L}.
\end{equation}
\begin{lem}\label{basic observation} 
Let $e^{qL}>2$ and $q'$ be given in \eqref{eqn:qprime}.
Assume  $\Phi_i$ satisifes 3-circle lemma on $Q_{i-1}\cup Q_{i}\cup Q_{i+1}$ for $(q,L)$.
Then 
\begin{itemize}
\item [1)]
either $\Phi_i\le e^{-q'L}\Phi_{i-1}$ or $\Phi_i\le e^{-q'L}\Phi_{i+1}$;
\item [2)] when $\Phi_{i-1}\le e^{-q'L}\Phi_i$, there holds $\Phi_{i}\le e^{-q'L}\Phi_{i+1}$;
\item [3)] when $\Phi_{i+1}\le e^{-q'L}\Phi_i$, there holds $\Phi_{i}\le e^{-q'L}\Phi_{i-1}$;
\end{itemize}
\end{lem}

\begin{proof}
For the proof of 1), if
\[
\Phi_i> e^{-q'L}\Phi_{i-1} \quad \text{and}\quad  \Phi_i> e^{-q'L}\Phi_{i+1},
\]
then 
\[
\Phi_i > \frac{1}{2}e^{-q'L} (\Phi_{i-1}+\Phi_{i+1})= e^{-qL}(\Phi_{i-1}+\Phi_{i+1}),
\]
which contradicts \eqref{eqn:qphi}.

For the proof of 2),  if $\Phi_{i-1}\le e^{-q'L}\Phi_{i}$, then \eqref{eqn:qphi} implies that
\[
(1-e^{-qL} e^{-q' L})\Phi_i \leq e^{-qL}\Phi_{i+1}.
\]
The definition of $q'$ and the assumption $e^{qL}\geq 2$ together imply that
\[
\frac{e^{-qL}}{(1-e^{-qL} e^{-q' L})}\leq e^{-q' L}.
\]
This concludes the proof of 2). The proof of 3) is similar.
\end{proof}

Here is a 3-circle lemma for harmonic functions. It is well known and we refer to Section 3 of \cite{Li-Yin} for a proof.
\begin{lem}\label{3-circle.harmonic}
For each $m\in \mathbb Z_+$ and $q\in (0,2)$, there is $L_0(m,q)>0$ such that if  a harmonic function defined on $Q$ has the  expansion
\begin{equation*}
u=a+bt+\sum_{k=1}^\infty\left((a_ke^{-kt}+b_ke^{kt})\cos k\theta+  (a_k'e^{-kt}+b_{k}'e^{kt})\sin k\theta\right)
\end{equation*}
with $b_m=b_m'=0$,
then $\int_{Q_2}u^2e^{-2mt} dtd\theta=0$ or
$$
\int_{Q_2}u^2e^{-2mt} dtd\theta<e^{-qL}\left(\int_{Q_1}u^2e^{-2mt} dtd\theta+\int_{Q_3}u^2e^{-2mt} dtd\theta\right),
$$
as long as $L>L_0(m,q)$.
\end{lem}

\subsection{A 3-circle lemma for linear elliptic equations}

In this subsection, we assume $u\in C^2$ is a solution of
\begin{align}\label{linear equation}
\Delta u=f
\end{align}
defined on $Q:=Q_1\cup Q_2\cup Q_3$.
In addition to the $\varphi$, $g^+$, $g^-$, $\F_i$ and $\G_i$ defined by \eqref{Fourier1}-\eqref{Gi} in the introduction. We also need to define
\begin{equation}\label{Fourier2}
\alpha(t)=\frac{1}{\pi}(\int_0^{2\pi}f\cos m\theta d\theta,\int_0^{2\pi}f\sin m\theta d\theta),
\end{equation}
and
\begin{align*}
\gamma_i(u)&=|g^+({(2i-1)L/2})|e^{-(2i-1)mL/2},\\
E_i(u)&=\int_{Q_i}|u|^2e^{-2mt}dtd\theta, \\
E_i^*(u)&=\int_{Q_i}|\F_i|^2e^{-2mt}dtd\theta=2\pi\gamma_{i}^2L/(2m)^2,\\
E_i^{\dagger}(u)&=\int_{Q_i}|u-\F_i(u)|^2e^{-2mt}dtd\theta.
\end{align*}

It follows from \eqref{linear equation} that
$$
(e^{-mt}g^+(t))'=\alpha(t)e^{-mt},\s (e^{mt}g^-(t))'=\alpha(t)e^{mt},
$$
from which we get 
\begin{equation}\label{ODE.g}
e^{-mt}g^+(t)-e^{-ms}g^+(s)=\int_{s}^t\alpha(\tau)e^{-m\tau}d\tau,\s  e^{mt}g^-(t)-e^{ms}g^-(s)=\int_{s}^t\alpha(\tau)e^{m\tau}d\tau.
\end{equation}

If $u$ is almost harmonic in the sense that the non-homogeneous term of \eqref{linear equation} is small (with respect to some weighted $L^2$ norm) compared to the weighted $L^2$ distance between $u$ and the space $\{e^{mt}(a\cos m\theta+b\sin m\theta)|a,b\in \mathbb{R}\}$, we still have a 3-circle lemma of the following form:
\begin{lem}\label{3-circle for relative harmonic function}
For any $q\in (0,2)$ and $L>L_0(m,q)$ (given in Lemma \ref{3-circle.harmonic}), there  exists $\epsilon_0=\epsilon_0(q,L)>0$ such 
 that if  $u\in C^2(Q)$ satisfies \eqref{linear equation}, and
\begin{align}\label{relatively almost harmonic}
\|e^{-mt}f\|_{L^2(Q)}^2\leq \epsilon_0 \int_{Q_2}e^{-2mt}|u-\F_2(u)|^2 dtd\theta,
\end{align}
 then
$\int_{Q_i}e^{-2mt}|u-\F_i(u)|^2 dtd\theta$
satisfies \eqref{eqn:qphi}.
\end{lem}
\begin{proof}
Assume the result  is false. Then there exists $u_k$, with $\Delta u_k=f_k$, such that $E_2^{{\dagger}}(u_k)>0$, 
$$
\epsilon_k=\frac{\|e^{-mt}f_k\|_{L^2(Q)}^2}{E_2^{{\dagger}}(u_k)}\rightarrow 0
$$
and
$$
e^{-qL}(E_1^{{\dagger}}(u_k)+E_3^{{\dagger}}(u_k))\leq E_2^{{\dagger}}(u_k)\neq 0.
$$

Define $v_k=u_k/\sqrt{E_2^{{\dagger}}(u_k)}$. By \eqref{ODE.g}, we have
$$
|\gamma_{i+1}(v_k)-\gamma_i(v_k)|\rightarrow 0 \quad \text{for} \quad i=1,2
$$
and
\begin{eqnarray*}
\|e^{-mt}(v_k-\F_2(v_k))\|_{L^2(Q)}&\leq&C+
\frac{\|e^{-mt}(\F_2(u_k)-\F_1(u_k))\|_{L^2(Q_1)}+\|e^{-mt}(\F_3(u_k)-\F_2(u_k))\|_{L^2(Q_3)}}{\|e^{-mt}(u_k-\F_2(u_k))\|_{L^2(Q_2)}}\\
&\leq& C(1+|\gamma_2(v_k)-\gamma_1(v_k)|+|\gamma_3(v_k)-\gamma_2(v_k)|)\\
&\leq& C'.
\end{eqnarray*}
Since
$$
\|\Delta (v_k-\F_2(v_k))\|_{L^2(Q)}=\|\Delta v_k\|_{L^2(Q)}\leq\frac{\|f\|_{L^2(Q)}}{\sqrt{E_2^{{\dagger}}(u_k)}}\rightarrow 0,
$$
$w_k=v_k-\F_2(v_k)$ is uniformly bounded in $W^{2,2}(K)$ for any $K\subset Q$ and hence  converges to a harmonic function $w$ in $W^{1,p}_{loc}(Q)$ with
$$
\|e^{-mt}w\|_{L^2(Q_2)}=1.
$$
Noting that $\F_2$ is a linear operator and  $ \F_2\circ \F_2=\F_2$, we know $\F_2(w_k)=0$. Moreover, for any $0<a<b<3L$, by taking $K=[a,b] \times S^1$, we obtain
$$\|\varphi(w_k)\|_{W^{2,2}([a,b])}\le C.$$
Since $W^{2,2}([a,b])$ is compactly embedded in $C^{1,\alpha}([a,b])$ for $\alpha<\frac{1}{2}$, we know $g^+(w_k)=m\varphi(w_k)-\varphi'(w_k)$ converges to $g^+(w)$ in $C^{0}([a,b])$ after passing to a subsequence. As a result, $ \F_2(w)=\lim_{k\to \infty}\F_2(w_k)=0$. 
Being a harmonic function, $w$ has an expansion similar to the one given in Lemma \ref{3-circle.harmonic}. $\F_2(w)=0$ implies that
$$ b_m(w)=b_m'(w)=0.$$
By Lemma \ref{3-circle.harmonic} and the fact that $w$ is not identically zero, we know  
$$
\int_{Q_2}w^2e^{-2mt} dtd\theta<e^{-qL}\left(\int_{Q_1}w^2e^{-2mt} dtd\theta+\int_{Q_3}w^2e^{-2mt} dtd\theta\right).
$$

However, by taking the limit and using the Fatou's lemma, we have
\begin{align*}
& e^{-qL}(\int_{Q_1}e^{-2mt}|w|^2 dtd\theta+\int_{Q_3}e^{-2mt}|w|^2 dtd\theta)\\
\leq&\lim_{k\rightarrow+\infty}e^{-qL}\left(\int_{[0,L]\times S^1}e^{-2mt}|w_k|^2 dtd\theta+\int_{[L,3L]\times S^1}e^{-2mt}|w_k|^2 dtd\theta\right)\\
\leq&\lim_{k\rightarrow+\infty}\int_{Q_2}e^{-2mt}|w_k|^2 dtd\theta\\
=&\int_{Q_2}e^{-2mt}|w|^2 dtd\theta.
\end{align*}
This is a contradiction and the proof is done.
\end{proof}

As a corollary, we derive the following decay properties on a long cylinder:

\begin{cor}\label{decay.linear1}
There exists $L_1$ depending on $q\in (0,2)$, and $\epsilon_0(q,L)>0$ depending on $q$ and a choice of $L>L_1$ such that the following holds.
Assume that $u$ is a $C^2$ map from $[0,lL]\times S^1$ to $\R^n$ satisfying \eqref{linear equation} and 
\begin{align}\label{comparison of harmoniousness}
\frac{1}{2}\|e^{-mt}f\|_{L^2(Q_{i+1})}\leq\|e^{-mt}f\|_{L^2(Q_{i})}\leq 2\|e^{-mt}f\|_{L^2(Q_{i+1})},\s \forall i=1,\cdots,l-1,
\end{align}
then there exist integers $1\le a,b\le l$ with $a-1\le b$ such that for $q'$ in \eqref{eqn:qprime},
\begin{itemize}
\item[1)]  $e^{-q'iL}\int_{Q_i}e^{-2mt}|u-\F_i|^2 dtd\theta$ is increasing on $[b+1,l]$, and 
$$
\int_{Q_i}e^{-2mt}|f|^2 dtd\theta\leq\frac{\epsilon_0}{10}\int_{Q_i}e^{-2mt}|u-\F_i|^2 dtd\theta,\quad \forall i\ge b+1.
$$

\item[2)] $e^{q'iL}\int_{Q_i}e^{-2mt}|u-\F_i|^2 dtd\theta$ is decreasing on $[1,a-1]$, and
$$
\int_{Q_i}e^{-2mt}|f|^2 dtd\theta\leq\frac{\epsilon_0}{10}\int_{Q_i}e^{-2mt}|u-\F_i|^2 dtd\theta, \quad \forall i\le a-1.
$$
\item[3)]  for any
$a\le i\le b$, we have
$$
\int_{Q_i}e^{-2mt}|f|^2 dtd\theta>\frac{\epsilon_0}{10}\int_{Q_i}e^{-2mt}|u-\F_i|^2 dtd\theta.
$$
\end{itemize}
\end{cor}

\begin{proof}
Fix $L_1>L_0$ (in Lemma \ref{3-circle.harmonic}) such that 
\[
e^{-q' L}<\frac{1}{5} \qquad \text{for}\quad L>L_1.
\]
Let $\epsilon_0=\epsilon_0(q,L)$ be the constant in Lemma \ref{3-circle for relative harmonic function} . 

If
\[
\|e^{-mt}f\|^2_{L^2(Q_k)}>\frac{\epsilon_0}{10}E_k^{{\dagger}}(u)
\]
for any $1<k<l$, we set
\begin{align*}
a=\left\{\begin{array}{cc}
 1   & \text{if} \s \|e^{-mt}f\|^2_{L^2(Q_1)}>\frac{\epsilon_0}{10}E_1^{{\dagger}}(u) \\
 2   & \text{if} \s \|e^{-mt}f\|^2_{L^2(Q_1)}\le \frac{\epsilon_0}{10}E_1^{{\dagger}}(u)
\end{array}\right.
\quad \quad \text{ and } \quad \quad 
b=\left\{\begin{array}{cc}
 l   & \text{if} \s \|e^{-mt}f\|^2_{L^2(Q_l)}>\frac{\epsilon_0}{10}E_l^{{\dagger}}(u) \\
 l-1   & \text{if} \s \|e^{-mt}f\|^2_{L^2(Q_l)}\le \frac{\epsilon_0}{10}E_l^{{\dagger}}(u).
\end{array}\right.
\end{align*}
One verifies directly that the statement 1)-3) hold.

In what follows, we assume that there exists some $1<k<l$ such that $\|e^{-mt}f\|^2_{L^2(Q_k)}\leq\frac{\epsilon_0}{10}E_k^{{\dagger}}(u)$, %\jie{3. the consideration on  $E_k^{{\dagger}}>0$ is not necessary in all the following argument. } 
which implies (using \eqref{comparison of harmoniousness}) 
\begin{equation}
    \label{eqn:ekdagger}
\int_{Q_{k-1}\cup Q_k\cup Q_{k+1}}e^{-2mt}\abs{f}^2 dtd\theta\leq\frac{\epsilon_0}{2} E_k^{{\dagger}}.
\end{equation}
Lemma \ref{3-circle for relative harmonic function} then shows that $E_k^{{\dagger}}(u)$ satisfies  \eqref{eqn:qphi}  on $Q_{k-1}\cup Q_{k} \cup Q_{k+1}$ for $q$ and $L$. By Lemma \ref{basic observation}, there are two cases. 

{\bf Case 1:} $E_k^{{\dagger}}\leq e^{-q'L}E_{k+1}^{{\dagger}}$. In this case, by \eqref{eqn:ekdagger} and \eqref{comparison of harmoniousness}, we have
$$
\int_{Q_{k+1}}e^{-2mt}|f|^2 dtd\theta\leq  e^{-q'L}\epsilon_0/2E_{k+1}^{{\dagger}}\le \epsilon_0/10E_{k+1}^{{\dagger}},
$$
which allows us to use Lemma \ref{3-circle for relative harmonic function} again to find that $E_k^{{\dagger}}$
satisfies the \eqref{eqn:qphi} condition on $Q_k\cup Q_{k+1}\cup Q_{k+2}$. By repeating the above argument, we find that 1) holds for any integer $i\in  [k,l]$.%\jie{4. I checked the index and modified the old version $j>k$ to the new version $i\ge k$. The same for case 2 and in the choice of a and b.  }

{\bf Case 2:} $E_k^{{\dagger}}\leq e^{-q'L}E_{k-1}^{{\dagger}}$. Similar to the proof of Case 1, but in a different direction, we find that 2) holds for any $i\in [1, k]$.

The above discussion shows that the set 
$$I:=\{1\le i\le l| \int_{Q_i}e^{-2mt}|f|^2 dtd\theta>\frac{\epsilon_0}{10}E_i^{\dagger}\}$$
equals to the intersection of $\{1,2,\cdots ,l\}$ with an interval. If $I$ is not empty, then $I=I\cap [i_0,i_1]$ for two integers $1\le i_0\le i_1\le l$.  We set $a=i_0$ and $b=i_1$. %Then, the conclusion holds. % to be the smallest $i$, such that $\int_{Q_i}e^{-2mt}|f|^2> \epsilon_0/10 E_i^{{\dagger}}$,\jie{6. I changed  $\int_{Q_i}e^{-2mt}|f|^2\ge \epsilon_0/10 E_i^{{\dagger}}>0$ to be $\int_{Q_i}e^{-2mt}|f|^2> \epsilon_0/10 E_i^{{\dagger}}$.}
 %and $b$  the largest $i$, such that $\int_{Q_i}e^{-2mt}|f|^2>\epsilon_0/10E_i^{{\dagger}}$. We get the conclusion.  
 If $I=\emptyset$, then there exists $1\le i_0<l$  such that  $e^{q'iL}\int_{Q_i}e^{-2mt}|u-\F_i|^2 dtd\theta$ is decreasing on $[1,i_0]$ and $e^{-q'iL}\int_{Q_i}e^{-2mt}|u-\F_i|^2 dtd\theta$ is increasing on $[i_0+1,l]$. We set $a=i_0+1$ and $b=i_0$. It is easy to check by the definition of $a,b$ that 1)-3) hold. Hence, the proof of the lemma is done.
 %then the conclusion holds. % for $i\le i_0$ and $2)$ holds for $i\ge i_0$. Then, we can choose $a=i_0+1$ and $b=i_0$,  then the conclusion holds.   
\end{proof}

The idea behind Lemma \ref{decay.linear1} is as follows. The key assumption in Lemma \ref{3-circle for relative harmonic function} is \eqref{relatively almost harmonic}. The subset of $i$ for which \eqref{relatively almost harmonic} fails is an interval (see the statement of 3)). To see this, we start from some $i$ for which \eqref{relatively almost harmonic} holds, meaning the obstruction to the three circle, $\norm{f e^{-mt}}_{L^2(Q_i)}$ is small compared with $\norm{e^{-mt} (u-\F_i)}_{L^2(Q_i)}$. Lemma \ref{relatively almost harmonic} implies that $\norm{e^{-mt} (u-\F_i)}_{L^2(Q_i)}$ grows exponentially in a nearby segment, say $Q_{i+1}$. Thanks to \eqref{comparison of harmoniousness}, the ratio between $\norm{f e^{-mt}}_{L^2(Q_{i+1})}$ and $\norm{e^{-mt} (u-\F_i)}_{L^2(Q_{i+1})}$ gets smaller. Hence, we can use Lemma \ref{3-circle for relative harmonic function} again for $Q_{i+1}$ until the end of the interval.

Using the same method, we can also prove a stronger version of Corollary \ref{decay.linear1}. The proofs are similar to the previous one and are omitted.

\begin{cor}\label{decay.linear3}For any $q\in (0,2), L>L_1(q)$,  there exists $\epsilon_0=\epsilon_0(q, L)>0$ such that the following holds. For  $u\in C^2([0,lL]\times S^1,\R^n)$ satisfying \eqref{linear equation}, we assume that
 there exists $\Lambda_i\geq 0$, with
$$
\int_{Q_i}e^{-2mt}|f|^2 dtd\theta\leq\Lambda_i,\s \frac{1}{2}\Lambda_{i+1}\leq\Lambda_i\leq 2\Lambda_{i+1},\s \forall i.
$$
Let $q'$ be given in \eqref{eqn:qprime}. Then there exists integers $0<a,b<l$ with $a-1\le b$, such that
\begin{itemize}
\item[1)] when $i\ge b+1$, $\Lambda_i\leq\frac{\epsilon_0}{10}\int_{Q_i}e^{-2mt}|u-\F_i|^2 dtd\theta$ and $e^{-q'iL}\int_{Q_i}e^{-2mt}|u-\F_i|^2 dtd\theta$ is increasing on $[b+1,l]$.
\item[2)] when $i\le a-1$, 
$\Lambda_i\leq\frac{\epsilon_0}{10}\int_{Q_i}e^{-2mt}|u-\F_i|^2 dtd\theta$ and $e^{q'iL}\int_{Q_i}e^{-2mt}|u-\F_i|^2 dtd\theta$ is decreasing $[1,a-1]$.
\item[3)]  for any
$a\le i\le b$, we have
$$
\Lambda_i>\frac{\epsilon_0}{10}\int_{Q_i}e^{-2mt}|u-\F_i|^2 dtd\theta.
$$
\item[4)] In particular,
\begin{align*}
\sum_{i=1}^{a-1} \Lambda_i+\sum_{i={b+1}}^l \Lambda_i\le \frac{\epsilon_0}{5}\big(\int_{Q_1}e^{-2mt}|u-\F_1|^2 dtd\theta+\int_{Q_l}e^{-2mt}|u-\F_l|^2 dtd\theta).
\end{align*}
\end{itemize}
\end{cor}

By making the variable substitution $t\rightarrow -t$, we arrive at the following conclusion:

\begin{cor}\label{decay.linear2}
For any $q\in (0,2), L>L_1(q)$,  there exists $\epsilon_0=\epsilon_0(q, L)>0$ such that the following holds. For  $u\in C^2([0,lL]\times S^1,\R^n)$ satisfying \eqref{linear equation},  we assume  that
there exists $\Lambda_i\geq 0$, with
$$
\int_{Q_i}e^{2mt}|f|^2 dtd\theta\leq\Lambda_i,\s \frac{1}{2}\Lambda_{i+1}\leq\Lambda_i\leq 2\Lambda_{i+1},\s \forall i.
$$
Let $q'$ be given in \eqref{eqn:qprime}. Then there exists $0<a,b<l$ with $a-1\le b$, such that
\begin{itemize}
\item[1)] when $i\ge b+1$, $\Lambda_i \leq\frac{\epsilon_0}{10}\int_{Q_i}e^{2mt}|u-\G_i|^2 dtd\theta$ and  $e^{-q'iL}\int_{Q_i}e^{2mt}|u-\G_i|^2 dtd\theta$ is increasing on $[b+1,l]$.
\item[2)] when $i\le a-1$, $\Lambda_i \leq\frac{\epsilon_0}{10}\int_{Q_i}e^{2mt}|u-\G_i|^2 dtd\theta$  and$e^{q'iL}\int_{Q_i}e^{2mt}|u-\G_i|^2 dtd\theta$ is decreasing on $[1,a-1]$ .
\item[3)]  for any
$a\le i\le b$, we have
$$
\Lambda_i>\frac{\epsilon_0}{10}\int_{Q_i}e^{2mt}|u-\G_i|^2 dtd\theta.
$$
\item[4)] In particular,
\begin{align*}
\sum_{i=1}^{a-1} \Lambda_i+\sum_{i=b+1}^l \Lambda_i\le \frac{\epsilon_0}{5}\big(\int_{Q_1}e^{2mt}|u-\G_1|^2 dtd\theta+\int_{Q_l}e^{2mt}|u-\G_l|^2 dtd\theta).
\end{align*}
\end{itemize}
\end{cor}

\subsection{3-circle lemmas for nonlinear equations}
\label{sub:nonlinear}
In this subsection, we assume that $u\in C^2(Q,\R^n)$ solves the equation
\begin{equation}\label{equation.u}
\Delta u=\phi(x,u,\nabla u)
\end{equation}
where
\begin{equation}\label{equation.u.phi}
\abs{\phi(x,u,\nabla u)}\leq \omega\cdot (|u|+|\nabla u|).
\end{equation}
for some nonnegative function $\omega$ defined on $Q$.

Obviously, when $\omega$ is small, $u$ is almost harmonic. In order to obtain 3-circle lemma for quantities like $E_i(u)$, we need to assume that $\gamma_i$ is small compared with $\sqrt{E_i(u)}$, because the assertion fails even for harmonic function $e^{m t}\cos \theta$. For that purpose, we define
$$
\lambda_i(u)=\frac{\gamma_i}{\sqrt{E_i(u)}}.
$$
Since $E_i(u)=0$ implies that $\varphi_u(t)=0$ for $t\in ((i-1)L,iL)$, $g^+(u)=\varphi'_u+m\varphi_u=0$ on $((i-1)L,iL)$ and hence $\gamma_i(u)=0$, we adopt the convention 
\begin{align}\label{convention}
\lambda_i(u)=0 \qquad \text{when} \quad E_i(u)=0.
\end{align}
As a result, when $\lambda_i>0$, we have $E_i>0$. This is used in the proof of Lemma \ref{gamma} and Theorem \ref{decay}. 

First, we prove that $E_i(u)$ satisfies the 3-circle lemma (see Definition \ref{def:3circle}) when $\|\omega\|_{L^\infty}$  is small and one of
$\lambda_i$ is sufficiently small.

\begin{thm}\label{3-circle.nonlinear1} For any $q\in (0,2)$ and $L>L_0(q)$, 
there exist $\epsilon_1$ and $\delta_0$ both depending on $q$ and $L$  such that if
$$
\|\omega\|_{L^\infty(Q)}<\delta_0,\s
\min\{\lambda_1,\lambda_2,\lambda_3\}\leq \epsilon_1,
$$
then $E_i$ satisfies the 3-circle lemma  on $Q$.
\end{thm}

\begin{proof}
Assume the result is false, then we can find
$u_k$ and some $i_0\in\{1,2,3\}$, with
$$
\|\omega_k\|_{L^\infty(Q)}\rightarrow 0,\s
\lambda_{i_0}(u_k)\rightarrow 0,
$$
and
\begin{equation*}
0\neq \int_{Q_2}|u_k|^2e^{-2mt}dtd\theta> e^{-q L}
\left(\int_{Q_1}|u_k|^2e^{-2mt}dtd\theta+
\int_{Q_3}|u_k|^2e^{-2mt}dtd\theta\right).
\end{equation*}
Set
$$
v_k=\frac{u_k}{\|u_ke^{-mt}\|_{L^2(Q_2)}}.
$$
We have
\begin{equation}\label{3.W2}
\int_{Q_2}|v_k|^2e^{-2mt}dtd\theta\geq e^{-q L}
\left(\int_{Q_1}|v_k|^2e^{-2mt}dtd\theta+
\int_{Q_3}|v_k|^2e^{-2mt}dtd\theta\right).
\end{equation}
and
$$
\int_{Q_2}|v_k|^2e^{-2mt}dtd\theta=1.
$$
Thus
$$
\int_Q|v_k|^2e^{-2mt}dtd\theta\leq
C.
$$
Since
$$
|\Delta  v_k|\leq \omega_k\cdot (|v_k|+|\nabla v_k|),
$$
by elliptic estimates we get
$$
\|\nabla v_k\|_{L^2(K)}\le C(K)\|v_k\|_{L^2(Q)}
$$
and
$$
\|v_k\|_{W^{2,p}(K)}\le C(K,p),
$$
for any compact subset $K$ inside $\mathring{Q}$ and $p\in (1,\infty)$.

Hence, $v_k$ converges to a harmonic function $v$ weakly in $W^{2,p}_{loc} (Q)$ and strongly in $C_{loc}^{1,\alpha}(Q)$ for any $p<\infty$ and $\alpha<1$, satisfying
$$
\int_{Q_2}|v|^2e^{-2mt}dtd\theta=1.
$$
Thus $v\neq 0$.

Moreover, since $\lambda_{i_0}(u_k)\rightarrow 0$,
$$
\gamma_{i_0}(v_k)=\lambda_{i_0}(v_k)\sqrt{E_{i_0}(v_k)}\rightarrow 0,
$$
then $\lambda_{i_0}(v)=0$. Since $v$ is harmonic, $\lambda_1(v)=\lambda_2(v)=\lambda_3(v)=0$, then $b_m(v)=b'_m(v)=0$  in the expansion of $v$ as required in Lemma \ref{3-circle.harmonic}.

However, by Fatou's lemma, we get
\begin{align*}
&e^{-q L}\left(\int_{Q_1}|v|^2e^{-2mt}dtd\theta+
\int_{Q_3}|v|^2e^{-2mt}dtd\theta\right)\\
\leq&\lim_{k\rightarrow+\infty}e^{-q L}\left(\int_{Q_1}|v_k|^2e^{-2mt}dtd\theta+
\int_{Q_3}|v_k|^2e^{-2mt}dtd\theta\right)\\
\leq&\lim_{k\rightarrow+\infty}\int_{Q_2}|v_k|^2e^{-2mt}dtd\theta=\int_{Q_2}|v|^2e^{-2mt} dtd\theta,
\end{align*}
which contradicts Lemma \ref{3-circle.harmonic} for $v$.
\end{proof}
\begin{rem}\label{maximal principle} 
If $E_1(u)=0$ (or $E_3(u)=0)$, by our convention, $\lambda_1=0$ (or $\lambda_3=0$ respectively), then the condition $\min\{\lambda_1,\lambda_2,\lambda_3\}\le \epsilon_1$ holds automatically.  So, Theorem \ref{3-circle.nonlinear1} implies that if $\|\omega\|_{L^\infty}<\delta_0$,  then $E_1(u)=0$ implies $E_2(u)\le e^{-qL} E_3(u)$. In particular, we have the following ``maximum principle": $E_1(u)=E_3(u)=0$ implies $E_2(u)=0$. 
\end{rem}

Since $\gamma_i(u)$ (in comparison with $\sqrt{E_i(u)}$) being small is important for the application of Theorem \ref{3-circle.nonlinear1}, we would like to study the change of $\gamma_i$. The idea is that when $\omega$ is small, $\gamma_i(u)$ does not change much.

\begin{lem}
\label{lem:gamma.diff}
There exists $C=C(m,L)$, such that if $\|\omega\|_{L^\infty}\le 1$, then  for $i=1$, $2$, we have
\begin{equation}\label{eqn:gamma.diff}
|\gamma_{i+1}(u)-\gamma_i(u)|\leq C\|\omega\|_{L^\infty(Q_i\cup Q_{i+1})}\sqrt{E_i(u)+E_{i+1}(u)}.
\end{equation}
\end{lem}

\begin{proof}
By \eqref{ODE.g}, we have
\begin{eqnarray*}
&&|\gamma_{i+1}(u)-\gamma_i(u)|\\
&\leq&C \int_{[(2i-1)/2L,(2i+1)/2L]\times S^1}\omega(|u|+|\nabla u|)e^{-mt} dtd\theta\\
&\leq& C \|\omega\|_{L^\infty([(2i-1)/2L,(2i+1)/2L]\times S^1)}
\sqrt{\int_{[(2i-1)/2L,(2i+1)/2L]\times S^1}(|u|^2+|\nabla u|^2)e^{-2mt} dtd\theta}.
\end{eqnarray*}
By elliptic estimate for \eqref{equation.u} and $\|\omega\|_{L^\infty}\le 1$,
$$
\int_{[(2i-1)/2L,(2i+1)/2L]\times S^1}|\nabla u|^2 dtd\theta\leq
C(1+\|\omega\|^2_{L^\infty})\int_{Q_i\cup Q_{i+1}}|u|^2 dtd\theta\le C \int_{Q_i\cup Q_{i+1}}|u|^2 dtd\theta,
$$
then
$$
|\gamma_{i+1}(u)-\gamma_i(u)|\leq C\|\omega\|_{L^\infty(Q_i\cup Q_{i+1})}
\sqrt{\int_{Q_i\cup Q_{i+1}}|u|^2e^{-2mt} dtd\theta}$$
\end{proof}

The next lemma bounds the ratio of $\lambda_i$ and $\lambda_{i+1}$. Note that in practise, we always have $\norm{\omega}_{L^\infty}$ small.
\begin{lem}\label{gamma}
Assume 
$$
\min\{\lambda_i,\lambda_{i+1}\}\geq \epsilon >0
$$ 
for some $\epsilon>0$ and $\norm{\omega}_{L^\infty(Q)}$ is smaller than some constant $\delta(\epsilon)$.
Then
\begin{equation}\label{decay.of.g+}
\frac{1- C\|\omega\|_{L^\infty(Q)}/\epsilon}{1+ C\|\omega\|_{L^\infty(Q)}/\epsilon}\leq \frac{\gamma_{i+1}}{\gamma_i}\leq \frac{1+ C\|\omega\|_{L^\infty(Q)}/\epsilon}{1- C\|\omega\|_{L^\infty(Q)}/\epsilon}
\end{equation}
\end{lem}
\begin{proof}
By the definition of $\lambda_i$, we have
\[
\sqrt{E_i(u)} \leq \frac{\gamma_i}{\lambda_i}.
\]
Hence by Lemma \ref{lem:gamma.diff} and our assumptions on $\lambda_i$ and $\lambda_{i+1}$,
$$
|\gamma_{i+1}-\gamma_i|\leq C\|\omega\|_{L^\infty(Q)}\left(\frac{\gamma_i}{\lambda_i}+\frac{\gamma_{i+1}}{\lambda_{i+1}}\right) \leq C \frac{\|\omega\|_{L^\infty(Q)}}{\epsilon} (\gamma_i+\gamma_{i+1}).
$$
Since $\lambda_i,\lambda_{i+1}>0$, by our definition of $\lambda_i$ and the convention, $\gamma_i \gamma_{i+1}\ne 0$. The proof is done by dividing both sides by $\gamma_i\gamma_{i+1}$.
\end{proof}

Lemma \ref{3-circle.nonlinear1} and Lemma \ref{gamma} would allow us to study the growth of $E_i(u)$ along the neck. In particular(see the first part of Theorem \ref{decay}), the neck will be divided into three parts. On the first and the third parts, we shall have $E_i(u)$ decreases/increases exponentially, while in the middle, we shall have a lower bound for $\lambda_i$. Our next version of 3-circle lemma is used to study the decay of $E_i^\dagger$ in this middle part. For that purpose, we need to define
$$ 
 \nu_i=\frac{E_i^{{\dagger}}(u)}{E_i^*(u)}.
$$
Here we adopt the convention that $\nu_i=0$ if  $E^{\dagger}_i(u)=E^*_i(u)=0$ and  $\nu_i=+\infty$ if $E_i^*(u)=0$ but $E^{\dagger}_i(u)>0$.

\begin{thm}\label{3-circle.nonlinear2}
Let $\epsilon_1$ be as given in Theorem \ref{3-circle.nonlinear1}.
For $q\in (0,2)$ and any $\epsilon>0$.
Assume that $u$ satisfies \eqref{equation.u} and  \eqref{equation.u.phi} and that
$$\min\{\lambda_1,\lambda_2,\lambda_3\}\geq\epsilon_1\quad \text{and}\quad
\max\{\nu_1,\nu_2,\nu_3\}>\epsilon>0.
$$
Then,
 for any $L>L_1(m,q)=\max\{L_0(m,q),\frac{2\log 2}{2-q}\}$,  there exists $\delta_1(q,L,\epsilon)>0$, such that if $\|\omega\|_{L^\infty(Q)}<\delta_1$,
then $E_i^{{\dagger}}$ and $\nu_i$ satisfies  \eqref{eqn:qphi}.
\end{thm}

\begin{proof}
Assume  \eqref{eqn:qphi}  is false for $E_i^{{\dagger}}$. We can find $u_k$ and $\omega_k$ satisfying \eqref{equation.u} and \eqref{equation.u.phi} such that
$$\|\omega_k\|_{L^\infty(Q)}\rightarrow 0,$$
\[
\min\set{\lambda_1(u_k),\lambda_2(u_k),\lambda_3(u_k)}\geq \epsilon_1,
\]
\begin{equation}
    \label{eqn:nui0}
\nu_{i_0}(u_k)>\epsilon, \quad \text{for some} \quad i_0=1,2,3
\end{equation}
and
\begin{equation}
    \label{eqn:dagger}
e^{-qL}(E_1^{{\dagger}}(u_k)+E_3^{{\dagger}}(u_k))< E_2^{{\dagger}}(u_k).
\end{equation}
Setting
$$
v_k=\frac{u_k}{\sqrt{E_2^{{\dagger}}(u_k)}},
$$
by \eqref{eqn:dagger}, we have
$$
\sum_{i=1}^3 E_i^{{\dagger}}(v_k)<C\quad \quad \text{ and } \quad \quad E_2^{\dagger}(v_k)=1.
$$
It follows from \eqref{eqn:nui0} that $E^*_{i_0}(v_k)$ is bounded.
By Lemma \ref{gamma},
$E_1^*(v_k)$, $E_2^*(v_k)$, $E_3^*(v_k)<C$ for large $k$. Since $\lambda_1,\lambda_2,\lambda_3\geq \epsilon_1$, we also get
$$
\sum_{i=1}^3 E_i(v_k)<C.
$$
Then $v_k$ converges to a harmonic function $v$ in $C^{1,\alpha}_{loc}(Q)$, and $\F_i(v_k)$ converges  to $\F_i(v)$ in $C^\infty(Q)$ and $E
_2^{\dagger}(v)=1$.
Since $v$ is harmonic, we have
$$
\F_1(v)=\F_2(v)=\F_3(v).
$$

Moreover, $E_2^{\dagger}(v)=1$ implies  $v-\F_2(v)\neq 0$ and we have
\begin{align*}
&e^{-qL}(\int_{Q_1}e^{-2mt}|v-\F_2(v)|^2 dtd\theta+
\int_{Q_3}e^{-2mt}|v-\F_2(v)|^2 dtd\theta)\\
=& e^{-qL}(\int_{Q_1}e^{-2mt}|v-\F_1(v)|^2 dtd\theta+
\int_{Q_3}e^{-2mt}|v-\F_3(v)|^2 dtd\theta)\\
\leq& \lim_{k\rightarrow+\infty}e^{-qL}\left(\int_{Q_1}e^{-2mt}|v_k-\F_1(v_k)|^2 dtd\theta+\int_{Q_3}e^{-2mt}|v_k-\F_3(v_k)|^2 dtd\theta\right)\\
\leq&\lim_{k\rightarrow+\infty}\int_{Q_2}e^{-2mt}|v_k-\F_2(v_k)|^2 dtd\theta\\
=&\int_{Q_2}e^{-2mt}|v-\F_2(v)|^2 dtd\theta.
\end{align*}
It contradicts Lemma \ref{3-circle.harmonic} for $v-\F_2(v)$ as long as $L>L_0(m,q)$ therein.
Hence, we have finished the proof for the claim for $E^\dagger_i$.

We now move on to show the claim for $\nu_i$. 

When $E_2^{\dagger}(u)=0$, then $\nu_2=0$ by our convention and there is nothing to prove.  When $E_2^{\dagger}(u)>0$, let $\tilde{q}=\frac{q+2}{2}$. We have just proved the existence of some $\delta'>0$ such that if $\norm{\omega}_{L^\infty(Q)}<\delta'$, then $E^\dagger_i$ satisfies the 3-circle lemma for $(\tilde{q},L)$. So, $E_1^{\dagger}+E_3^{\dagger}\ge e^{\tilde q L}E_2^{\dagger}(u)>0$.  Since $\min\{\lambda_1,\lambda_2,\lambda_3\}\ge \epsilon_1>0$, we also know $E^*_i(u)>0$ for $i=1,2,3$.    Applying Lemma \ref{gamma}, with $\epsilon=\epsilon_1$, we find $\delta''$ such that as long as $\norm{\omega}_{L^\infty}<\delta''$, we have
\[
\frac{1}{\sqrt{2}}\le\frac{\gamma_2}{\gamma_1}, \frac{\gamma_2}{\gamma_3}\le \sqrt{2}.
\]
Hence, if $\norm{\omega}_{L^\infty}<\delta_1:=\min \set{\delta',\delta''}$, then
\begin{align*}
\frac{\nu_2}{\nu_1+\nu_3}=\frac{\frac{E_2^{{\dagger}}}{E_2*}}{\frac{E_1^{{\dagger}}}{E_1^*}+\frac{E_3^{{\dagger}}}{E_3^*}}\le \frac{e^{-\tilde{q}L}}{\min\{\frac{E_2^*}{E_1^*},\frac{E_2^*}{E_3^*}\}}\le e^{-\tilde{q} L}\max\{\frac{\gamma_1^2}{\gamma_2^2},\frac{\gamma^2_3}{\gamma^2_2}\}\le 2 e^{-\tilde{q}L}\leq e^{-qL}.
\end{align*}
Here in the last line above, we used $L>L_1\geq \frac{2 \log 2}{2-q}$.
\end{proof}

Next, we prove the following decay properties on a long cylinder.

\begin{thm} \label{decay}
 Assume  $u$ is defined on $[0,lL]\times S^1$ and satisfies \eqref{equation.u} and \eqref{equation.u.phi}, where $L>L_1(m,q)$. Let $q'$ be given in \eqref{eqn:qprime} and $\epsilon_1$ be given in Theorem \ref{3-circle.nonlinear1}. Then there exists
 \[
 0\leq \mathfrak{a} < \mathfrak{b} \leq l+1
 \]
 such that
\begin{itemize}
\item[1)] there exists $\delta_2>0$, such that if $\|\omega\|_{L^\infty([0,lL]\times S^1)}<\delta_2$, then
\[
E_i\leq C e^{-i q' L} \sup_{k=1,\cdots,l}E_k \quad \text{for} \quad i\in [1,\mathfrak{a}],
\]
\[
E_i\leq C e^{-(l-i) q' L}\sup_{k=1,\cdots,l}E_k \quad \text{for} \quad i\in [\mathfrak{b},l],
\]
and
\[
\lambda_i \ge \epsilon_1 \quad \text{for} \quad i\in [\mathfrak{a}+1 , \mathfrak{b}-1] \text{ and } \lambda_i<\epsilon \text{ for } i\le [1,\mathfrak a]\cup [\mathfrak b,l].
\]

\item[2)] For any $2<s<\frac{1}{4}(\mathfrak{b}-\mathfrak{a})$, there exists
$\delta=\delta(s)$ such that if $\|\omega\|_{L^\infty([0,lL]\times S^1)}<\delta$, then
$$
\nu_i\leq 2(1+ \epsilon_1^{-2})e^{-q'(s-2)L/2},\s \forall \mathfrak{a}+s<i<\mathfrak{b}-s.
$$
\end{itemize}
\end{thm}
\begin{proof}
(1) %Assume there exists $3\leq i\leq l-3$ such that $\lambda_i<\epsilon_1$. If otherwise, either because there is no such $i$, or $i<3$, or $i>l-3$, then the conclusions in the first part of the theorem hold automatically by taking $\mathfrak{a}=3$ and $\mathfrak{b}=l-3$.\footnote{$3$ may not be necessary, but there does exist the extreme case that $i=1$ or $i=l$, so that we can not run 3-circle for even once.}
Let $$I=\{1\le i\le l| \lambda_i\ge \epsilon_1\}.$$
We first prove the property that 
\begin{align}\label{connectedness}
\text{either } I=\emptyset \quad  \text{ or } \quad I=[\alpha,\beta]
\end{align} 
for some integers $1\le \alpha\le\beta\le l$. 

We consider the following four extreme cases:
\[
I= [1,l], [1,l-1], [2,l], [2,l-1].
\]
In any of these four cases, the Part 1) holds trivially. Hence, we may assume the existence of some $1<i<l$ such that $i\notin I$, i.e. $\lambda_i<\epsilon_1$. The strategy of proving \eqref{connectedness} is to show that for any $i\notin I$, either $[1,i] \cap I = \emptyset$, or $[i,l]\cap I =\emptyset$.

 By Theorem \ref{3-circle.nonlinear1}, when $L>L_0(q)$ and $\delta_2$ sufficiently small, 
$\set{E_k}$ satisfies  \eqref{eqn:qphi} on $Q=Q_{i-1}\cup Q_{i}\cup Q_{i+1}$.
Then by Lemma \ref{basic observation},  we have either $E_i\leq e^{-q'L}E_{i+1}$
or $E_i\leq e^{-q'L}E_{i-1}$. 

In the first case ($E_i\leq e^{-q'L}E_{i+1}$), we claim that by taking $\delta_2$ small, we have $\lambda_{i+1}<\epsilon_1$.  In fact, if the claim is not true, then $\lambda_{i+1}\geq \epsilon_1>0$ and hence $\gamma_{i+1}>0$ and $E_{i+1}>0$ by our convention \eqref{convention}. By Lemma \ref{lem:gamma.diff}, we know 
\begin{align*}
|\gamma_{i+1}-\gamma_{i}|\le C\|\omega\|_{L^\infty(Q_i\cup Q_{i+1})}\sqrt{E_{i}+E_{i+1}}\le  C\delta_2 \sqrt{(1+e^{-q'L})E_{i+1}}\le C\delta_2 \sqrt{E_{i+1}},
\end{align*}
which implies that
\begin{align*}
|1-\frac{\gamma_i}{\gamma_{i+1}}|\le C\delta_2 \frac{\sqrt{E_{i+1}}}{\gamma_{i+1}}=\frac{C\delta_2}{\lambda_{i+1}}\le \frac{C\delta_2}{\epsilon_1}.
\end{align*}
 Hence, by taking $\delta_2=\delta_2(\epsilon_1,q',L)=\delta_2(q,L)$ small, we can have $C\delta_2/ \epsilon_1$ as small as we need so that
\[
\gamma_{i+1} < e^{q'L/2} \gamma_i.
\]
If $E_i=0$, then we know $\varphi_u(t)=0$ on $((i-1)L,iL)$ and hence $\gamma_i(u)=0$, which contradicts the above inequality. So, we know $E_i>0$. 
Since $E_i\leq e^{-q'L} E_{i+1}$, we get $\lambda_{i+1}<\lambda_i<\epsilon_1$, which is a contradiction, and the claim is proved.

As a result, if $i+2\le l$,  we may apply Theorem \ref{3-circle.nonlinear1} on $Q=Q_i\cup Q_{i+1} \cup Q_{i+2}$. This time, by Lemma \ref{basic observation} again, our assumption on $E_i\le e^{-q'L} E_{i+1}$ forces $E_{i+1}\le e^{-q'L} E_{i+2}$. And we can continue to do so until $Q=Q_{l-2}\cup Q_{l-1}\cup Q_l$. We then obtain that 
\begin{align}\label{small to right}
\max_{i\le k\le l}\lambda_k<\epsilon_1
\end{align}
and $e^{-kqL}E_k$ is increasing for $k=i,\cdots,l$. Namely, $[i,l]\cap I =\emptyset$.

Similar discussion applies to the second case ($E_i\leq e^{-q' L} E_{i-1}$), for which we obtain
\begin{align}\label{small to left}
\max_{1\le k\le i}\lambda_k<\epsilon_1
\end{align}
and  $e^{kqL}E_k$ is decreasing for $k=1,2,\ldots, i$. Hence, $[1,i]\cap I=\emptyset$. In summary, we have proved \eqref{connectedness}.

With \eqref{connectedness}, we now prove Part 1) of the theorem by the following discussion. 

If $I=\emptyset$, by \eqref{small to right} and \eqref{small to left}, we know there exists $0\le i_0\le l$ such that $e^{qkl}E_k$ is decreasing on $[1,i_0]$ and $e^{-qkl}E_k$ is increasing on $[i_0+1,l]$. Letting $\mathfrak a=i_0$ and $\mathfrak b=i_0+1$, we verify that Part 1) of the theorem holds. 

If $I=[\alpha,\beta]$ for some integers $1\le \alpha\le\beta\le l$, letting $\mathfrak a=\alpha-1$ and $\mathfrak b=\beta+1$, by \eqref{small to right} and \eqref{small to left}, we also have Part 1) of the theorem holds. 
\begin{comment}
consider the set
\[
\set{i| 3\leq i\leq l-3, \s \lambda_i<\epsilon_1, \s E_i\le e^{-q'L} E_{i+1}}.
\]
If it is empty, then set $\mathfrak{b}=l-3$; if not, let $\mathfrak{b}$ be its minimum.
For $\mathfrak{a}$, consider the set
\[
\set{i| 3\leq i\leq \mathfrak{b}, \s \lambda_i<\epsilon_1, \s E_i\le e^{-q'L} E_{i-1}}.
\]
If it is empty, then set $\mathfrak{a}=3$; if not, let $\mathfrak{a}$ be its maximum. Note that By Lemma \ref{basic observation}, $\mathfrak{a}< \mathfrak{b}$ and that if $\mathfrak{b}-1\geq i\geq \mathfrak{a}+1$, then $\lambda_i\geq \epsilon_1$.
\end{comment}
This finishes the proof of Part 1).

(2)  For $2<s< (\mathfrak{b}-\mathfrak{a})/4$ fixed, set 
\[
\epsilon=2e^{-(s-2)q'L/2}(1+\epsilon_1^{-2}).
\]
 If the claim in Part 2) is not true, then there is some $\mathfrak{a}+s<i<\mathfrak{b}-s$ such that $\nu_i> \epsilon>0$. Theorem \ref{3-circle.nonlinear2} determines $\delta$ depending on $\epsilon$ (hence $s$) such that as long as $\norm{\omega}_{L^\infty}<\delta$, $\set{E_k^{\dagger}}$ satisfies \eqref{eqn:qphi} on $ Q_{i-1}\cup Q_i\cup Q_{i+1}$ for $q$ and $L>L_1$. By Lemma \ref{basic observation}, either $E_i^{{\dagger}}\le e^{-q'L}E^{{\dagger}}_{i+1}$ or $E^{{\dagger}}_{i}\le e^{-q'L} E_{i-1}^{{\dagger}}$. Assume without loss of generality that $E_i^{{\dagger}}\le e^{-q'L}E^{{\dagger}}_{i+1}$.

Since $\min\{\lambda_{i-1},\lambda_i,\lambda_{i+1}\}\ge \epsilon_1$, and $\|\omega\|_{L^\infty}\le \delta$, 
by the convention \eqref{convention}, Lemma \ref{gamma} and $E^*_i=\frac{2\pi L\gamma_i^2}{(2m)^2}$, we know $E^*_{i+1}>0$ and 
\begin{align*}
1-C\frac{\delta}{\epsilon_1}\le \frac{E_i^*}{E_{i+1}^*}\le 1+C\frac{\delta}{\epsilon_1}.
\end{align*}
By asking $\delta$ to be small, we have $0<E^*_{i+1}\leq e^{q' L/2} E^*_i$. Together with $E_i^{{\dagger}}\le e^{-q'L}E^{{\dagger}}_{i+1}$, we find that $\nu_{i+1}\geq e^{q' L/2}\nu_i>e^{q' L/2}\epsilon$.

We then repeat the argument on $Q=Q_i\cup Q_{i+1}\cup Q_{i+2}$ and so on for exactly $s-2$ times, so that we get
\begin{equation}
    \label{eqn:is2}
\nu_{i+s-2} > e^{q'(s-2)L/2} \epsilon = 2(1+\epsilon_1^{-2}).
\end{equation}
However, since $i':=i+s-2<\mathfrak{b}$, we have $\lambda_{i'}\geq \epsilon_1$, which implies
$$
\nu_{i'}=\frac{\int_{Q_{i'}}|u-\mathcal{F}_{i'}(u)|^2e^{-2mt} dtd\theta}{\int_{Q_{i'}}|\mathcal{F}_{i'}(u)|^2e^{-2mt} dtd\theta}\leq 2\left(1+\frac{\int_{Q_{i'}}|u|^2e^{-2mt} dtd\theta}{\int_{Q_{i'}}|\mathcal{F}_{i'}(u)|^2e^{-2mt} dtd\theta}\right)\leq 2(1+\frac{1}{\lambda_{i'}^2})\le 2(1+\epsilon_1^{-2}).
$$
This is a contradiction to \eqref{eqn:is2} and the proof of Part 2) is done.
\end{proof}

\section{Asymptotic limit of the neck}
\label{sec:limit}
There are three subsections in this section.  The second and the last are devoted to the proof of Theorem \ref{main1} and Theorem \ref{neck.first.order} respectively. While in the first subsection, we discuss the techniques needed to prove the decay of second fundamental forms.

\subsection{3-circle lemma for the second fundamental forms}
We can prove a 3-circle lemma for the Dirichlet energy functional of the Gauss map, which is just the  $L^2$ norm of the second fundamental form. The basic ingredients of the proof are: the equation of the Gauss map
\begin{equation}\label{tension.gauss}
\Delta \n-A_{G(2,\n)}(d \n,d \n)=\nabla^\bot_1H\wedge f_2-\nabla^\bot_2H\wedge f_1,
\end{equation}
 the following inequality (see Lemma 4.2 in \cite{Li-Yin})
\begin{equation}\label{Po.A}
\left|\left|\frac{\partial \n_k}{\partial t}\right|^2-\left|\frac{\partial \n_k}{\partial \theta}\right|^2\right|\leq Ce^{2u_k}\|A_k\|_{g_k}|H_k|,
\end{equation}
and the following 3-circle lemma for harmonic function
\begin{lem}\label{lem:oldthreecircle} (Lemma A.1 in \cite{Li-Yin})
For any $q\in (0,2)$, there exists $L_0'(q)$ such that if $L>L_0$, the following holds. Define $Q_i:=[(i-1)L,iL] \times S^1$.
Let $v$ be a harmonic function from $Q=Q_1\cup Q_2\cup Q_3$. If
\begin{equation}\label{Po.equation}
\int_{Q_2}\left( |\partial_t v|^2- |\partial_\theta v|^2\right) dtd\theta=0,
\end{equation}
then 
\[
\int_{Q_2} |\nabla v|^2 dtd\theta\leq e^{-qL}\left( \int_{Q_1}|\nabla v|^2 dtd\theta + \int_{Q_3} |\nabla v|^2 dtd\theta \right).
\]
\end{lem}
By \eqref{Po.A}, \eqref{Po.equation} is almost true if $H$ is small when compared with $A$. Therefore, 
we can prove 
\begin{lem}\label{3circle A} (Lemma 4.3 in \cite{Li-Yin})
Let $L>L_0'(q)$.
Assume that $f:Q\rightarrow\R^n$ is a conformal Willmore immersion with $g=e^{2u}(dt^2+d\theta^2)$, and
$$
|\nabla u|<\beta,
$$
There exists $\epsilon_0=\epsilon_0(q,\beta)>0$ and $\delta=\delta(q,\beta)>0$, such that if
$$
\int_{Q_i}|A|^2dV_g<\epsilon_0,\s and\s 		\int_{Q} |H|^2 dV_g\leq \delta \int_Q |A|^2 dV_g
$$
then  $\int_{Q_i}|A|^2$  satisfies   \eqref{eqn:qphi}  on $Q$.
\end{lem}
The proof of Lemma \ref{3circle A}, Lemma \ref{lem:oldthreecircle} and \eqref{Po.A} can be found of \cite{Li-Yin}.

While the above lemma is good at showing the decay of $A$ when we know the decay of $H$, for this paper, we need another version that can be applied when $H$ does not change much. To state the result, we define
$$
\mu_i(f)=\frac{\int_{Q_i}|H|^2dV_g}{\int_{Q_i}|A|^2dV_g}.
$$
and we use the convention that $\mu_i=0$ if $\int_{Q_i}|A|^2dV_g=0$. 
\begin{lem}\label{3-circle.A.mu}
Let $L>L_2(m,q)=\max\{L_0',\frac{2\log 2}{2-q}\}$ .
	Assume that $f:Q\rightarrow\R^n$ is a conformal Willmore immersion with $g=e^{2u}(dt^2+d\theta^2)$ satisfying
$$
|\nabla u|<\beta
$$
and
\begin{align}\label{Willmore cmpr. ass.}
\int_{Q_i}|H|^2dV_{g}\leq 2\int_{Q_{i+1}}|H|^2dV_{g}\le 4\int_{Q_{i}}|H|^2dV_{g},\s i=1,2.
\end{align}
There exists $\delta_3=\delta_3(q,\beta)>0$ and $\epsilon_2=\epsilon_2(q,\beta)>0$, such that if
$$
\int_{Q_i}|A|^2dV_g<\delta_3,\s and\s \min\{\mu_1,\mu_2,\mu_3\}< \epsilon_2,
$$
then $\int_{Q_i}|A|^2$ and $\frac{1}{\mu_i}$ satisfies   \eqref{eqn:qphi}  on $Q$.
\end{lem}
\begin{proof}
We prove by contradiction. Assume the existence of a sequence of $f_k:Q\to \mathbb{R}^n$ with $f_k^*g_k=e^{2u_k}(dt^2+d\theta^2)$ satisfying $|\nabla u_k|\le \beta$,
\begin{align}\label{mean curvature equal}
\frac{1}{2} \int_{Q_i}|H_k|^2dV_{g_k}\le \int_{Q_{i+1}}|H_k|^2dV_{g_k}\le 2\int_{Q_i}|H_k|^2dV_{g_k}, i=1,2,
\end{align}
and
$$
\int_{Q_i}|A_k|^2dV_{g_k}\to 0 \text{ and } \min\{\mu_1(f_k),\mu_2(f_k),\mu_3(f_k)\}\to 0.
$$
After scaling, we may assume $\|u_k\|_{L^\infty}\le C$ without loss of generality.
Then $f_k$ converges smoothly  to a conformal Willmore immersion  $f:Q\to \mathbb{R}^n$ with $f^*g=e^{2u}(dt^2+d\theta^2)$ and $A_f=0$. Consider the Gauss map  $\n_k=e^{-2u_k}\partial_1f_k\wedge \partial_2f_{k}$, which satisfies the equation
$$
\Delta_{g_k}\n_k-A_{G(2,n)}(d\n_k,d\n_k)=\nabla_1^{\bot}H_k\wedge\partial_2f_k-\nabla_2^{\bot}H_k\wedge \partial_1f_k.
$$
If the lemma is not true, then either
\begin{align}\label{contra. ass. 1.}
\int_{Q_2}|A_k|^2dV_{g_k}>e^{-qL}(\int_{Q_1}|A_k|^2dV_{g_k}+\int_{Q_3}|A_k|^2dV_{g_k}),
\end{align}
or 
\begin{align}\label{contra. ass.2}
\frac{1}{\mu_2(f_k)}>e^{-qL}(\frac{1}{\mu_1(f_k)}+\frac{1}{\mu_3(f_k)}).
\end{align}
If \eqref{contra. ass. 1.} were to be true, then by the conformal invariance, we know
\begin{align}\label{contradiction assumption}
\int_{Q_2}|\nabla \n_k|^2dtd\theta>e^{-qL}(\int_{Q_1}|\nabla \n_k|^2dtd\theta+\int_{Q_3}|\nabla \n_k|^2dtd\theta).
\end{align}
Then by \eqref{mean curvature equal}, there holds
$$
\mu_2(f_k)\le 2 e^{qL}\min\{\mu_1(f_k),\mu_2(f_k),\mu_3(f_k)\}\to 0.
$$
Setting $\phi_k=\frac{\n_k}{\sqrt{\int_{Q_2}|\nabla \n_k|^2dtd\theta}}=:\frac{\n_k}{c_k}$, we have
\begin{equation}
    \label{eqn:phik}
\Delta_{g_k}\phi_k-A_{G(2,n)}(d\n_k,d\phi_k)=\nabla_1^{\bot}\frac{H_k}{c_k}\wedge\partial_2f_k-\nabla_2^{\bot}\frac{H_k}{c_k}\wedge \partial_1f_k,
\end{equation}
$$
\|\nabla \phi_k\|_{L^2(Q_2)}=1 \text{ and } \|\nabla \phi_k\|_{L^2(Q)}\le \sqrt{1+e^{qL}}\le C.
$$
In order to control the right hand side above, we use the equation of $H_k$, which by Lemma \ref{equ. struc.} takes the form
$$
\Delta H_k=\alpha_k H_k+\beta_k\nabla H_k,
$$
with $\|\alpha_k\|_{L^\infty}+\|\beta_k\|_{L^\infty}\to 0$. Here to verfiy the assumptions of Lemma \ref{equ. struc.}, we have used A1), A2) and the assumption that $\norm{u_k}_{L^\infty}\le C$.
Now, the elliptic estimate implies that for any $\eta>0$,
\begin{align*}
\|\nabla^{\bot}H_k\|_{L^\infty((\eta,3L-\eta)\times S^1)}&\le C(\eta)\|H_k\|_{L^2(Q)}\\
&\le 5C(\eta)\|H_k\|_{L^2(Q_2)}\\
&= 5C(\eta)\sqrt{\mu_{2}(f_k)}\|\nabla \n_k\|_{L^2(Q_2)}=5C(\eta)c_k\sqrt{\mu_{2}(f_k)}.
\end{align*}
Combining this with $|\partial_jf_k|=e^{u_k}\to e^u$, we know
\begin{align*}
\|\nabla_1^{\bot}\frac{H_k}{c_k}\wedge\partial_2f_k-\nabla_2^{\bot}\frac{H_k}{c_k}\wedge \partial_1f_k\|_{L^\infty((\eta,3L-\eta)\times S^1)}\le C\sqrt{\mu_2(f_k)}\to 0.
\end{align*}
Moreover, we also have
\[
\|A_{G(2,n)}(d\n_k,d\phi_k)\|_{L^2(Q)}\le C\|\nabla \n_k\|_{L^\infty}\|\nabla \phi_k\|_{L^2(Q)}\to 0
\]
thanks to A2) and the $\epsilon$-regularity.
In summary, we have proved that the $L^2$ norm of $\Delta \phi_k$ on $[\eta,3L-\eta]\times S^1$ converges to zero. 
Hence, 
$$
\|\phi_k-\bar{\phi}_k\|_{W^{2,2}[2\eta,3L-2\eta]}\le C(\eta), \quad \forall \eta>0.$$ %This further implies $\|B_k\|_{L^{p}(K)}\to 0$  and then $\|\phi_k-\bar{\phi}_k\|_{W^{2,p}(K)}\le C(K,p)$ for any $K\subset\subset Q$ and $p<2$.
Thus $\phi_k-\bar{\phi}_k$ converges in $W^{1,s}$ to some harmonic function $\phi$ for any $s<\infty$.
By \eqref{Po.A}, we know
\begin{align*}
|\int_{Q_2}(|\frac{\partial \phi_k}{\partial t}|^2-|\frac{\partial \phi_k}{\partial \theta}|^2) dtd\theta|
\le C\frac{\|A_k\|_{L^2(Q_2,dV_{g_k})}\|H_k\|_{L^2(Q_2,dV_{g_k})}}{c^2_k}=C \sqrt{\mu_2(f_k)}\to 0.
\end{align*}
 The convergence is strong enough so that 
 \begin{align*}
 \int_{Q_2}|\frac{\partial \phi}{\partial t}|^2dtd\theta=\int_{Q_2}|\frac{\partial \phi}{\partial \theta}|^2 dtd\theta.
  \end{align*}
The three-circle for harmonic function (Lemma \ref{lem:oldthreecircle}) implies that
$$
\int_{Q_2}|\nabla \phi|^2dtd\theta< e^{-qL}(\int_{Q_1}|\nabla \phi|^2dtd\theta+\int_{Q_3}|\nabla \phi|^2dtd\theta).
$$
 This contradicts \eqref{contradiction assumption}.  Therefore, \eqref{contra. ass. 1.} does not hold, that is,  $\int_{Q_i}|A_k|^2dV_{g_k}$ satisfies \eqref{eqn:qphi}. 

 It remains to show that \eqref{contra. ass.2} is not true. For this purpose, we set $\tilde q=\frac{q+2}{2}$ and by the first half that has been proved above, there exist $\delta_3$ and $\epsilon_2$ such that $\int_{Q_i}|A_k|^2dV_{g_k}$ satisfies \eqref{eqn:qphi} for $\tilde q$.

 By \eqref{Willmore cmpr. ass.}, if one of $\int_{Q_i}\abs{H_k}^2dV_g$ vanishes, so do the other two. In this case, $\frac{1}{\mu_i}=\infty$ for $i=1,2,3$ and there is nothing to prove. Hence, we may assume all $\int_{Q_i} \abs{H_k}^2 dV_g$ (hence $\int_{Q_i}\abs{A_k}^2 dV_g$) are positive. It follows from the following computation that \eqref{contra. ass.2} is not true.
$$
\frac{\frac{1}{\mu_2(f_k)}}{\frac{1}{\mu_1(f_k)}+\frac{1}{\mu_3(f_k)}}\le 2 \frac{\int_{Q_2}|A_k|^2dV_{g_k}}{\int_{Q_1}|A_k|^2dV_{g_k}+\int_{Q_3}|A_k|^2dV_{g_k}}\le 2e^{-\tilde qL}\le e^{-q L}.
$$
The proof of the lemma is complete.  
\end{proof}

\subsection{Proof of Theorem \ref{main1}}

For a conformal Willmore immersion $f:[-L,4L]\times S^1\rightarrow \R^n$ with $g=f^*(g_{\R^n})=e^{2u}(dt^2+d\theta^2)$, its mean curvature satisfies the equation
\begin{equation*}
2\Delta H + 4{\rm div}(H\cdot A_{pq}g^{ip}\partial_i f) -{\rm div}(|H|^2\nabla f)=0.
\end{equation*}
The main idea in the proof of Theorem \ref{main1} is to apply Theorem \ref{decay} to this equation. 
After a scaling in necessary, the above equation can be written as
\begin{equation}
    \label{WEF3C}
|\Delta \hat H|\leq \omega \cdot (|\hat H|+ |\nabla \hat H|)
\end{equation}
with 
\begin{equation}
    \label{eqn:smallomega}
\|\omega\|_{C^0([0,3L]\times S^1)}<C\|A\|_{L^2([-L,4L]\times S^1)}.
\end{equation}
This is a consequence of the $\epsilon$-regularity and we give details in the appendix (see Lemma \ref{equ. struc.}).

First, we prove the following slightly weaker version:

\begin{pro}\label{weak version}
Let $f_k$ be a conformal and Willmore immersion from $[0,m_kL]\times S^1$
into $\R^n$, which satisfies A1) and A2) and $L>\max\{L_1(m,q),L_2(m,q)\}$ (for $L_1(m,q)$ and $L_2(m,q)$ in Theorem \ref{decay} and Lemma \ref{3-circle.A.mu} respectively). Then for any $p>0$, there exists integers $0\le \mathfrak a_k<\mathfrak b_k\le m_k+1$, such that
$$
(\sum_{i=1}^{\mathfrak a_k}+\sum_{\mathfrak b_k}^{m_k})\|A_k\|_{L^2(Q_i)}^p<C(L,p,q)\Theta_k^\frac{p}{2}
$$
and (for $\epsilon_1$ in Theorem \ref{decay} and $\epsilon_2$ in Lemma \ref{3-circle.A.mu})
$$
\lambda_i(H_k)\geq \epsilon_1,\s \mu_i(f_k)\geq\frac{\epsilon_2}{2},\s when\s \mathfrak a_k+1\le i\leq \mathfrak b_k-1.
$$
Moreover, 
$$
\lim_{l\rightarrow+\infty}\lim_{k\rightarrow+\infty}\max_{\mathfrak a_k+l\leq i\leq \mathfrak b_k-l}(\nu_i(H_k)+|\lambda_i(H_k)-1|)=0.
$$ 
\end{pro}

\begin{proof}
The proof will be divided into several steps.\\

{\bf Step 1} By \eqref{WEF3C} and \eqref{eqn:smallomega}, we can apply Theorem \ref{decay} to $H_k$. Note that for $v_k$ as in A1), $L>1$  and  $\Theta_k=\sup_{t\in [0,m_kL-1]}\int_{[t,t+1]\times S^1}|A_k|^2\text{d}V_{g_k}$, we have 
$$
e^{2v_k(P_i)- 2\text{osc}_{Q_i} v_k}\int_{Q_i}e^{-2mt}|H_k|^2 dtd\theta\leq W(f_k,Q_i)\leq e^{2v_k(P_i)+2\text{osc}_{Q_i} v_k}\int_{Q_i}e^{-2mt}|H_k|^2 dtd\theta.
$$
 and
 $$\int_{Q_i}|A_k|^2dV_{g_k}\le C_L\Theta_k, \quad \forall 1\le i\le m_k.$$
 
We denote the $\mathfrak a$ and $\mathfrak b$ in Theoerm \ref{decay} by $a_k$ and $b_k$, then we have (for $k$ sufficiently large)
\begin{itemize}
\item[1)] $W(f_k,Q_i)\leq\left\{\begin{array}{ll}
                                                  Ce^{-q'iL}\Theta_k     &1\le i\le a_k\\
                                                  Ce^{-q'(m_k-i)L}\Theta_k   &b_k\le i\le m_k\end{array}\right.$
\item[2)] $\lambda_i(H_k)\geq\epsilon_1$ , for all $a_k+1\leq i\leq b_k-1$, and $\lambda_i(H_k)<\epsilon_1$ for all $i\le a_k$ or $i\ge b_k$
\item[3)] For $2<l<\frac{b_k-a_k}{4}$, we have
$$\nu_i(H_k)\leq 2e^{-q'(l-2)L/2}(1+\frac{1}{\epsilon_1^2}), \qquad \text{for all}\quad a_k+l<i<b_k-l.$$
\end{itemize}

{\bf Step 2.}{\it We show that 
\begin{equation}\label{eqn:decayA}
\int_{Q_i}|A_k|^2d\mu_k<
\left\{
\begin{array}{ll}
 C\Theta_k (e^{-iq'L}+e^{-(a_k-i)q'L})    &  \text{for} \s 1\le i\le a_k,\\
   C\Theta_k (e^{-q'(m_k-i)L}+e^{-(i-b_k)q'L})  & \text{for}\s b_k\le i\le m_k.
\end{array}
\right.
\end{equation} 
}
 For simplicity, we only discuss the case $1\le i\le a_k$. Note that \eqref{eqn:decayA} holds for $i=1$ and $i=a_k$ trivially. Hence, we can assume $1<i<a_k$. If
\begin{equation}\label{eqn:good}
\int_{Q_{i-1}\cup Q_i\cup Q_{i+1}}|A|^2 dV_g \leq \frac{1}{\delta}\int_{Q_{i-1}\cup Q_i\cup Q_{i+1}} \abs{H}^2 dV_g,
\end{equation}
for $\delta=\delta(q)$  as in Lemma \ref{3circle A}, then by item 1) of  Step 1, % the decay of $\int_{Q_i} \abs{H}^2 dV_g$ (by Theorem \ref{decay})
\begin{equation*}
\int_{Q_{i-1}\cup Q_i\cup Q_{i+1}}|A|^2 dV_g \leq \frac{e^{-q'iL}+e^{-q'(i-1)L}+e^{-q'(i+1)L}}{\delta}\cdot C\Theta_k\le C(q,L)e^{-iq'L}\Theta_k,
\end{equation*}
which implies \eqref{eqn:decayA}.
Hence, we may assume \eqref{eqn:good} is not true for $i$, then by Lemma \ref{3circle A} (for  $Q=Q_{i-1}\cup Q_i\cup Q_{i+1}$), we have either 
$\int_{Q_i}|A|^2 dV_g <e^{-q'L}\int_{Q_{i-1}}|A|^2 dV_g$ or $\int_{Q_i}|A|^2 dV_g<e^{-q'L}\int_{Q_{i+1}}|A|^2 dV_g$.
Without loss of generality, we assume the later one is true. Hence, the proof of \eqref{eqn:decayA} for $i$ is reduced to the proof of the same inequality with $i$ replaced by $i+1$. 

Now, we may repeat the above argument to consider $i+2$, $i+3$,$\cdots$, until either \eqref{eqn:good} holds for $i+l<a_k$ , or $i+l=a_k$. In the case $i+l<a_k$, we know 
\begin{align*}
\int_{Q_i}|A|^2d V_g\le \frac{e^{-lq'L}}{\delta}\int_{Q_{i+l-1}\cup Q_{i+l}\cup Q_{i+l+1}}|H|^2dV_g\le \frac{C\cdot e^{-q'(l+i+l-1)L}}{\delta}\Theta_k\le C\Theta_ke^{-iq'L}.
\end{align*}
In the case $i+l=a_k$, we know 
\begin{align*}
\int_{Q_i}|A|^2d V_g\le \frac{e^{-(a_k-i)q'L}}{\delta}\int_{Q_{a_k}}|A|^2dV_g\le C \Theta_ke^{-(a_k-i)q'L}.
\end{align*}
In both cases, we have \eqref{eqn:decayA} hold for $i$.

{\bf Step 3.} 
{\it There exist integers $l=l(q,L)$ such that if $b_k-a_k>2l+10$, then there exist $a_k',b_k'$ with $a_k+l<a_k'<b_k'<b_k-l$,  such that  
\begin{itemize}
\item [1)]  $e^{iqL}\int_{Q_i}|A_k|^2dV_{g_k}$ is  decreasing on $ [a_k+l+1,a_k']$;
\item [2)]  $e^{-iqL}\int_{Q_i}|A_k|^2dV_{g_k}$ is increasing on $ [b_k',b_k-l-1]$;
\item [3)] $\mu_i\geq \frac{\epsilon_2}{2}$ for any $a_k'< i< b_k'$.\\
\end{itemize}}

The proof of Step 3 uses Lemma \ref{3-circle.A.mu}. In order to verify the assumption \eqref{Willmore cmpr. ass.}, we need to prove the following claim first, in which we determined the constant $l(q,L)$.

{\bf Claim A.} There exists $l(q,L)$ such that \eqref{Willmore cmpr. ass.} holds for all $i\in [a_k+l+1,b_k-l-1]$ and sufficiently large $k$.

By item $2)$ of Step 1 and Lemma \ref{gamma}, we know for any $i=a_k+1,\cdots ,b_k-2$, there holds 
\begin{align}\label{E_i^* equal}
\frac{E^*_{i+1}(H_k)}{E^*_i(H_k)}=\frac{\gamma_{i+1}^2}{\gamma_i^2}\in [1-C\frac{\vartheta}{\epsilon_1}, 1+C\frac{\vartheta}{\epsilon_1}],
\end{align}
where
\[
\vartheta:= \norm{\omega_k}_{L^\infty}\qquad \text{in \eqref{WEF3C}}
\]
which can be as small as we need when $k$ is sufficiently large.
By item $3)$ of Step 1, we know for $a_k+l< i< b_k-l$, there holds 
\begin{align*}
|\sqrt{E_i(H_k)}-\sqrt{E^*_i(H_k)}|\le \sqrt{E^{\dagger}_i(H_k)}\le 2e^{-q'(l-2)L/2}(1+\frac{1}{\epsilon_1^2}) \sqrt{E^*_i(H_k)}.
\end{align*}
For some $l=l(q,L)$ to be determined,  the above inequality implies
\begin{align}\label{E_i equal E_i^*}
\frac{E^*_i(H_k)}{E_i(H_k)}\in (1-Ce^{-q'lL/2}(1+\frac{1}{\epsilon_1^2}),1+Ce^{-q'lL/2}(1+\frac{1}{\epsilon_1^2}))
\end{align}
for $a_k+l< i< b_k-l$ and $k\gg 1$.  Combining \eqref{E_i^* equal}, \eqref{E_i equal E_i^*}  and  assumption $A 1)$,   we know  for $k$ large enough and $a_k+l< i< b_k-l -1$, there holds 
\begin{align}\label{comp.}
\frac{\int_{Q_i}|H_k|^2\text{d}V_{g_k}}{\int_{Q_{i+1}}|H_k|^2\text{d}V_{g_k}}&\le (1+C\|\nabla v_k\|_{L^\infty}) \frac{E_i(H_k)}{E_{i+1}(H_k)}\\
&\le (1+C\|\nabla v_k\|_{L^\infty}+ Ce^{-q'lL/2}(1+\frac{1}{\epsilon_1^2})+ C \frac{\vartheta}{\epsilon_1})< \frac{3}{2}. 
\end{align}
For the last inequality, we choose $l$ (depending on $q$ and $L$) such that
\[
Ce^{-q'lL/2}(1+\frac{1}{\epsilon_1^2})<\frac{1}{10}.
\]
 Similar argument implies $\frac{\int_{Q_i}|H_k|^2\text{d}V_{g_k}}{\int_{Q_{i+1}}|H_k|^2\text{d}V_{g_k}}> \frac{2}{3}$. So, we have verified \eqref{Willmore cmpr. ass.} for $i\in [a_k+l+1,b_k-l-1]$. This is the proof for Claim A.

In what follows, we assume $b_k-a_k> 2l+10$ and consider the set 
$$J_k=\{a_k+l+1<j<b_k-l-1| \mu_j(f_k)\ge \frac{\epsilon_2}{2}\}$$
for $k\gg 1$.  Note that given Claim A, if $j\notin J_k$, Lemma \ref{3-circle.A.mu} is valid for $Q_{j-1}\cup Q_j\cup Q_{j+1}$.

We claim

{\bf Claim B.} either $J=\emptyset$ or there exist $a_k+l+2\le \alpha_k\le \beta_k\le b_k-l-2$ such that 
\begin{align}\label{connectness 2.}
J_k=[\alpha_k,\beta_k]. 
\end{align}

Assume $J$ is not empty, otherwise there is nothing to prove. If $J_k=[a_k+l+2,b_k-l-2]$, there is still nothing to prove because we may take $a_k'=a_k+l+1$ and $b_k'=b_k-l-1$ so that Part 1) and 2) in Step 3 become trivial and Part 3) holds by the definition of $J_k$. 

In other cases, take any $j\in [a_k+l+1,b_k-l-1]$ and $j\notin J_k$. Since $f_k$ satisfies A1) and we have verified \eqref{Willmore cmpr. ass.} (see Claim A) for $i\in [a_k+l+1,b_k-l-1]$ and $k\gg 1$, by  Lemma \ref{3-circle.A.mu}, $\frac{1}{\mu_i}$ satisfied  \eqref{eqn:qphi} on $Q_{j-1}\cup Q_j\cup Q_{j+1}$. Note that $\lambda_i(H_k)\ge \epsilon_1>0$ implies $\int_{Q_i}|A_k|^2dV_{g_k}\ge \int_{Q_i}|H_k|^2dV_{g_k}>0$ and hence $\mu_i(f_k)>0$ for any $i\in [a_k+1,b_k-1]$ and $k\gg 1$. So, by Lemma \ref{basic observation}, we have
$$
\text{either}\quad  \frac{1}{\mu_j}\le e^{-q'L}\frac{1}{\mu_{j-1}}<\infty,\s \text{or}\s \frac{1}{\mu_j}\le e^{-q'L}\frac{1}{\mu_{j+1}}<\infty,
$$
where $q'=q-\frac{\log 2}{L}$. 

When $\frac{1}{\mu_{j}}\le e^{-q'L}\frac{1}{\mu_{j-1}}<+\infty$,  by Lemma \ref{3-circle.A.mu},  $\frac{1}{\mu_i}$ also satisfies  \eqref{eqn:qphi} on $Q_{j-2}\cup Q_{j-1}\cup Q_j$, we have
$$
\frac{1}{\mu_{j-1}}\le e^{-q'L}\frac{1}{\mu_{j-2}},
$$
which implies that
$$
\mu_{j-2}< \mu_{j-1}<\mu_j<\frac{\epsilon_2}{2}.
$$
Then $\frac{1}{\mu_i}$ also satisfies  \eqref{eqn:qphi} on $Q_{j-3}\cup Q_{j-2}\cup Q_j$. We repeat the above argument until we obtain 
\begin{align}\label{incr. to. left.}
\frac{e^{q'Li}}{\mu_i(f_k)}\text{ is decreasing on }[a_k+l+1,j] \text{ and } \mu_i<\frac{\epsilon_2}{2}, \forall i\in [a_k+l+1, j].
\end{align}
 Moreover, by Lemma \ref{3-circle.A.mu} again,  we also have $\{\int_{Q_i}|A_k|^2dV_{g_k}\}$ satisfies  \eqref{eqn:qphi} on $\cup_{a_k-l+1\le i\le j} Q_i$. We are going to show the assertion \eqref{same direct.}. To see this, we start with the part $Q_{j-1}\cup Q_j\cup Q_{j+1}$. As above, Lemma \ref{3-circle.A.mu} and Lemma \ref{basic observation} implies that either
 \[
\int_{Q_j}|A_k|dV_{g_k}\le e^{-q'L}\int_{Q_{j+1}}|A_k|^2dV_{g_k}
 \]
 or
 \[
\int_{Q_j}|A_k|dV_{g_k}\le e^{-q'L}\int_{Q_{j-1}}|A_k|^2dV_{g_k}.
 \]
We would like to rule out the first possibility.  If $\int_{Q_j}|A_k|dV_{g_k}\le e^{-q'L}\int_{Q_{j+1}}|A_k|^2dV_{g_k}$, by \eqref{Willmore cmpr. ass.} which we just verified, we get $\mu_{j+1}\le 2e^{-q'L}\mu_j<\frac{\epsilon_0}{2}$. Applying Lemma \ref{3-circle.A.mu} repeatedly will imply that $\mu_i\le \frac{\epsilon_0}{2}$ for $i\in [j,b_k-l-1]$. This together with \eqref{incr. to. left.} contradicts the assumption  $J_k\neq \emptyset$. Therefore, we have  $\int_{Q_j}|A_k|^2dV_{g_k}\le e^{-q'L}\int_{Q_{j-1}}|A_k|^2dV_{g_k}$ and hence
 \begin{align}\label{same direct.}
 e^{iqL}\int_{Q_i}|A_k|^2dV_{g_k} \text{ is  decreasing on } [a_k+l+1,j]. 
 \end{align}

In the same way, when $\frac{1}{\mu_{j}}<e^{-q'L}\frac{1}{\mu_{j+1}}$, we have 
\begin{align}\label{incr. to. right.}
\frac{e^{iqL}}{\mu_i}\text{ and } e^{-iqL}\int_{Q_i}|A_k|^2dV_{g_k}  \text{ are  increasing on } [j,b_k-l-1] \text{ and } \mu_i
<\frac{\epsilon_2}{2},  \forall i\in [j,b_k-l-1].
\end{align}
As a result, Claim B follows from \eqref{incr. to. left.} and \eqref{incr. to. right.}.  

Now we finish the proof of Step 3. When $J_k\neq \emptyset$, by setting $a_k'=\alpha_k-1$ and $b_k'=\beta_k+1$, we find that the claims in Step 3. follow from \eqref{same direct.} and \eqref{incr. to. right.}.  When $J_k=\emptyset$, we know $\mu_i<\frac{\epsilon_0}{2}$ for any $i\in [a_k+l+1, b_k-l-1]$,  by Lemma \ref{3-circle.A.mu},  there exists $i_0\in [a_k+l+1, b_k-l-1]$ such that $e^{iqL}\int_{Q_i}|A_k|^2dV_{g_k}$ is   decreasing on $ [a_k+l+1,i_0]$ and $e^{-iqL}\int_{Q_i}|A_k|^2dV_{g_k}$ is increasing  on $[i_0+1,b_k-l-1]$. In this case, set $a_k'=i_0$ and $b_k'=i_0+1$ and the proof of Step 3 is done.\\

{\bf Step 4.} We complete the proof of Proposition \ref{weak version} in this step. Let $l=l(q,L)$ be chosen as in Step 3. If $b_k-a_k\le 2l+10$, then by Step 2, we know 
\begin{align*}
&\sum_{i=1}^{m_k}\|A_k\|_{L^2(Q_i)}^{p}\\
&\le C \Theta_k^{\frac{p}{2}}\left(\sum_{i=1}^{a_k}(e^{-pq'iL/2}+e^{-pq'(a_k-i)L/2})+2l+10+\sum_{i=b_k}^{m_k} (e^{-pq'(m_k-i)L/2}+e^{-(i-b_k)pq'L/2})\right)\\
&\le C(p,q,L)\Theta_k^{\frac{p}{2}}. 
\end{align*}
Proposition \ref{weak version} holds trivially by setting $\mathfrak a_k=a_k$ and $\mathfrak b_k=a_k+1$.  

Next, we consider the case $b_k-a_k>2l+10$. By Step 3, we know $e^{iq'L}\int_{Q_i}|A_k|^2$ is decreasing on $[a_k+l+1,a_k']$. Then we have
$$
\int_{Q_i}|A_k|^2d\mu_k\leq e^{-q'(i-(a_k+l+1))L}\Theta_k, \forall i\in [a_k+l+1,a_k'],
$$
which implies that
\begin{eqnarray*}
\sum_{i=1}^{a_k'}\|A_k\|_{L^2(Q_i)}^p&\leq& C \Theta_k^\frac{p}{2}\left(\sum_{i=1}^{a_k}(e^{-pq'iL/2}+e^{-pq'(a_k-i)L/2})+\sum_{i=a_k+1}^{a_k+l}1+\sum_{i=a_k+l+1}^{a_k'}e^{-pq'(i-(a_k+l))L/2}\right)\\
&\leq& C(p,q,L)\Theta_k^\frac{p}{2}.    
\end{eqnarray*}
For similar reason, we also have 
$$
\sum_{i=b_k'}^{m_k}\|A_k\|_{L^2(Q_i)}^p\leq C\Theta_k^\frac{p}{2}.
$$
By taking $\mathfrak a_k=a_k'$ and $\mathfrak b_k=b_k'$ and applying item 3) of Step 1, item 3) of Step 3 and  and \eqref{E_i equal E_i^*},  we complete the proof. 
\end{proof}

In Step 3 of the above proof, we have verified the comparison assumption \eqref{Willmore cmpr. ass.} of the Willmore energy on nearby segments in $\bigcup_{i \in [\mathfrak a_k, \mathfrak b_k]} Q_i$. We summarize it in the form of a corollary for later use.

\begin{cor}\label{comparison}
Let $f_k$ be a conformal and Willmore immersion from $[0,m_kL]\times S^1$
into $\R^n$ satisfying A1) and  $0<\mathfrak a_k<\mathfrak b_k<m_k$ chosen as in Proposition \ref{weak version}. Then, for $k\gg 1$ and $\mathfrak a_k< i< \mathfrak b_k$ there holds  $\mu_i\ge \frac{\epsilon_2}{2}$, $\frac{1}{2}\le \lambda_i(H_k)\le 2$ and 
\begin{align}
\frac{2}{3}\int_{Q_i}|H_k|^2dV_{g}\leq \int_{Q_{i+1}}|H_k|^2dV_{g}<\frac{3}{2}\int_{Q_{i}}|H_k|^2dV_{g}.
\end{align}
\end{cor}
\begin{proof} By item 3) of Step 3 and  \eqref{E_i equal E_i^*} \eqref{comp.} in the proof of Step 3. 
 \end{proof}

\subsection{Proof of Theorem \ref{neck.first.order}}
\begin{proof}[Proof of Theorem \ref{neck.first.order}]
1) First, we consider the convergence of $\hat H_k/\hat\epsilon_k$. As mentioned in the introduction, for a suitable choice of orthonormal basis in $\R^n$, $\hat f_k$ converges to \eqref{eqn:finfinity}. Denote this orthonormal basis by $(e_1,\cdots, e_n)$.

For any fixed $k_0$, since by Proposition \ref{weak version}, $\nu_i(\hat H_k)\rightarrow 0$ on $Q(k_0):=[-k_0L,k_0L]\times S^1$,  we have
$$
\frac{\int_{Q_i}|\mathcal{F}_i(\hat H_k)|^2e^{-2mt}dtd\theta}{\int_{Q_i}|\hat H_k|^2e^{-2mt}dtd\theta}
\rightarrow 1.
$$
Then, by Lemma \ref{gamma}, we have
$$
\int_{Q(k_0)}|\hat{H}_k|^2e^{-2mt} dtd\theta\leq
C\int_{Q_i}|\hat H_k|^2e^{-2mt} dtd\theta.
$$
Then we get
$$
\int_{Q(k_0)}|\frac{\hat H_k}{\hat\epsilon_k}|^2 dtd\theta<C,
$$
which implies that $\frac{\hat H_k}{\hat\epsilon_k}$
converges to a harmonic vector function $h$ with
$$
h\cdot e_1=h\cdot e_2=0,\s h=\mathcal{F}_1(h).
$$
Then we may write
$$
h=v_1e^{mt}\cos m\theta+v_2 e^{mt}\sin m\theta,
$$
where $v_1$ and $v_2$ are constant vectors perpendicular to $e_1$ and $e_2$.
By \eqref{def.hat.epsilon}, we have
\[
\int_{Q_1} h^2 e^{-2mt} dtd\theta = \pi L,
\]
from which we derive that $\abs{v_1}^2+ \abs{v_2}^2 =1$.

2)
Next, we consider the convergence of $\hat \n_k=e^{-2\hat u_k}\partial_t\hat f_k\wedge \partial_\theta \hat f_k\to e^{2mt}f_t\wedge f_\theta=-e_1\wedge e_2$.
Due to the equation of the Gauss map $\hat{\n}_k$ (see \eqref{tension.G}), we have
\begin{align}\label{equation for the Gauss map}
\Delta \frac{\hat{\n}_k-\bar{\n}_k}{\hat \epsilon_k}
=A_{G(2,n)}(d\hat \n_k,d\hat \n_k/\hat\epsilon_k)+\frac{\nabla_t^\perp \hat H_{k}}{\hat\epsilon_k}\wedge \hat f_{k,\theta}-\frac{\nabla_\theta^\perp \hat H_{k}}{\hat\epsilon_k}\wedge \hat f_{k,t}.
\end{align}
Noting that $\int_{Q_i}|\hat{A}_k|^2_{\hat{g}_k}e^{2\hat{u_k}}dtd\theta\le\frac{2}{\epsilon_2}  \int_{Q_i}|\hat{H}_k|^2e^{2\hat{u}_k}dtd\theta$ ,by the $\epsilon$-regularity of Willmore equation we know that $\|\hat{A}_k\|_{L^\infty(Q_i)}\le C \|\hat{A}_k\|_{L^2(Q_{i-1}\cup Q_i\cup Q_{i+1})}$, the first term in the right hand side converges to zero in $C^\infty_{loc}(\mathbb{R}\times S^1)$. Similarly, the limit of the covariant derivatives in the second and the third term is the same as the limit of the partial derivatives.  

Recall that $h$ and $f_\infty$ are the limit of $\hat H_k/\hat \epsilon_k$ and $\hat f_k$ respectively. By \eqref{eqn:finfinity} and Part 1) above, we have
\begin{align*}
h_t=me^{mt}(v_1\cos m\theta+v_2\sin m\theta), \s  &h_\theta=me^{mt}(-v_1\sin m\theta+v_2\cos m\theta) \\
f_{\infty,t} =-e^{-mt}(\cos m\theta e_1+\sin m\theta e_2), \s  &f_{\infty,\theta}=e^{-mt}(-\sin m\theta e_1+\cos m\theta e_2).
\end{align*}
Then $\frac{\nabla_t^\perp \hat H_{k}}{\hat\epsilon_k}\wedge \hat f_{k,\theta}-\frac{\nabla_\theta^\perp \hat H_{k}}{\hat\epsilon_k}\wedge \hat f_{k,t}$ converges smoothly locally to 
\begin{align*}
h_t\wedge f_{\infty,\theta} -h_\theta\wedge f_{\infty,t}=m\left[\cos 2m\theta(v_1\wedge e_2+v_2\wedge e_1)-\sin 2m\theta (v_1\wedge e_1-v_2\wedge e_2)\right].
\end{align*}

Since $\int_{Q_i}|\nabla \hat \n_k|^2dtd\theta\rightarrow 0$,  $|\nabla\hat  \n_k|$ converges to 0 in $C^\infty_{loc}(\R \times S^1)$.
By the Poincar\'e inequality,
$\|\hat \n_k-\bar{\n}_k\|_{W^{1,2}([-k_0L,k_0L]\times S^1)}\leq C(k_0)\hat\epsilon_k$. Then we may assume $\frac{\n_k-\bar{\n}_k}{\hat\epsilon_k}$
converges to a vector function $v$ weakly in $C^\infty_{loc}(\R\times S^1)$, satisfying (as the limit of \eqref{equation for the Gauss map})
$$
\Delta v=m[\cos 2m\theta(v_1\wedge e_2+v_2\wedge e_1)-\sin 2m\theta (v_1\wedge e_1-v_2\wedge e_2)]
$$
and 
$$\int_{Q_i}|\nabla v|^2dtd\theta\leq C\int_{Q_i}h^2e^{-2mt}dtd\theta<C.
$$
It is easy to check that $v'$, defined by
$$
v'=-\frac{1}{4m}\big(\cos 2m\theta(v_1\wedge e_2+v_2\wedge e_1)-\sin 2m\theta (v_1\wedge e_1-v_2\wedge e_2)\big)
$$
satisfies the same equation.
If $w=v-v'$, then $w$ is a harmonic function satisfying
\begin{equation}
    \label{eqn:goodw}
 \sup_{i}\int_{Q_i}|\nabla w|^2 dtd\theta\le C<\infty.
\end{equation}
By the well-known expansion of harmonic functions and \eqref{eqn:goodw} , there exist constants $a, a^{i\alpha}, a^{\alpha\beta}$ and $b$ such that
$$
w=\big(ae_1\wedge e_2+a^{i\alpha}e_i\wedge e_\alpha +a^{\alpha\beta}e_\alpha\wedge e_\beta\big) t+b.
$$
Here $i=1,2$, $\alpha,\beta=3,4,\cdots,n$.

Obviously, $\frac{\partial \hat{\n}_k}{\partial t}$ is in the tangent space of $G(2,n)$ at $e^{-2\hat u_k}\hat f_{k,t}\wedge \hat f_{k,\theta}$. Taking the limit, we know that $v_t$ is in the tangent space of $G(2,n)$ at $-e_1\wedge e_2$, which implies that
\[
a=0; \quad a^{\alpha\beta}=0.
\]
By \eqref{der.G}, using the interior product formula 
\begin{equation}
    \label{eqn:interior}
V\wedge U\lrcorner W  = (V\cdot W) U - (U\cdot W)V,
\end{equation}
we have
\[
\hat \n_{k,t}\lrcorner \hat f_{k,t} = - \hat \n_{k,\theta}\lrcorner \hat f_{k,\theta}.
\]
Taking the limit again,
\begin{equation}
    \label{eqn:further}
v_t\lrcorner f_{\infty,t} = - v_{\theta}\lrcorner f_{\infty,\theta}.
\end{equation}
Since $v=v'+w$, using the explicit formula for $w$ and $v'$, we have
\begin{eqnarray*}
	v_t\lrcorner f_{\infty,t} &=&( a^{i\alpha}e_i\wedge e_\alpha) \lrcorner (-e^{-mt}(\cos m\theta e_1+\sin m\theta e_2))\\
 &=&  - e^{-mt} (a^{1\alpha}\cos m\theta + a^{2\alpha}\sin m\theta)e^\alpha 
 \end{eqnarray*}
 and 
 \begin{eqnarray*}
v_\theta\lrcorner f_{\infty,\theta}&=& \frac{1}{2}\big(\sin 2m\theta(v_1\wedge e_2+v_2\wedge e_1)+\cos 2m\theta(v_1\wedge e_1-v_2\wedge e_2)\big) \\
&&\lrcorner\big(e^{-mt}(-\sin m\theta e_1 +\cos m\theta e_2)\big)\\
&=&\frac{1}{2}e^{-mt}(-\sin m\theta v_1 + \cos m\theta v_2).
\end{eqnarray*}
Then we derive from \eqref{eqn:further} that
$$
a^{1\alpha}=\frac{1}{2}\langle v_2,e_\alpha\rangle, \s a^{2\alpha}=-\frac{1}{2}\langle v_1,e_\alpha\rangle.
$$
Finally, by the definition of $\bar \n_k$, we have $\int_0^{2\pi}v(0,\theta)d\theta=0$, which implies that $b=0$. In summary, 
\begin{eqnarray*}
v&=&\frac{1}{2}(e_1\wedge v_2-e_2\wedge v_1)t- \frac{1}{4m}\big(\cos 2m\theta(v_1\wedge e_2+v_2\wedge e_1)\\
&&-\sin 2m\theta (v_1\wedge e_1-v_2\wedge e_2)\big).   
\end{eqnarray*}

3)
Taking the interior product (see \eqref{eqn:interior}) of both sides of \eqref{der.G} with $\partial_j f$, we obtain
\[
\hat A_{k,i2} =  \frac{\partial \hat \n_k}{\partial x^i}\lrcorner \partial_1 f_k ; \quad \hat A_{k,i1} = - \frac{\partial \hat \n_k}{\partial x^i}\lrcorner \partial_2 f_k .
\]

Denote by $A_{tt},  A_{t\theta}, A_{\theta\theta}$ the limit of $\frac{\hat A_{k,11}}{\epsilon_k}, \frac{\hat A_{k,12}}{\epsilon_k}, \frac{\hat A_{k,22}}{\epsilon_k}$ respectively. By our definition of $v$ and $f_\infty$,
\begin{align*}
A_{tt}=-  v_t\lrcorner f_{\infty,\theta} ,  \s
A_{t\theta}= A_{\theta t}= v_t\lrcorner f_{\infty,t},  \s A_{\theta \theta}=v_\theta\lrcorner f_{\infty,t}. 
\end{align*}
By \eqref{eqn:limitv}, we have
\begin{align*}
v_t&=\frac{1}{2}(e_1\wedge v_2-e_2\wedge v_1)\\
v_\theta&=\frac{1}{2}(\sin 2m\theta (v_1\wedge e_2+v_2\wedge e_1)+\cos 2m\theta(v_1\wedge e_1-v_2\wedge e_2)),
\end{align*}
which implies that
\begin{align*}
A_{tt}&=- \frac{1}{2}(e_1\wedge v_2-e_2\wedge v_1)\lrcorner e^{-mt}(-\sin m\theta e_1+\cos m\theta e_2)\\
&=\frac{1}{2}e^{-mt}(v_1\cos m\theta+v_2\sin m\theta),
\end{align*}
\begin{align*}
A_{\theta\theta}&= \frac{1}{2}(\sin 2m\theta(v_1\wedge e_2+v_2\wedge e_1)+\cos 2m\theta(v_1\wedge e_1-v_2\wedge e_2))\lrcorner (-e^{-mt}(\cos m\theta e_1+\sin m\theta e_2))\\
&=\frac{e^{-mt}}{2}\big(v_1(\sin 2m\theta\sin m\theta+\cos 2m\theta\cos m\theta)+v_2(\sin 2m\theta\cos m\theta-\cos 2m\theta\sin m\theta)\big)\\
&=\frac{1}{2}e^{-mt}(v_1\cos m\theta+v_2\sin m\theta),
\end{align*}
and 
\begin{align*}
A_{t\theta}&=A_{\theta t}=- \frac{1}{2}(e_1\wedge v_2-e_2\wedge v_1)\lrcorner e^{-mt}(\cos m\theta e_1+\sin m\theta e_2)\\
&=\frac{1}{2}e^{-mt}(v_1\sin m\theta-v_2\cos m\theta).
\end{align*}

4) The asymptotic behavior of Gauss curvature is
\begin{eqnarray*}
\lim_{k\rightarrow+\infty}\frac{\hat K_k}{\epsilon_k^2}&=&
\lim_{k\rightarrow+\infty}e^{-4\hat u_k}\frac{\hat A_{k,tt}\cdot\hat A_{k,\theta\theta}-|\hat A_{k,t\theta}|^2}{\epsilon_k^2}\\
&=&\frac{e^{2mt}}{4}(|v_1\cos m\theta+v_2\sin m\theta|^2-|v_2\cos m\theta-v_1\sin  m\theta|^2)\\
&=& \frac{e^{2mt}}{4}\big( (|v_1|^2-|v_2|^2)\cos 2m\theta+2\langle v_1, v_2\rangle \sin 2m\theta\big).
\end{eqnarray*}

The normal curvature 
\begin{align*}
\lim_{k\to\infty}\frac{\hat K^{\bot}_k}{\epsilon_k^2}&=\lim_{k\to\infty }e^{-2\hat u_k} \frac{\hat A_{k,1i}\wedge \hat A_{k,2i}}{\epsilon_k^2}\\
&=e^{2mt}(A_{tt}\wedge A_{\theta t}+A_{t\theta}\wedge A_{\theta\theta})=0.
\end{align*}

5)
Lastly, we study the limit of $\tau_2(f_k,S)/\hat\epsilon_k$. Recall that given $i_k$ as in the assumptions, we have $c_k^{i_k}$, $f_k^{i_k*}$ defined in the introduction (see \eqref{eqn:fki}).

On one hand, $\hat f_k$ is obtained from $f_k$ by a scaling and translation. By Lemma \ref{transformation},
\[
\tau_2(f_k,S)= \tau_2(\hat f_k,S) + \tau_1(\hat f_k, e^{-c_k^{i_k}} S(f_k^{i_k*})).
\]
By Lemma \ref{neck.first.order} and the definition of $\tau_1$,
\[
\frac{1}{\epsilon_k}\norm{\tau_1(\hat f_k, e^{-c_k^{i_k}} S(f_k^{i_k*}))} \leq C(S) \abs{e^{-c_k^{i_k}} f_k^{i_k*}}.
\]
The right hand side above converges to zero by Corollary \ref{cor:lambdak}. Hence,
\[
\lim_{k\to \infty} \frac{1}{\epsilon_k} {\tau_2(f_k,S)}=\lim_{k\to \infty} \frac{1}{\epsilon_k} {\tau_2(\hat f_k,S)}.
\]
To compute the limit in the right hand side, we use Lemma \ref{neck.first.order} again
\begin{eqnarray*}
 \lim_{k\to \infty}\frac{1}{\hat\epsilon_k}\tau_2(f_k,S)
 &=&2\int_{\set{t}\times S^1} (h\cdot \partial_t (S f_\infty) - \partial_t h\cdot S f_\infty) d\theta\\
 &=&
-4\int_0^{2\pi}(v_1\cos m\theta+v_2\sin m\theta)\cdot(a^{i1}\cos m\theta e_i+ a^{2i}\sin m\theta)e_i d\theta\\
&=&-4\pi( v_1\cdot\sum_{i=3}^na^{i1}e_i+v_2\cdot\sum_{i=3}^na^{i2}e_i)\\
&=&-4\pi\mathbb{L}_{v_1.v_2}(S), 
\end{eqnarray*}
where $(a^{ij})_{i,j=1,\cdots,n}$ is the skew-symmetric matrix that represents $S$. Therefore,
$$
\frac{\|\tau_2(f_k,\cdot)\|}{\hat\epsilon_k}\rightarrow 2\sqrt{2}\pi.
$$
\end{proof}

Now, we can complete the proof of Theorem \ref{main1}.
\begin{proof}[Proof of Theorem \ref{main1}] 
We only need to show that 
$$
\lim_{m\rightarrow+\infty}\lim_{k\rightarrow+\infty}\max_{\mathfrak a_k+m<i<\mathfrak b_k-m}\mu_i=1.
$$
Assume this is not true, then we can find $i_k$, such that $i_k-\mathfrak a_k$, $\mathfrak b_k-i_k\rightarrow+\infty$, such that $\mu_{i_k}(f_k)\rightarrow a\neq 1$. 

However, by Theorem \ref{neck.first.order}, we compute
\begin{align*}
\lim_{k\to +\infty}\frac{|\hat A_k|^2_{\hat{g}_k}}{\epsilon_k^2}&=\lim_{k\to+\infty}e^{-4u_k}\frac{|\hat A_{k,tt}|^2+|\hat A_{k,\theta\theta}|^2+2|\hat A_{k,t\theta}|^2}{\epsilon_k^2}\\
&=\frac{e^{2mt}}{2}(|v_1\cos m\theta+v_2\sin m\theta|^2+|v_2\cos m\theta-v_1\sin  m\theta|^2)\\
&=\frac{e^{2mt}}{2}.
\end{align*}
By the definition of $\mu_i$,
\begin{align*}
\lim_{k\to +\infty}\mu_i(\hat{f}_k)=\lim_{k\to +\infty}\frac{\int_{Q_i}|\hat H_k|^2dV_{\hat{g}_k}}{\int_{Q_i}|\hat A_k|^2_{\hat{g}_k}dV_{\hat g_k}}=\frac{2}{|Q_i|}\int_{Q_i}|v_1\cos m\theta+v_2\sin m\theta|^2dtd\theta=1.
\end{align*}
This is a contradiction.
\end{proof}

\section{Higher order expansions of the neck}
\label{sec:higher}
In this section, we apply the three circles theorem to equations of $f_k$ and $H_k$ respecttively to prove Theorem \ref{main2} and Theorem \ref{2nd.asymp.}.

\subsection{Proof of Theorem \ref{main2}}

\begin{proof}
First, recall the equation $\Delta f_k=e^{2u_k}H_k$. We are about to apply Corollary \ref{decay.linear2} to it. In order to check the assumption there, we note that
\[
\int_{Q_i}|e^{2u_k}H_k|^2e^{2mt} dtd\theta = \int_{Q_i} \abs{H_k}^2 e^{2u_k} e^{2v_k} dtd\theta.
\]
By A1), for sufficiently large $k$, the oscillation of $v_k$ on $Q_{i-1}\cup Q_i\cup Q_{i+1}$ is as small as we need.
By Corollary \ref{comparison}, we derive that for $\mathfrak a_k\le i<\mathfrak b_k$ and for $k$ large,
\[
\frac{1}{2}\int_{Q_{i-1}}|e^{2u_k}H_k|^2e^{2mt} dtd\theta\leq \int_{Q_i}|e^{2u_k}H_k|^2e^{2mt} dtd\theta\leq 2 \int_{Q_{i+1}}|e^{2u_k}H_k|^2e^{2mt} dtd\theta.
\]
Then Corollary \ref{decay.linear2} gives $a_k',b_k'$ with $\mathfrak a_k<a_k'<b_k'<\mathfrak b_k$ such that (for $P_i$ in \eqref{eqn:pci})
\begin{align}\label{middle}
\int_{Q_i}e^{2mt}|f_k-\G_i(f_k)|^2 dtd\theta<Ce^{2v_k(P_i)}W(f_k,Q_i),\s a_k'<i<b_k',
\end{align}
and for $i$ in between $a_k$ and $a_k'$ (or $b_k$ and $b_k'$),
\[
e^{2v_k(P_i)}W(f_k,Q_i)\leq C \int_{Q_i}e^{2mt}|f_k-\G_i(f_k)|^2 dtd\theta 
\]
and the right hand side above decays (or grows) exponentially.
By A1) again, we find $W(f_k,Q_i)$ and $\int_{Q_i}e^{-2u_k}|f_k-\G_i(f_k)|^2$ also decay (or grow) exponentially at a slightly smaller rate. Combining with the fact $\mu_i>\epsilon_2/2$ on $[\mathfrak a_k,\mathfrak b_k]$, we obtain
\begin{align}\label{end. f}
\left(\sum_{i=\mathfrak{a}_k}^{a_k'}+\sum_{i=b_k'}^{\mathfrak{b}_k}\right) \|A_k\|_{L^2(Q_i)}^p\leq C\Phi_k^\frac{p}{2},
\end{align}
where 
$$
\Phi_k:=\max\{\int_{Q_{\mathfrak{a}_{k}}}e^{-2u_k}|f_k-\G_{\mathfrak a_k}(f_k)|^2 dtd\theta,\int_{Q_{\mathfrak{b}_{k}}}e^{-2u_k}|f_k-\G_{\mathfrak b_{k}}(f_k)|^2 dtd\theta\}.
$$
By Corollary \ref{cor:lambdak}, we know  $\Phi_k\to 0$. 
Moreover, \eqref{middle} implies that 
\begin{align}\label{middle f}
   \int_{Q_i}e^{-2u_k}|f_k-\G_i(f_k)|^2 dtd\theta<CW(f_k,Q_i),\s a_k'<i<b_k'. 
\end{align}

Next we consider the equation of $H_k$ for which we are about to apply Corollary \ref{decay.linear3}. For simplicity, we denote the right hand side by $F_k$ so that
$$
\Delta H_k=F_k
$$
By Part 1) in Lemma \ref{mean.value} and A1), we know $u_k+|\nabla u_k|\le C$ and then Lemma \ref{equ. struc.} implies  $$|F_k|=|\Delta H_k|\le C(\int_{Q_{i-1}\cup Q_i\cup Q_{i+1}}|A_k|^2dV_{g_k})^{\frac{1}{2}}(|H_k|+|\nabla H_k|).$$ By 
the $\epsilon$-regularity of Willmore surfaces, we know 
$$e^{u_k}(|H_k|+|\nabla H_k|)\le C(\int_{Q_{i-1}\cup Q_i\cup Q_{i+1}}|A_k|^2dV_{g_k})^{\frac{1}{2}},$$  
hence
$$
\int_{Q_{i}}e^{2u_k}|\Delta H_k|^2 dtd\theta\leq C(\int_{Q_{i-1}\cup Q_i\cup Q_{i+1}}|A_k|^2dV_{g_k})^2. 
$$
For $i\in [\mathfrak a_k,\mathfrak b_k]$, $\mu_i\geq \epsilon_2/2$ and Corollary \ref{comparison} imply that
\begin{align*}
\int_{Q_{i-1}\cup Q_i\cup Q_{i+1}}|A_k|^2dV_{g_k}\le \frac{2}{\epsilon_2}\int_{Q_{i-1}\cup Q_i\cup Q_{i+1}}|H_k|^2dV_{g_k}\le \frac{6}{\epsilon_2}\int_{Q_i}|H_k|^2dV_{g_k}.
\end{align*}
In summary, we have
\begin{align}\label{Delta H_k bd.}
\int_{Q_i}e^{-2mt}|F_k|^2 dtd\theta\leq e^{-2v_k(P_i)}\int_{Q_{i}}e^{2u_k}|\Delta H_k|^2 dtd\theta\le Ce^{-2v_k(P_i)}(\int_{Q_i}|H_k|^2dV_{g_k})^2.
\end{align}
By Corollary \ref{comparison} and A1) again, we may assume
$$
\frac{1}{3}e^{-2v_k(P_{i-1})}(\int_{Q_{i-1}}|H_k|^2dV_{g_k})^2 \leq e^{-2v_k(P_i)} (\int_{Q_i}|H_k|^2dV_{g_k})^2\leq 3e^{-2v_k(P_{i+1})}(\int_{Q_{i+1}}|H_k|^2dV_{g_k})^2.
$$
Then, by Corollary \ref{decay.linear3}, there exists $a_k''>a_k'$, $b_k''<a_k'$, such that
\begin{align}\label{middle H}
\int_{Q_{i}}e^{2u_k}|H_k-\F_i(H_k)|^2dV_{g_k}\leq C(\int_{Q_i}|H_k|^2d\mu_k)^2,\s \forall i\in[a_k'',b_k'']. 
\end{align}
and 
\begin{align*}
(\sum_{i=a_k'}^{a_k''}+\sum_{i=b_k''}^{b_k'})\|A_k\|_{L^2(Q_i)}^p&\le C(\sum_{i=a_k'}^{a_k''}+\sum_{i=b_k''}^{b_k'})(\int_{Q_i}|H_k|^2dV_{g_k})^\frac{p}{2}\\
&\le C \left(\int_{Q_{a_k'}}e^{2u_k}|H_k-\F_{a_k'}(H_k)|^2 dtd\theta+\int_{Q_{b_k'}}e^{2u_k}|H_k-\F_{b_k'}(H_k)|^2 dtd\theta\right)^{\frac{p}{4}}.
\end{align*}
By Corollary \ref{comparison},
we know $\lambda_i(H_k)\le 2$  for $i=a_k', b_k'$ and hence  Lemma \ref{mean.value} implies
\begin{align*}
\|e^{2u_k}(H_k-\F_{a_k'}(H_k))\|_{L^2(Q_{a_k'})}&\le (1+\lambda_{a_k'})\|e^{2u_k}H_k\|_{L^2(Q_{a_k'})}\\
\le C \|e^{u_k}H_k\|_{L^2(Q_{a_k'})}\le C\Theta_k.
\end{align*}
Same argument implies  $\|e^{2u_k}(H_k-\F_{a_k'}(H_k))\|_{L^2(Q_{b_k'})}\le C\Theta_k$, hence we get 
\begin{align}\label{end. H}
(\sum_{i=a_k'}^{a_k''}+\sum_{i=b_k''}^{b_k'})\|A_k\|_{L^2(Q_i)}^p\le C \Theta_k^{\frac{p}{4}}.
\end{align}
Combining \eqref{middle f}\eqref{middle H}\eqref{end. f} and \eqref{end. H} together, we complete the proof of Theorem \ref{main2} by choosing $\mathfrak a_k'=a_k''$ and $\mathfrak b_k'=b_k''$. 
\end{proof}

\subsection{Proof of Theorem \ref{2nd.asymp.}}
As a consequence of Theorem \ref{main2}, using the inequalities in Part 2), we are able to study the asymptotic limit of
\[
\frac{f_k- \G_i(f_k)}{\epsilon_k} \quad \text{and} \quad \frac{H_k- \F_i(H_k)}{\epsilon_k^2}.
\]

\begin{proof}
1) Let $i_k$ and $\hat f_k$ be given in the assumptions. 
We won't be able to show the convergence of $\frac{\hat f_k - \G_1(\hat f_k)}{\hat \epsilon_k}$ directly. Instead, we consider a translation defined by
$$f_k'(t,\theta)=e^{-u_k(P_{i_k})}f_k(t+(i_k-1)L,\theta)=\hat f_k(t,\theta) +e^{-u_k(P_{i_k})}f_k^*((i_k-1)L).$$
We will prove below that $\frac{f'_k -\G_1 (f'_k)}{\hat \epsilon_k}$ converges in $C^\infty_{loc}(\R \times S^1)$ to the $\phi$ in the statement of the theorem. Then the proof of the first part is done by noticing that $\frac{f'_k -\G_1(f'_k)}{\hat \epsilon_k}$ is a translation of $\frac{\hat f_k - \G_1(\hat f_k)}{\hat \epsilon_k}$.
Moreover, by Corollary \ref{cor:lambdak}, $e^{-u_k(P_{i_k})}f_k^*((i_k-1)L)\to 0$, which implies that $f'_k$ and $\hat f_k$ have the same limit, $f_\infty$.

By Theorem \ref{main2}, we have
$$
\int_{Q_i}e^{-2u_k}|f_k-\G_i(f_k)|^2dtd\theta\leq CW(f_k,Q_i), \quad  \forall i\in (\mathfrak a_k',\mathfrak b_k').
$$
By the scaling property that
\begin{align*}
(e^{-u_k}f_k)(t+ (i_k-1)L,\theta)&= (e^{-\hat u_k} f'_k)(t,\theta) \\
 e^{-u_k}(\G_{i_k+j} f_k)(t+ (i_k-1)L,\theta)&= e^{-\hat u_k} (\G_j(f'_k)) (t,\theta),
\end{align*}
we obtain
\begin{equation*}
\int_{Q_j}e^{-2\hat u_k}|f_k'-\G_j(f_k')|^2dtd\theta\leq CW(\hat f_k,Q_j),\quad \forall j\in (\mathfrak a_k'+i_k,\mathfrak b_k'-i_k).
\end{equation*}

Now, for any fixed $j_0>0$, by Theorem \ref{neck.first.order}, we know 
\begin{equation}
    \label{eqn:goodhatu}
\lim_{k\to\infty}\max_{|j|\le j_0}\frac{W(\hat f_k,Q_j)}{\hat \epsilon_k^2}=\pi L \quad \text{and} \quad \lim_{k\to +\infty}\max_{|j|\le j_0}\|\hat u_k+mt\|_{C^0(Q_j)}=0.
\end{equation}
So, for $k$ larger than some $k_0$ depending on $j_0$, there holds 
\begin{align}\label{f_k' middle}
\int_{Q_j}e^{2mt}|\frac{f_k'-\G_j(f_k')}{\hat \epsilon_k}|^2dtd\theta\leq C, \forall j\in [-j_0, j_0],
\end{align}
for some constant $C$ independent of $j_0$. 
Set
$$
\phi_k=\frac{f_k'-\G_1(f_k'))}{\hat \epsilon_k}.
$$
Then, 
$$\Delta \phi_k=\frac{\Delta f_k'}{\hat \epsilon_k}=\frac{\hat H_k}{\hat \epsilon_k}e^{2\hat u_k}$$
whose right hand side converges to (due to Theorem \ref{neck.first.order} and the second equation in \eqref{eqn:goodhatu})
$$e^{-2mt}h=e^{-mt}(v_1\cos{m\theta}+v_2\sin{m\theta}).$$

Next, we claim that
\[
\|e^{mt}\phi_k\|_{L^2(Q_j)}\le C, \quad \forall \abs{j}\leq j_0,\s  k\geq k_0.
\]
To see this, we note that
\begin{align*}
\|e^{mt}\phi_k\|_{L^2(Q_j)}\le \left(\int_{Q_j}e^{2mt}|\frac{f_k'-\G_j(f_k')}{\hat \epsilon_k}|^2dtd\theta\right)^{\frac{1}{2}}+\left(\int_{Q_j}e^{2mt}|\frac{\G_1(f_k')-\G_j(f_k')}{\hat \epsilon_k}|^2dtd\theta\right)^{\frac{1}{2}}.
\end{align*}
By \eqref{Gi} and the quantitative Pohozaev identity \eqref{ODE.g}, we know 
\begin{align*}
|\G_1(f_k')-\G_j(f_k')|&\le \frac{e^{-mt}}{2m}\abs{g^-(\frac{L}{2})e^{\frac{mL}{2}}-g^{-}(\frac{(2j-1)L}{2})e^{\frac{m(2j-1)L}{2}}}\\
&\le \frac{e^{-mt}}{2m}\int_{\frac{L}{2}}^{\frac{(2j-1)L}{2}}e^{m\tau}|\alpha(\tau)|d\tau,
\end{align*}
where 
\begin{align*}|\alpha(\tau)|=\frac{1}{\pi}\abs{\left(\int_{0}^{2\pi}\hat H_k e^{2\hat u_k}(\tau,\theta)\cos m\theta d\theta, \int_{0}^{2\pi}\hat H_k e^{2\hat u_k}(\tau,\theta)\sin m\theta d\theta \right)}
\le \frac{1}{\sqrt{\pi}}\left(\int_{0}^{2\pi}|\hat H_k e^{2\hat u_k}|^2(\tau,\theta)d\theta\right)^{\frac{1}{2}}.
\end{align*}
Using the second part in \eqref{eqn:goodhatu} and integrating the above inequality, we get
\begin{align*}
\int_{Q_j}e^{2mt}|\frac{\G_1(f_k')-\G_j(f_k')}{\hat \epsilon_k}|^2dtd\theta\le C\frac{j_0 L}{\hat \epsilon_k^2}\int_{\frac{L}{2}}^{\frac{(2j-1)L}{2}}\int_0^{2\pi}|\hat H_k|^2e^{2\hat u_k}(\tau,\theta)d\tau d\theta\le C(j_0,L),
\end{align*}
which together with \eqref{f_k' middle} finishes the proof of the claim.

With the a priori estimate given by the claim, we know $\phi_k$ converges in $C_{loc}^\infty(\mathbb{R}\times S^1)$ to $\phi$ satisfying
$$
\Delta\phi=e^{-mt}(v_1\cos m\theta +v_2\sin m\theta)
$$
and
\[
\max_{|j|\le j_0}\|e^{mt}\phi\|_{L^2(Q_j)}<C(j_0,L)\quad \text{and} \quad \G_1(\phi)=0.
\]

Setting
$$
\phi'=-\frac{t e^{-mt}}{2m}(v_1\cos m\theta +v_2\sin m\theta),
$$
we verify that
\begin{align*}
\Delta \phi'=e^{-mt}(v_1\cos m\theta +v_2\sin m\theta) 
\end{align*}
and
\begin{align*}
\phi'-\G_1(\phi')=(-\frac{t}{2m}+\frac{L}{4m}-\frac{1}{4m^2})e^{-mt}(v_1\cos m\theta +v_2\sin m\theta).
\end{align*}

Due to the fact that $\G_1(\phi)=0$, we can estimate (for any $j\in \mathbb Z$)
\begin{align}\label{eqn:2growth}
\int_{Q_j}e^{2mt}|\phi-\phi'-\G_1(\phi-\phi')|^2 dtd\theta&=\int_{Q_j}e^{2mt}|\phi-\phi'-\G_1(\phi')|^2 dtd\theta\\ \nonumber
&\leq C\left[(jL)^2+\int_{Q_j}e^{2mt}(C_1e^{-mt}+C_2te^{-mt})^2 dtd\theta\right]\\
&\le C((jL)^2+1).
\end{align}
By our choice of $\phi'$, $\phi-\phi'$ and hence $\phi-\phi'-\G_1(\phi-\phi')=\phi-\phi'-\G_j(\phi-\phi')$ is harmonic function defined on $\R\times S^1$.
It satisfies 3-circle Lemma \ref{3-circle.harmonic} on $Q_{j-1}\cup Q_j\cup Q_{j+1}$ for each $j\in \mathbb{Z}$. As a consequence, it must vanish identically, because if otherwise, $\int_{Q_j}e^{2mt}|\phi-\phi'-\G_1(\phi-\phi')|^2 dtd\theta$ would grow exponentially, contradicting \eqref{eqn:2growth}.
By $\G_1(\phi)=0$, we get
\begin{align*}
\phi=\phi'-\G_1(\phi')=(-\frac{t}{2m}+\frac{L}{4m}-\frac{1}{4m^2})e^{-mt}(v_1\cos m\theta +v_2\sin m\theta).
\end{align*}

2)
Next, we consider the convergence of $\frac{\hat H_k-\F_i(\hat H_k)}{\hat \epsilon_k^2}$ by similar argument.  Again by Theorem \ref{main2}, we have
\begin{equation}
\int_{Q_i}e^{2u_k}|H_k-\F_i(H_k)|^2dtd\theta\leq C(W(f,Q_i))^2, \forall i\in (\mathfrak a_k',\mathfrak b_k').
\end{equation}
Both sides of the above inequalty are invariant under scaling and translation, hence
\begin{align*}
\int_{Q_j}e^{2\hat u_k}|\hat H_k-\F_i(\hat H_k)|^2dtd\theta\leq C(W(\hat f_k,Q_j))^2, \forall j\in (\mathfrak a_k'-i_k,\mathfrak b_k'-i_k).
\end{align*}
By Theorem \ref{neck.first.order}, we know for any $j_0>0$ and $k>k_0$ (for some $k_0$ depending on $j_0$), there holds 
\begin{align*}
\int_{Q_j}e^{-2mt}|\frac{\hat H_k-\F_j(\hat H_k)}{\hat \epsilon_k^2}|^2dtd\theta\leq C, \forall j\in [-j_0,j_0],
\end{align*}
where $C$ is independent of $j_0$. 

Set $\psi_k=(\hat H_k-\F_1(\hat H_k))/{\hat\epsilon_k^2}$, the equation of the mean curvature implies that
\begin{equation}
    \label{eqn:psik}
\Delta \psi_k = -2{\rm div}(\frac{\hat H_k}{\hat \epsilon_k}\cdot \frac{ \hat A_{k,pq}}{\hat\epsilon_k}\hat g_k^{ip}\partial_i \hat f_k)+\frac{1}{2}{\rm div}(|\frac{\hat H_k}{\hat\epsilon_k}|^2\nabla \hat f_k).
\end{equation}
In order to take the limit in the above equation, we will first claim
\begin{equation}
    \label{eqn:psikl2}
\norm{e^{-mt} \psi_k}_{L^2(Q_j)}\leq C (j_0+1), \quad \forall \abs{j}\leq j_0, \, k\geq k_0.
\end{equation}
To see this, we note that
\[
\|e^{-mt}\psi_k\|_{L^2(Q_j)}\le \|e^{-mt}\frac{\hat H_k-\F_j(\hat H_k)}{\hat \epsilon_k^2}\|_{L^2(Q_j)}+\|e^{-mt}\frac{\F_1(\hat H_k)-\F_j(\hat H_k)}{\hat \epsilon_k^2}\|_{L^2(Q_j)}.
\]
The first term in the right hand side above is bounded by \eqref{eqn:goodhatu} and Theorem \ref{neck.first.order}. For the second term, we use \eqref{ODE.g} to see
\begin{align*}
|\F_1(\hat H_k)-\F_j(\hat H_k)|&\le \frac{e^{mt}}{2m}\sqrt{\cos^2 m\theta+\sin^2m\theta}|g^+(\frac{L}{2})e^{\frac{mL}{2}}-g^{+}(\frac{(2j-1)L}{2})e^{-\frac{m(2j-1)L}{2}}|\\
&\le \frac{e^{mt}}{2m}\int_{\frac{L}{2}}^{\frac{(2j-1)L}{2}}e^{-m\tau}|\alpha(\tau)|d\tau,
\end{align*}
where 
\begin{align*}|\alpha(\tau)|=\frac{1}{\pi}\big|\big(\int_{0}^{2\pi}\Delta \hat H_k(\tau,\theta)\cos m\theta d\theta, \int_{0}^{2\pi}\Delta \hat H_k (\tau,\theta)\sin m\theta d\theta \big)\big|
\le \frac{1}{\sqrt{\pi}}(\int_{0}^{2\pi}|\Delta \hat H_k |^2(\tau,\theta)d\theta )^{\frac{1}{2}}.
\end{align*}
By Theorem \ref{neck.first.order}, \eqref{eqn:goodhatu}, the $\epsilon$-regularity and the scaling property of mean curvature, we have
\begin{align*}
\int_{\frac{L}{2}}^{\frac{(2j-1)L}{2}}e^{-m\tau}|\alpha(\tau)|d\tau&\le C\sqrt{j_0L}(\int_{\frac{L}{2}}^{\frac{(2j-1)L}{2}} \int_0^{2\pi}e^{-2m\tau}|\Delta \hat H_k|^2d\tau d\theta)^{\frac{1}{2}} \\
&\le C\sqrt{j_0 L}(\int_{\frac{L}{2}+(i_k-1)L}^{\frac{(2j-1)L}{2}+((i_k-1)L)} \int_0^{2\pi} e^{2u_k}|\Delta  H_k|^2d\tau d\theta)^{\frac{1}{2}}\\
&\le C(L) j_0 \sup_{|i-\hat i_k|\le j_0+1}W(f_k,Q_i)\\
&\leq C(L) j_0 \hat \epsilon_k^2.
\end{align*}
Therefore,
\begin{align*}
\|e^{-mt}\frac{\F_1(\hat H_k)-\F_j(\hat H_k)}{\hat \epsilon_k^2}\|_{L^2(Q_j)}\le \frac{|Q_j|}{2m\hat \epsilon_k^2}\int_{\frac{L}{2}}^{\frac{(2j-1)L}{2}}e^{-m\tau}|\alpha(\tau)|d\tau\le C(L)j_0,
\end{align*}
which concludes the proof of the claim \eqref{eqn:psikl2}.

With the estimates given by Theorem \ref{neck.first.order} and \eqref{eqn:psikl2},  we learn from the equation \eqref{eqn:psik} that $\psi_k$ converges to a function $\psi$ in $C^\infty_{loc}(\R \times S^1)$, satisfying
\begin{equation}
    \label{eqn:deltapsi}
\Delta\psi=-2\partial_q(h\cdot A_{pq} e^{2mt}\partial_p  f_\infty)+\frac{1}{2}{\rm div}(|h|^2\nabla f_\infty).
\end{equation}
Using the orthonormal basis $(e_i)$ given in Theorem \ref{neck.first.order}, we write
\[
\psi = (\psi',\psi'') \in \R^2 \times \R^{n-2}.
\]
To finishing the proof of Theorem \ref{2nd.asymp.}, it suffices to show that $\psi''\equiv 0$. 

For that purpose, we first note that the projection of the right hand side of \eqref{eqn:deltapsi} onto the space spanned by $(e_3,\cdots,e_{n})$ is zero. To see this,  we expand the right hand side of \eqref{eqn:deltapsi}
\[
-2\partial_q(h\cdot A_{pq} e^{2mt})\partial_pf_\infty -2h\cdot A_{pq} e^{2mt}\partial^2_{pq} f_{\infty}+\frac{1}{2}\partial_p(|h|^2)\partial_p f_{\infty}+\frac{1}{2}|h|^2\Delta f_\infty
\]
and notice that $f_\infty$ (see \eqref{eqn:finfinity}) together with its derivatives take value in $\R^2 \times \set{0}$. Therefore, $\psi''$ is a harmonic function defined on $\R\times S^1$.

On the other hand, it follows from \eqref{eqn:psikl2} that
\[
\norm{e^{-mt} \psi''}_{L^2(Q_j)} \leq C(\abs{j}+1), \quad \forall j\in \mathbb Z.
\]
Moreover, by the definition of $\psi_k$, we have $\F_1(\psi_k)=0$, which yields that $\F_1(\psi)=0$ and $\F_1(\psi'')=0$. We may then apply Lemma \ref{3-circle.harmonic} to $\psi''$ on any $Q_{j-1}\cup Q_j\cup Q_{j+1}$ to get an exponential growth that contradicts the above estimate, unless $\Psi''$ is identically zero. This concludes our proof of Theorem \ref{2nd.asymp.}.
\end{proof}

\subsection{A corollary of Theorem \ref{2nd.asymp.}}
In this final subsection, we use Theorem \ref{2nd.asymp.} to derive a corollary. It is the key fact used in the proof of the geodesic part of Theorem \ref{thm:main}.

\begin{lem}\label{normal term}
Under the same assumptions of Theorem \ref{2nd.asymp.} and using the same notations, there exist $C>0$ (independent of the choice of $i_k$) such that

1) 
\[
\sup_{t\in [-L,\tilde 2L]}\abs{\int_0^{2\pi}(\nabla_t^\bot \hat H_k\wedge \partial_\theta \hat f_{k} -\nabla_\theta^\bot \hat H_k\wedge \partial_t \hat f_{k})/\hat \epsilon_k^2 d\theta}\leq C.
\]

2) The limit of 
\[
t\mapsto \int_0^{2\pi}(\nabla_t^\bot \hat H_k\wedge \partial_\theta \hat f_{k} -\nabla_\theta^\bot \hat H_k\wedge \partial_t \hat f_{k})/\hat \epsilon_k^2 d\theta
\]
as a sequence of functions defined $[-L,2L]$ is a (continuous) function whose image lies in the normal space of $G(2,n)$ at $-e_1\wedge e_2$ in $\Lambda^2(\R^n)$.
\end{lem}

\begin{proof}
By the definition of $\nabla^\bot$, 
\begin{eqnarray*}
\nabla_t^\bot \hat H_k\wedge\partial_\theta\hat f_k&=&\partial_t\hat H_k\wedge\partial_\theta\hat f_k- \lan \partial_t \hat H_{k}, \partial_t \hat f_{k}\ran e^{-2\hat u_k}\partial_t \hat f_{k}\wedge \partial_\theta \hat f_{k}\\
&=&\partial_t\hat H_k\wedge\partial_\theta\hat f_k- \lan \hat H_{k},\partial^2_{tt} \hat f_{k}\ran e^{-2\hat u_k}\partial_t\hat f_{k}\wedge \partial_\theta \hat f_{k}\\
&=&\partial_t\hat H_k\wedge\partial_\theta\hat f_k- \lan \hat H_{k},\hat A_{k,tt}\ran e^{-2 \hat u_k}\partial_t \hat f_{k}\wedge \partial_\theta \hat f_{k}. 
\end{eqnarray*}
By Theorem \ref{neck.first.order}, both $\hat H_k/\hat \epsilon_k$ and $\hat A_{k,tt}/\hat \epsilon_k$ have a smooth limit. Hence, 
$$
\lan \hat H_{k},\hat A_{k,tt}\ran e^{-2\hat u_k}\partial_t \hat f_{k}\wedge \partial_\theta \hat f_{k}/\hat \epsilon_k^2
$$
is bounded and converges to a smooth function times
\[
e^{2mt} \partial_ t f_\infty \wedge \partial_\theta f_\infty,
\]
which is in the normal space of $G(2,n)$ at $-e_1\wedge e_2$.
The same analysis work for $\nabla^\bot_\theta \hat H_k \wedge \partial_t \hat f_k$. The proof of the lemma is reduced to the same two claims but with the integrand replaced by  
the limit of 
$$
R_k(t)=\frac{1}{\hat \epsilon_k^2}\int_0^{2\pi}(\partial_t \hat H_k\wedge\partial_\theta\hat f_k-\partial_\theta \hat H_k\wedge\partial_t\hat f_k)d\theta.
$$ 

In order to use Theorem \ref{2nd.asymp.}, we compute
\begin{eqnarray*}
&&\frac{1}{\hat \epsilon_k^2}\int_0^{2\pi}\partial_t \hat H_k\wedge\partial_\theta\hat f_kd\theta\\
&=&
\hat \epsilon_k\int_0^{2\pi}\partial_t\frac{\hat H_k-\F_1(\hat H_k)}{\hat\epsilon_k^2}\wedge\partial_\theta\frac{ \hat f_k-\G_1( \hat f_k)}{\hat\epsilon_k}d\theta+\int_0^{2\pi}\partial_t\frac{\hat H_k-\F_1(\hat H_k)}{\hat\epsilon_k^2}\wedge\partial_\theta \G_1(\hat f_k)d\theta\\
&&+\int_0^{2\pi}\partial_t\frac{\F_1(\hat H_k)}{\hat\epsilon_k}\wedge\partial_\theta\frac{\hat f_k-\G_1( \hat f_k)}{\hat\epsilon_k}d\theta+\int_0^{2\pi}\partial_t\frac{\F_1(\hat H_k)}{\hat\epsilon_k}\wedge\partial_\theta\frac{\G_1(\hat f_k)}{\hat\epsilon_k}d\theta\\
&:=&I_1+I_2+I_3+I_4,
\end{eqnarray*}
and
\begin{eqnarray*}
&&\frac{1}{\hat \epsilon_k^2}\int_0^{2\pi}\partial_\theta \hat H_k\wedge\partial_t \hat f_kd\theta\\
&=&\hat \epsilon_k\int_0^{2\pi}\partial_\theta \frac{\hat H_k-\F_1(\hat H_k)}{\hat\epsilon_k^2}\wedge\partial_t\frac{\hat f_k-\G_1(\hat f_k)}{\hat\epsilon_k}d\theta+\int_0^{2\pi}\partial_\theta \frac{\hat H_k-\F_1(\hat H_k)}{\hat\epsilon_k^2}\wedge\partial_t\G_1(\hat f_k)d\theta\\
&&+\int_0^{2\pi}\partial_\theta\frac{\F_1(\hat H_k)}{\hat\epsilon_k}\wedge\partial_t\frac{\hat f_k-\G_1(\hat f_k)}{\hat\epsilon_k}d\theta+\int_0^{2\pi}\partial_\theta \frac{\F_1(\hat H_k)}{\hat\epsilon_k}\wedge\partial_t\frac{\G_1(\hat f_k)}{\hat\epsilon_k}d\theta\\
&=&I_1'+I_2'+I_3'+I_4'.
\end{eqnarray*}
By the definition of $\F$ and $\G$, (for fixed $t$) there are $a,b,a',b' \in \R^n$ such that
$$
\F_1(\hat H_k)/\hat \epsilon_k=ae^{mt}\cos m\theta+be^{mt}\sin m\theta,\s \G_1(f_k')/\hat\epsilon_k= a'e^{-mt}\cos m\theta+b'e^{mt}\sin m\theta.
$$
We find by direct verification that
\begin{eqnarray*}
I_4-I_4'
&=&m^2\int_0^{2\pi}(a\cos m\theta+b\sin m\theta)\wedge(-a'\sin m\theta+b'\cos m\theta)d\theta\\
&& +m^2\int_0^{2\pi}(-a\sin m\theta+b\cos m\theta)\wedge(a'\cos m\theta+b'\sin m\theta)d\theta\\
&=&m^2\int_0^{2\pi}(a\wedge b'\cos^2m\theta-b\wedge a'\sin^2m\theta)d\theta\\
&&+m^2\int_0^{2\pi}(-a\wedge b'\sin^2m\theta+b\wedge a'\cos^2m\theta)d\theta\\
&=&0.
\end{eqnarray*}
Now, Theorem \ref{2nd.asymp.} implies that the limit (in smooth topology) of both $I_1$ and $I_1'$ are zero. 
Theorem \ref{2nd.asymp.} and Theorem \ref{neck.first.order} together imply $I_2,I_3,I_2',I_3'$ all have smooth limit. In particular,
$$
\sup_{[-L,2L]}\abs{R_k(t)} = \abs{\sum_{i=1}^3(I_i-I_i')}\leq C.
$$
Moreover, by computing the limit explicitly, we observe that the limit of  $I_2-I_2'$ and $I_3-I_3'$ are  parallel to $e_1\wedge e_2$ and $v_1\wedge v_2$ respectively. In summary, the limit of $R_k(t)$ is in the normal space of $G(2,n)$ at $-e_1\wedge e_2$.
\end{proof}

\section{Proof of Theorem \ref{thm:main}}
\label{sec:proof}
Let $f_k$ be as assumed in Theorem \ref{thm:main}. Theorem \ref{main1} and Theorem \ref{main2} defines
\[
0<\mathfrak{a}_k< \mathfrak{a}_k' < \mathfrak{b}_k'<\mathfrak{b}_k<m_k.
\]
Set
$$\bar{i}_{k}=\frac{\mathfrak b_k'+\mathfrak a_k'}{2} \quad \text{and} \quad \tilde{T}_k=\frac{(\mathfrak b_k'-\mathfrak a_k')L}{2},$$
where for simplicity we assume that $\bar{i}_k$ is an integer. 
Let $f_k^{\bar{i}_k}$ be defined in \eqref{eqn:fki} and set
$$\hat{f}_k=f_k^{\bar{i}_k}$$ 
and (similar to \eqref{def.hat.n})
$$\hat{\n}_k=\n(\hat f_k)$$
Let
$$
\alpha_k(t)=\hat \n_k^*(t)=\frac{1}{2\pi}\int_0^{2\pi}\hat \n_k(t,\theta)d\theta
$$
for $t\in [-\bar{i}_kL, (\mathfrak b_k - \bar{i}_k)L]$.
In the second part of Theorem \ref{thm:main}, the claim is about the limit of the image of the Gauss map. By A2), 
\[
\sup_{i=1,\cdots,m_k} \text{osc}_{Q_i} \n_k \to 0,
\]
and therefore, it suffices to study the image of the curve $\alpha_k$.

\subsection{Easy case.}
We first discuss Theorem \ref{thm:main} in an almost trivial case, namely, 
\begin{equation}
\label{eqn:trivial}
\limsup_{k\to \infty} (\mathfrak b_k'-\mathfrak a_k')=\tilde{T}<\infty.
\end{equation}
In such cases, due to Part 1) of Theorem \ref{main2}, we have 
\begin{align}\label{finite. t.}
\lim_{k\to \infty}\sum_{i=1}^{m_k}\|A_k\|^p_{L^2(Q_i)} =0.
\end{align}
This implies that both the Willmore energy on the neck and the total length of the curve converge to zero. The energy part is trivial and the claim on the length follows from the formula
\begin{align*}
|\alpha_k(t_1)-\alpha_k(t_2)|&\le \frac{1}{2\pi}\int_{t_1}^{t_2}\int_0^{2\pi}|\nabla \hat \n_k|(t,\theta)d\theta dt\\
&\le C \sum_{i=1}^{m_k} \norm{A_k}_{L^2(Q_i)},
\end{align*}
for any $t_1,t_2\in [0, m_kL]$.
Hence, the image of $\alpha_k$ converges (in the Hausdorff sense) to a single point. 
The proof of Theorem \ref{thm:main} is done in this case by Part 5) of Theorem \ref{neck.first.order}, which implies that the limit of $\norm{\tau_2(f_k,\cdot)}$ is zero.

\subsection{General case: energy and length}
In what follows, we assume
\begin{equation*}
\lim_{k\to \infty} \tilde{T}_k<\infty.
\end{equation*}
Similar to the discussion above, we have
\begin{align}
    \label{eqn:endnoenergy}
\lim_{k\to \infty} W(f_k,[0,T_k]\times S^1) =  \lim_{l\to\infty}\lim_{k\to \infty}W(\hat{f}_k, [-\tilde{T}_k+l,\tilde{T}_k-l])
\end{align}
and
\begin{align}\label{eqn:endnoneck}
\lim_{k\to \infty}\text{Length}(\alpha_k)=\lim_{l\to\infty}\lim_{k\to \infty}\text{Length}(\alpha_k, [-\tilde{T}_k+l,\tilde{T}_k-l]). 
\end{align} 
The first claim in Theorem \ref{thm:main} follows from \eqref{eqn:endnoenergy} and Part 5) of Theorem \ref{neck.first.order}.
\begin{rem} \label{rem:uniform}
It is worth noting that Theorem \ref{neck.first.order} implies
$$
\lim_{l\rightarrow+\infty}\lim_{k\rightarrow \infty}\max_{\mathfrak a_k+l<i<\mathfrak b_k-l}\left|\frac{\|\tau_2(f_k,\cdot)\|}{\sqrt{W(f_k,Q_{i})}}-2\sqrt{\frac{2\pi}{L}}\right|=0.
$$
In fact, if the above limit is not true, then we may find a subsequence $i_k$ contradicting Part 5) of Theorem \ref{neck.first.order}.
\end{rem}

For the length of $\alpha_k$, we derive from Par 2) and 5) of Theorem \ref{neck.first.order} that for any sequence $t_k$ satisfying $t_k+ \tilde{T}_k \to \infty$ and $\tilde T_k -t_k \to \infty$,
\begin{align}\label{eqn:regular_curve}
\lim_{k\to \infty}\frac{|\dot \alpha _k(t_k)|}{\sqrt{W(f_k,Q_{[t_k]+\bar{i}_k})}}=\frac{1}{2\sqrt{\pi L}}, 
\end{align}
where $[t_k]$ is the unique $j$ satisfying $t_k\in [(j-1)L,jL)$.

Hence,
\[
\lim_{k\to \infty} \text{Length}(\alpha_k) =\frac{1}{4\sqrt{2}\pi}\cdot \lim_{k\to\infty}\|\tau_2(f_k,\cdot)\|m_kL. 
\]

\subsection{General case: being geodesic}
To conclude the proof of Theorem \ref{thm:main}, it suffices to show that $\alpha_k$ in its arc-length parametrization converges in some sense to a geodesic in $G(2,n)$. For that purpose, we use $s$ for the arc-length parameter, denote the parameter transformation by $t=\xi_k(s)$. Assume $\xi_k(0)=0$ and for fixed $l$, set
\[
S_{1k}:= \xi_k^{-1}(-\tilde T_k+l); \quad S_{2k}:= \xi_k^{-1}(\tilde T_k-l).
\]
By taking subsequence, we assume
\[
S_1=\lim_{k\to \infty} S_{1k}; \quad S_2=\lim_{k\to \infty} S_{2k}.
\]
Note that it is possible that the above limit is infinity.

We need the following notion of weak geodesic.
\begin{defi}
    \label{def:weakgeodesic}
    Assume that $M$ is an embedded submaniold of $\R^n$.
    Let $\gamma:(a,b)\to M\subset \R^n$ be a $C^{1,\alpha}$ curve (parametrized by arc-length). It is said to be a weak geodesic if for any tangent vector field $V$ of $M$ such that $V\circ \gamma$ has compact support, we have
    \[
    \int_a^b \frac{d}{ds} \gamma(s)\cdot \frac{d}{ds}V(\gamma(s)) ds =0.
    \]
\end{defi}
Obviously, a smooth geodesic is a weak geodesic. Moreover, it is known that a weak geodesic is smooth and hence a classical one. In our case, we will regard $G(2,n)$ as a submanifold of $\Lambda^2(\R^n)$.

\begin{lem}\label{lem:geodesic}  $\{\alpha_k\circ \xi_k(s)\}_{k=1}^{\infty}$ is  $C^2$ uniformly bounded on any compact subset of $I:=(S_1,S_2)$ and converges to a geodesic $\alpha:I\to G(2,n)$ with 
$$\text{Length}(\alpha)=\frac{1}{4\sqrt{2}\pi}\cdot \lim_{k\to\infty}\|\tau_2(f_k,\cdot)\|m_kL.$$ 
\end{lem}
\begin{proof}
In what follows, we use $\dot \alpha_k$ for the $t$-derivative of $\alpha_k$. 
By direct calculation, we get 
\begin{align*}
\frac{d\alpha_k\circ \xi_k}{ds}=\frac{\dot \alpha_k(\xi_k(s))}{|\dot \alpha_k(\xi_k(s))|}, 
\end{align*}
and 
\begin{equation}
\label{eqn:d2alpha}
\frac{d^2\alpha_k\circ \xi_k}{ds^2}=\frac{\ddot \alpha_k(\xi_k(s))}{|\dot \alpha_k(\xi_k(s))|^2}-\langle\frac{\ddot \alpha_k(\xi_k(s))}{|\dot \alpha_k(\xi_k(s))|^2}, \frac{\dot \alpha_k(\xi_k(s))}{|\dot \alpha_k(\xi_k(s))|}\rangle \frac{\dot \alpha_k(\xi_k(s))}{|\dot \alpha_k(\xi_k(s))|}.
\end{equation}
By the definition of arc-length paramtrization, $\abs{d\alpha_k/ds}=1$. For the $C^2$ bound, it suffices to estimate $\abs{d^2\alpha_k/ds^2}$. By \eqref{eqn:regular_curve}, this amounts to 
\begin{equation}
    \label{eqn:suf1}
\abs{\ddot \alpha_k(\xi_k(s))}\leq C W(f_k,Q_{[\xi_k(s)]}),
\end{equation}
which we prove now.

For any fixed $s_0$, let $t_k=\xi_k(s_0)$ and $\tilde i_k = \bar i_k + [t_k]$ and set (using \eqref{eqn:fki})
\begin{align*}
  \tilde \epsilon_k&=\sqrt{W(f_k,Q_{\tilde i_k})}   \\
  \tilde f_k&= f_k^{\tilde i_k} (= (\hat f_k)^{[t_k]})\\
  \tilde \n_k&= \n(\tilde f_k).
\end{align*}
Since translation and scaling do not change the Gauss map, if we define $\tilde \alpha_k$ by
$$
\tilde \alpha_k(t)=\tilde \n_k^{*}(t)=\frac{1}{2\pi}\int_0^{2\pi}\tilde \n_k(t,\theta)d\theta,\quad \text{for} \quad t\in [-\tilde T_k - [t_k]L, \tilde T_k -[t_k]L],
$$
we have
\[
\tilde \alpha_k (t) =  \alpha_k(t+[t_k]L).
\]
Using \eqref{tension.G}, we have
\begin{eqnarray*}
    \ddot \alpha_k(t_k) &=& \ddot {\tilde \alpha}_k(t_k- [t_k]L) \\
    &=& \left(\frac{1}{2\pi}\int_0^{2\pi} A_{G(2,n)}(d\tilde \n_k, d\tilde \n_k) + e^{-2\tilde u_k} (\nabla^\perp_t \tilde H_k \wedge \partial_\theta \tilde f_k - \nabla^\perp_\theta \tilde H_k \wedge \partial_t \tilde f_k)  d\theta\right) (t_k-[t_k]L).
\end{eqnarray*}
Note that for any $t_k$, we always have
\[
0\leq t_k -[t_k]L\leq L.
\]
By Theorem \ref{neck.first.order} and the Part 1) of Lemma \ref{normal term} (with $\tilde f_k$ as $\hat f_k$ there),  we know that 
\[
\abs{\ddot \alpha_k(t_k)} \leq C \tilde \epsilon_k^2,
\]
from which \eqref{eqn:suf1} follows.
It also follows from Theorem \ref{neck.first.order} that
\[
\lim_{k\to \infty}\frac{\dot \alpha_k(t_k)}{\tilde \epsilon_k}=-\frac{1}{2}(e_1\wedge v_2-e_2\wedge v_1)\in T_{-e_1\wedge e_2}G(2,n).
\]
Together with \eqref{eqn:suf1}, we obtained that $\frac{d^2}{ds^2}\alpha_k\circ \xi_k$ is bounded. Hence, we may assume that by taking subsequence, $\alpha_k\circ \xi_k$ converges in $C^{1,\alpha}$ topology to a curve $\alpha(s)$.

To show that $\alpha(s)$ is a weak geodesic, it suffices to show
\begin{equation*}
   \lim_{k\to \infty} \int_{S_1}^{S_2} \frac{d}{ds} \alpha_k\circ \xi_k (s)\cdot \frac{d}{ds}V(\alpha_k\circ \xi_k(s)) ds =0
\end{equation*}
for any smooth tangent vector field $V$ as in Definition \ref{def:weakgeodesic}. Since $\alpha_k$ is smooth, by integration by parts, the above is equivalent to
\begin{equation}
    \label{eqn:suf2}
   \lim_{k\to \infty} \int_{S_1}^{S_2} \frac{d^2}{ds^2} \alpha_k\circ \xi_k (s)\cdot V(\alpha_k\circ \xi_k(s)) ds =0.
\end{equation}
Due to the Dominated Convergence Theorem, it suffices to show for each $s_0$,
\[
\lim_{k\to \infty} \frac{d^2}{ds^2} \alpha_k\circ \xi_k (s_0) \quad \text{is a normal vector of $G(2,n)$ at $\alpha(s_0)$.}
\]
Let $t_k$ be as before. 
\[
\alpha(s_0)=\lim_{k\to\infty}\alpha_k(t_k)=\lim_{k\to \infty}\tilde \alpha_k(t_k-[t_k]L)=-e_1\wedge e_2,
\]
where $e_1$ and $e_2$ are those in Theorem \ref{neck.first.order} when we apply it to $\tilde f_k$. By \eqref{eqn:d2alpha}, it remains to show
\[
\lim_{k\to \infty} \frac{\ddot \alpha_k(t_k)}{(\tilde \epsilon_k)^2}\quad \text{is a normal vector of $G(2,n)$ at $-e_1\wedge e_2$.}
\]
To see this, we use the equation
\begin{equation*}
    \frac{\ddot \alpha_k(t_k)}{(\tilde \epsilon_k)^2}  = \left(\frac{1}{2\pi}\int_0^{2\pi} A_{G(2,n)}(\frac{d\tilde \n_k}{\tilde \epsilon_k}, \frac{d\tilde \n_k}{\tilde \epsilon_k}) + e^{-2\tilde u_k}\frac{1}{(\tilde \epsilon_k)^2} (\nabla^\perp_t \tilde H_k \wedge \partial_\theta \tilde f_k - \nabla^\perp_\theta \tilde H_k \wedge \partial_t \tilde f_k)  d\theta \right)(t_k-[t_k]L).
\end{equation*}
For the first term, the limit is a normal vector at $-e_1\wedge e_2$ because $A_{G(2,n)}$ is continuous and that $\tilde \n_k(t_k-[t_k]L,\theta)$ converges uniformly to $-e_1\wedge e_2$. For the second term, we use the Part 2) of Lemma \ref{normal term}.

The claim about the length has been proved in the previous subsection and the proof of Lemma \ref{lem:geodesic} is done.
\end{proof}

\section{An example of Willmore surface}
In this section, we present a family of Willmore surfaces which is model case of what we have proved in this paper.

\subsection{A Willmore surface in $\Real^3$}

The surface is defined on $\Real^2 \setminus \set{0}$ and we use the cylinder coordinate $(t,\theta)$ where $t=\log \abs{x}$. For some parameter $l>0$, define a skew-symmetric matrix
\[
A=
\left(
\begin{array}{ccc}
    0 & 0 & -l  \\
    0 & 0 & 0 \\
    l & 0 & 0
\end{array}
\right).
\]
The surface is parametrized by
\begin{equation}
    \label{eqn:ex3}
    f(t,\theta)= \exp(t A) (e^{-t} \cos \theta, e^{-t} \sin \theta, 0)
\end{equation}
for $(t,\theta)\in \Real \times S^1$. With the simple choice of $A$, we obtain
\[
	f(t,\theta)= (e^{-t} \cos(lt) \cos \theta, e^{-t} \sin \theta, e^{-t} \sin(lt)\cos \theta).
\]
Note that $\exp(tA)$ is a one parameter family of rotations in the direction of $A$ and $l$ is the speed of rotation.

The parametrization is not a conformal one, but we can compute the first fundamental form 
\[
	\left(
	\begin{array}{cc}
		f_t \cdot f_t & f_t \cdot f_\theta \\
		f_\theta \cdot f_t & f_\theta\cdot f_\theta
	\end{array}
	\right)
	=
	\left(
	\begin{array}{cc}
		e^{-2t}(1+ l^2 \cos^2 \theta) & 0 \\
		0 & e^{-2t}
	\end{array}
	\right).
\]
The normal vector is given by
\begin{eqnarray*}
	\n&=&\frac{f_t \times f_\theta}{\abs{f_t \times f_\theta}}\\
		&=&\left(- \frac{l \cos(lt)\cos^2 \theta - \sin (lt)}{\sqrt{1+l^2\cos^2 \theta}},  -\frac{l\cos\theta\sin\theta}{\sqrt{1+l^2\cos^2 \theta}}, -\frac{\cos(lt)+l\sin(lt)\cos^2\theta}{\sqrt{1+l^2\cos^2 \theta}}\right).
\end{eqnarray*}
The second fundamental form is
\begin{eqnarray*}
	A=-
	\left(
	\begin{array}{cc}
		f_t\cdot \n_t & f_t \cdot \n_\theta \\
		f_\theta \cdot \n_t & f_\theta \cdot \n_\theta
	\end{array}
	\right)
	=
	\left(
	\begin{array}{cc}
		e^{-t} l \cos \theta \sqrt{1+l^2\cos^2 \theta} & \frac{e^{-t} l \sin\theta}{\sqrt{1+l^2\cos^2 \theta}} \\
		\frac{e^{-t} l \sin \theta}{\sqrt{1+l^2\cos^2 \theta}} & \frac{e^{-t} l\cos \theta}{\sqrt{1+l^2\cos^2 \theta}}
	\end{array}
	\right)
\end{eqnarray*}
The mean curvature function is
\[
	H=\frac{2e^{t}l \cos\theta}{\sqrt{1+l^2 \cos^2 \theta}}.
\]
We may also verify by direct computation that 
\[
	\Delta H+ \abs{\mathring{A}}^2 H=0.
\]
and hence it is indeed a Willmore surface.

\subsection{The model case of the neck with nonvanishing residue}

For each $k\in \mathbb N$, let $f_k$ be the map given by \eqref{eqn:ex3} with $l=\frac{1}{k}$. Obviously, $f_k$ is a sequence of Willmore surfaces defined on $\Real \times S^1$ satisfying A1)-A3).

We can find the expansion of the some quantities in terms of the power series of $l=k^{-1}$. 
\begin{itemize}
    \item  The parametrization of the surface
\[
f_k(t,\theta) =e^{-t} (\cos \theta, \sin\theta,0) + t e^{-t} (0,0,\cos\theta) k^{-1} + O(k^{-2}).
\]
\item The (scalar) mean curvature
\[
	H_k=  (2 e^{t} \cos \theta) k^{-1}  + O(k^{-3}).
\]
\item The normal vector
\[
	\n_k= (0,0,-1) - (-t+\cos^2 \theta, \sin\theta \cos\theta,0) k^{-1} + O(k^{-2}).
\]
\item The mean curvature vector
\[
	\vec{H}_k= (0,0,-2e^{t}\cos\theta) k^{-1} - 2e^{t}\left(- t\cos\theta+\cos^3\theta, \sin\theta \cos^2\theta,0 \right) k^{-2} + O(k^{-3}).
\]
\item The normalizing constant $\epsilon_k$
\begin{align*}
&\int_{[0,L]\times S^1} \abs{H_k}^2 dV_g = 4\pi L k^{-2}+ O(k^{-3}),\\
& \epsilon_k= 2 k^{-1} + O(k^{-2}).
\end{align*}
\item  
\[
\frac{\n_k-\bar \n_k}{\epsilon_k}= \frac{1}{2}t(1,0,0) -\frac{1}{4}(\cos 2\theta,\sin 2\theta,0)  + O(k^{-1}).
\]
\item The Second fundamental form 
\begin{eqnarray*}
    \hat A_{tt}, \hat A_{\theta\theta} &=& e^{-t}\cos \theta \cdot k^{-1} + O(k^{-3}) \\
    \hat A_{t\theta} &=& e^{-t} \sin \theta \cdot k^{-1} + O(k^{-3}).
\end{eqnarray*}
\end{itemize}
By taking $v_1=(0,0,-1)$ and $v_2=(0,0,0)$, we verify that the above results agree with Theorem \ref{neck.first.order} and Theorem \ref{2nd.asymp.}.

\subsection{Generalizations}

In dimension three, we may also consider some more general matrix $A$. For example, we may introduce one more parameter $\beta\in S^1$ as follows
\[
A=
\left(
\begin{array}{ccc}
    0 & 0 & l\cos \beta  \\
    0 & 0 & l\sin \beta \\
    - l\cos\beta & -l\sin\beta & 0
\end{array}
\right).
\]
It is not difficult to see that \eqref{eqn:ex3} still gives the same surface up to a rotation of $\Real^3$. Moreover, we may consider a covering of the above surface. For example, for $m\in \mathbb N$, the surface 
\[
\tilde{f}(t,\theta)=f(m t, m\theta)
\]
is also Willmore. They also arise as model case in the neck analysis.

We may also consider similar construction in dimension four. In this case, $\dim G(2,4)=4$. The infinitesimal rotation (including the speed) is a vector space of dimension four. Let
\[
	f(t,\theta):=\exp(tA) \left( e^{-mt}\cos m\theta, e^{-mt}\sin m\theta, 0,0 \right)
\]
where
\begin{equation}
	\label{eqn:rotation4}
	A=
	\left(
	\begin{array}{cccc}
		0 & 0 & a & b \\
		0 & 0 & c & d \\
		-a & -c & 0 & 0 \\
		-b & -d & 0 & 0
	\end{array}
	\right).
\end{equation}
It is not easy to verify that this is indeed a Willmore surface in $\R^4$. The problem is the expression of $\exp (tA)$ is too difficult to be computed by hand. However, we did manage to verify the Willmore equation with the help of the computer software: Mathematica and the symmetry of surface. It is worth noting that the software produces symbolic computation, not numerical simulation. Therefore the result is rigorous. 

\begin{rem}
	We definitely want to have the upper-left corner to be zero. If otherwise, the first fundamental form is no longer diagonal. It is a rotation in the plane itself, not interesting. 
\end{rem}
\begin{rem}
	When the upper-left corner is not zero, obviously, our computation fails. But we have no proof of the claim that "they are not Willmore". 	When the lower-right corner is not zero, the formulas used in this computation works, but we haven't run the computer program for it (because it takes too long). Again, we can't say they are not Willmore.
	Our guess is that when they are not zero, they will not give essentially new surfaces. 
\end{rem}

\appendix 

\section{The compactness of Willmore surfaces}
In the first subsection below, we briefly explain why we may assume A1) and A2) in the study of energy idenity problem of Willmore surfaces. In the second subsection, we give consequences of the assumption A3), that will be used in the proof of the main theorem. In the final subsection, we prove a basic lemma about the structure of the curvature of the mean curvature and state without proof some formulas of the residues $\tau_1$ and $\tau_2$ under various transformations.

\subsection{Estimate on the conformal factor}

A natural setting for the compactness problem is to consider a sequence of closed surfaces $\Sigma_k$ with bounded $\norm{A_k}_{L^2}$. Simon's monotonicity inequality \cite{Simon} ensures the existence of $\Lambda$ such that the following holds for each $f_k$ uniformly
\begin{equation}\label{noncollapsing}
\frac{\mu(f^{-1}(B_r^n(y)))}{\pi r^2}<\Lambda,\s \forall y\in\R^n,\s r\in\R^+.
\end{equation}

The upper bound \eqref{noncollapsing} allows us to obtain uniform upper bound for the gradient of the conformal factor. Here we present a proof using the blowup technique as in \cite{Li-Wei-Zhou}.
\begin{lem}
Let $f$ be a conformal and Willmore immersion from $D$ into $\R^n$. Assume \eqref{noncollapsing} holds. Then there exist constants $C$ and $\epsilon_0'$, such that if $\int_D|A|^2 d\mu<\epsilon_0'$, then
$$
\|\nabla u\|_{C^0(D_\frac{1}{2})}<C.
$$
\end{lem}

\begin{proof}
Assume the Lemma is not true. Then we can find $f_k$ such that $\int_D|A_k|^2d\mu_k\rightarrow 0$ and
$$
\|\nabla u_k\|_{C^0(D_\frac{1}{2})}\rightarrow+\infty.
$$
Set 
$$
\rho_{k}=\sup_{x \in D_{2/3}}(2/3-|x|)|\nabla u_k|(x).
$$
Then we can find $x_k\in D_\frac{2}{3}$, such that $\rho_k=(\frac{2}{3}-|x_k|)  |\nabla u_k|(x_k)$. Define 
$r_k=1/|\nabla u_k(x_k)|$.

Since 
$$
\rho_k\geq(2/3-1/2)\|\nabla u_k\|_{C^0(D_\frac{1}{2})}\rightarrow+\infty,
$$
we have% $r_k=(|2/3-|x_k|)\frac{1}{\rho_k}\rightarrow 0$, and
$$
\frac{r_k}{2/3-|x_k|}=\frac{1}{\rho_k}\rightarrow 0,
$$
which implies that $D_{Rr_k}(x_k)\subset D_\frac{2}{3}$ for any fixed  $R$, when $k$
is sufficiently large. 

Now consider a scaling and translation of $f_k$ by
\[
\tilde f_k(x) = e^{-u_k(x_k)-\log r_k}(f_k(x_k+r_k x)-f_k(x_k))
\]
and $v_k$ is defined by $g_{F_k}=e^{2v_k}g_{euc}$, i.e.
\[
v_k(x)= u_k(r_k x + x_k) - u_k(x_k).
\]
$\tilde f_k$ and $v_k$ are defined on any $D_R$ as long as $k$ is large. We claim that for large $k$
\[
\sup_{x\in D_R}\abs{\nabla v_k}(x) \leq 2.
\]
In fact, for large $k$ and any $x\in D_R$, $x_k+ r_k x \in D_{\frac{2}{3}}$, therefore
$$
(\frac{2}{3}-|x_k|) |\nabla u_k|(x_k)\geq (2/3-|x_k+r_k x|)|\nabla u_k|(x_k+ r_k x),
$$
which yields that
\begin{eqnarray*}
\abs{\nabla v_k}(x) &=& r_k|\nabla u_k(x_k+r_kx)| \\
&\leq& \frac{\frac{2}{3}-|x_k|}{2/3-|x_k+r_kx|} \\
&\leq& \frac{\frac{2}{3}-|x_k|}{\frac{2}{3}-|x_k|-r_k R} \\
&=&\frac{1}{1-\frac{r_k}{2/3-|x_k|}R}<2\\
&\leq& 2. 
\end{eqnarray*}
when $k$ is sufficiently large.

Applying the $\epsilon$-regularity, we may assume $\tilde f_k$ converges in $C^\infty_{loc}(\R^2)$ to a conformal and Willmore immersion $\tilde f$ from $\C$ into $\R^n$ with $g_{\tilde f}=e^{2\tilde v}g_{euc}$. It is easy to check that $A_{\tilde f}=0$ and $|\nabla \tilde v|<2$, $|\nabla \tilde v|(0)=1$. By Gauss curvature equation, $\tilde v$ is harmonic and therefore $\tilde v(x)=ax^1+bx^2$ with $a^2+b^2=1$. Choosing new coordinates, we may assume $\tilde f=e^{z}$ which is a map from $\R^2$ to $\R^2$. We can choose $R$, such that 
$$
\frac{\mu_{f_k}(f_k(D_{Rr_k}(x_k))\cap B_{Rr_k}^n(x_k))}{\pi (Rr_k)^2}=\frac{\mu_{\tilde f_k}(\tilde f_k(D_R)\cap B_R^n(0))}{\pi R^2}\rightarrow\frac{\mu_{\tilde f}(\tilde f(D_R)\cap B_R^n(0))}{\pi R^2}>2\Lambda,
$$
which leads to a contradiction to \eqref{noncollapsing}.
\end{proof}
\begin{rem}
The same estimate as above can be proved using some uniform on the gradient of the Green's function (see \cite{L-R2}).
\end{rem}

With this control over $\abs{\nabla u}$, we can use (after some proper scaling) the famous $\epsilon$-regularity theorem for the Willmore equation:
\begin{thm}\cite{K-S1,R}
Let $f$ be a conformal and Willmore immersion from $D$ into $\R^n$ with $g=e^{2u}g_{euc}$. Assume $|u|<\beta$. Then, there exists $\epsilon_0>0$, such that if $\int_D|A|^2d\mu<\epsilon_0$, then 
$$
\|\nabla^k \n\|_{L^\infty(D_\frac{1}{2})}<C\|A\|_{L^2(D)}
$$
for some constant $C$ depending on $\beta$.
\end{thm}

Then by a standard bubble tree arguments (\cite{Chen-Li} for example), we can divide $\Sigma_k$ into finitely many parts $\Sigma_k^1$, $\cdots$, $\Sigma_k^{j_0}$, $\Sigma_k^{j_0+1}$, $\cdots$,
$\Sigma_k^{j_1}$, $\Sigma_k^{j_0+1}$, $\cdots$,
$\Sigma_k^{j_2}$, such that
\begin{itemize}
\item [1)] After tanslation and rescaling, $\Sigma_k^i$ converges to a nontrivial Willmore surface in $C^\infty_{loc}(\R^n\setminus\mathcal S_i)$ for $i=1$, $\cdots$, $j_0$, where $\mathcal S_i$ is a finite set.
\item[2)]  After tanslation and rescaling, $\Sigma_k^i$ converges to a plane in $C^\infty_{loc}(\R^n\setminus\mathcal S_i)$ for $i=j_0+1$, $\cdots$, $j_1$. (We usually call them `ghost bubbles').
\item[3)] For $j>j_1$, $\Sigma_k^j$ is a image of a conformal map $f_{kj}$ from $[0,T_k]\times S^1$ into $\R^n$, such that $f_{kj}$ has no bubbles. (We usually call them the `neck' part). 
\end{itemize}

As a consequence of the above construction, to prove the energy identity, it is suffice to study a conformal and Willmore immersion $f_k$ from $[0,T_k]\times S^1$ into $\R^n$ under the assumption that $f_k$ has no bubble. The assumption A2) is a consequence of this no more bubble assumption.

For a sequence of $f_k$ satisfying A2), similar to Lemma 5.2 of \cite{Li-Yin}, by shrinking the length of the cylinder by a finite amount, we may assume
\[
\norm{\nabla v_k}_{L^\infty([0,T_k]\times S^1)}\to 0.
\]
By a scaling of $f_k$ if necessary, we have A1).

\subsection{Consequences of the assumption A3)}
For any sequence $f_k$ satisfying A1) and A2), we may always assume the extra A3) by a translation of $f_k$, because all the results in the main theorems are invariant under translations. Nonetheless, A3) is necessary if we want to obtain good control over $f_k$.

\begin{lem}\label{mean.value}
Assume $f_k$ is a sequence of conformal and Willmore immersions from $[0,m_kL]\times S^1\rightarrow\R^n$ satisfying A1) and A2). Let 
$$
u_k^*=\frac{1}{2\pi}\int_0^{2\pi}u_k(t,\theta)d\theta,\s and\s 
f^*=\frac{1}{2\pi}\int_0^{2\pi}f_k(t,\theta)d\theta.$$
Then 

(1) for any $t\in [0, m_k L]$,
\[
u_k(t,\theta)\le -\frac{mt}{2}+2.
\]
In particular
\[
\text{Vol}(f_k,[0,m_kL]\times S^1)\le \frac{2 \pi e^4}{m}.
\]

(2) for any $1<t<T\leq m_kL -1$, we have
\begin{equation}
    \label{eqn:mapcenter}
e^{-u_k^*(t)}|f_k^*(t)-f_k^*(T)|\le C[\sqrt{\Theta_k}+(T-t)e^{-\frac{m}{2}(T-t)}],
\end{equation}

\end{lem}
\begin{proof}

(1) By A1), $(u^*_k)'<-\frac{m}{2}$ when $k$ is sufficiently large, which implies that
\begin{equation}
    \label{eqn:ukst}
u_k^*(s)-u_k^*(t)\leq -\frac{m}{2}(s-t),\s \forall t<s.
\end{equation}
Also by A1),  we have 
\begin{align*}
u_k(s,\theta)\le u_k^*(s)+\text{osc}_{\{s\}\times S^1}v_k\le-\frac{ms}{2}+ u_k^*(0) +2\pi\|\nabla v_k\|\le  -\frac{ms}{2}+2, 
\end{align*}
and hence 
\begin{align*}
\text{Vol}(f_k,[0,m_kL]\times S^1)=\int_{[0,m_kL]\times S^1}e^{2u_k} dtd\theta\le \frac{2\pi e^4}{m}.
\end{align*}

(2)
By the $\epsilon$-regularity of Willmore surface, we have $|e^{u_k}H_k|(t,\theta) <C\sqrt{\Theta_k}$ for $t\in [1,m_kL-1]$, which implies that for $1\leq t\leq T\leq m_kL-1$,
\begin{eqnarray*}
|\partial_t f_k^*(t)-\partial_tf_k^*(T)|&=&|\int_t^T\int_0^{2\pi}\Delta f d\theta ds|\\
&=&|\int_t^T\int_0^{2\pi}e^{2u_k}H_k d\theta ds| \leq C\sqrt{\Theta_k}\int_t^{T}e^{u_k^*(s)} ds.
\end{eqnarray*}
Here in the last line above, we have used the fact that $\abs{\partial_\theta u_k}\leq C$.
By \eqref{eqn:ukst},
\begin{eqnarray*}
|\partial_t f_k^*(t)|
&\leq& %Ce^{u_k^*(t)}\sqrt{\Theta_k}\int_t^Te^{u_k^*(s)-u_k^*(t)}ds
Ce^{u_k^*(t)}\sqrt{\Theta_k}\int_t^Te^{-\frac{m}{2}(s-t)}ds + \abs{\partial_tf_k^*(T)}\\
&\leq& Ce^{u_k^*(t)}\sqrt{\Theta_k}+ \abs{\partial_tf_k^*(T)}.
\end{eqnarray*}
Using the definition of $f_k^*$, we estimate the second term above as
\begin{eqnarray*}
    \abs{\partial_tf_k^*(T)} &\leq& C \int_0^{2\pi} \abs{\partial_t f_k}(T,\theta) d\theta \\
    &\leq& C u_k^*(T).
\end{eqnarray*}
Integrating over $[t,T]$, we obtain
$$
|f_k^*(t)-f_k^*(T)|\leq C\sqrt{\Theta_k}\int_t^Te^{u_k^*(s)}ds+ C u_k^*(T)(T-t).
$$
Multiplying both sides by $e^{-u_k^*(t)}$ and using \eqref{eqn:ukst} again, we have
\begin{eqnarray*}
e^{-u_k^*(t)}|f_k^*(t)-f_k^*(T)|&\leq& C\sqrt{\Theta_k}\int_t^Te^{u_k^*(s)-u_k^*(t)}ds+ C e^{u_k^*(T)-u_k^*(t)}(T-t)\\
&\leq& C\sqrt{\Theta_k}\int_t^Te^{-\frac{m}{2}(s-t)}ds+
C (T-t)e^{-\frac{m}{2}(T-t)},
\end{eqnarray*}
which yields the lemma.
\end{proof}

\begin{cor}\label{cor:lambdak}
Assume $f_k$ is a sequence of conformal and Willmore immersions from
$[0,m_kL]\times S^1\rightarrow\R^n$ satisfying A1)-A3).
Then, we have 
\begin{equation}
    \label{eqn:diam}
\lim_{k\to \infty}\max_{1\le t\le (m_k-1)L}e^{-u_k^*(t)} |f_k^*(t)|=0,
\end{equation}
and 
\begin{equation}
    \label{Lambda_k}
\lim_{k\to \infty}\sup_{2\le i\le m_k-2}\int_{Q_i}e^{-2u_k}|f_k-\G_i(f_k)|^2 dtd\theta=0.
\end{equation}
\end{cor}
\begin{proof} We first argue by contradiction to show \eqref{eqn:diam}. If \eqref{eqn:diam} is not true, then there exists $t_k\in [1,(m_k-1)L)$ such that 
$$e^{-u_k^*(t_k)} |f_k^*(t_k)| \to \eta>0.$$
If $m_kL-t_k\to \infty$, then \eqref{eqn:diam} follows from A3) and  Lemma \ref{mean.value} by taking $t=t_k$ and $T=m_k L-1$.  Hence, we may assume that $\sup_{k}(m_k L-1)-t_k<+\infty$. Without loss of generality, we assume $s_k:=(m_kL -1)-t_k\to s_\infty$. Then by A1)-A3) and the $\epsilon$-regularity of Willmore immersions,  we know that
$$e^{-u_k^*(t_k)}f_k(t_k+t,\theta)\to cRf_\infty+v, \text{  on } [-1,s_\infty+\frac{1}{2}]\times S^1$$
for some  $c> 0, R\in SO(n), v\in \mathbb{R}^n$, where  $f_\infty(t,\theta)=\frac{1}{m}e^{-mt}(e_1\cos m\theta+e_2\sin m\theta)$.  Noticing that
$$\lim_{k\to \infty} |e^{-u_k^*(t_k)}\partial_t f_k(t_k,\theta)|=e^{u_k(t_k,\theta)-u_k^*(t_k)}=1,$$
we find that $c=1$.  Moreover, by A3) and the fact that $t_k+s_k= m_k L -1$, we have
$$\lim_{k\to \infty}e^{-u_k^*(t_k)}f^*_k(t_k+s_k)  = v =0.$$
This is a contradiction, because
$$0<\eta=\lim_{k\to\infty} e^{-u_k^*(t_k)}|f_k^*(t_k)|= |Rf_\infty^*(0)|=0.$$
Thus, we have proved \eqref{eqn:diam}

For the proof of \eqref{Lambda_k}, it suffices to show that for any sequence $i_k\in [2,m_k-2]$, there holds 
\begin{equation}
    \label{eqn:fkG}
\lim_{k\to \infty}\int_{Q_{i_k}}e^{-2u_k}|f_k-\G_{i_k}(f_k)|^2 dtd\theta=0.
\end{equation}
For this purpose, we consider the same scaling as above. Namely, by setting $t_k=(i_k-1)L$, we consider
\[
\tilde f_k (t,\theta) = e^{-u^*_k(t_k)} f_k(t+t_k,\theta).
\]
The difference is that we now have \eqref{eqn:diam}, with the help of which we may derive from A1) and A2) that
\[
\tilde f_k \to cR f_\infty \quad \text{on} \quad [-L,L],
\]
for $c,R$ and $f_\infty$ as before. For the same reason as above, we have $c=1$.
On the other hand,
\[
\int_{Q_{i_k}}e^{-2u_k}|f_k-\G_{i_k}(f_k)|^2  dtd\theta= \int_{Q_1} e^{-2\tilde u_k} \abs{\tilde f_k- \G_1(\tilde f_k)}^2 dtd\theta.
\]
Since $f_\infty -\G_1(f_\infty)=0$, the limit of the right hand side vanishes. Hence \eqref{eqn:fkG} (and therefore \eqref{Lambda_k}) is proved.
\end{proof}

\subsection{The equation of mean curvature}

The following lemma shows how the equation for the mean curvature of a Willmore surface is of the form of the nonlinear equation discussed in Section \ref{sub:nonlinear}.

\begin{lem}\label{equ. struc.} Let $f$ be a conformal and Willmore immersion from $[a,b]\times S^1\rightarrow\R^n$ and $u$ be defined by $df\otimes df=e^{2u}(dt^2+d\theta^2)$. Assume that  
$$ u\le C, \s \|\nabla u\|_{L^\infty([a,b]\times S^1)}\le C\s \text{ and } \int_{[a,b]\times S^1}|A|^2dV_g\le \epsilon $$
for some $C>0$ and $\epsilon$ sufficiently small.
Then 
$$
|\Delta  H|\leq \alpha \cdot (|H|+ |\nabla H|)
$$
with 
$$
\|\alpha\|_{L^\infty([a+\delta,b-\delta]\times S^1)}<C(\delta) \|A\|_{L^2([a,b]\times S^1)}
$$
for small $\delta>0$.
\end{lem}
\begin{proof}
Recall that the Willmore equation in conformal coordinates is (see Section 2 of \cite{Li-Yin})
\begin{align*}
\Delta H=-2\partial_q(H\cdot A_{pq}  e^{-2u}f_p)+\frac{1}{2}\partial_q(|H|^2f_q).
\end{align*}
The right hand side of the above equation is 
\begin{align*}
rhs:=-2\partial_q(H\cdot A_{pq} e^{-2u})f_p-2  H\cdot A_{pq}e^{-2u}f_{pq}+H\cdot H_qf_q+\frac{1}{2}|H|^2H e^{2u},
\end{align*}
where we have used $\Delta f=H e^{2u}$.

For $\epsilon$ small, we can use the $\epsilon$-regularity to bound
\[
\sup_{[a+\delta,b-\delta]\times S^1} \abs{H}+ \abs{A_{pq}} + \abs{\partial H} + \abs{\partial A_{pq}} \leq C(\delta) \norm{A}_{L^2([a,b]\times S^1)}.
\]
Due to the upper bound of $u$, the first derivatives of $f$ are bounded. For the second order derivatives of $f$, we note that
\[
f_{pq}=A_{pq}+\langle f_{pq},f_l\rangle e^{-2u}f_l.
\]
Taking derivative of the equations $\langle f_1, f_1\rangle =e^{2u}$, we can bound
\[
\langle f_{11},f_1 \rangle,\quad \langle f_{12}, f_1 \rangle.
\]
Moreover, due to $\langle f_1, f_2 \rangle =0$, we have $\langle f_{11},f_2\rangle = - \langle f_1, f_{12}\rangle$. Switching $1$ and $2$ in the subscript, we are able to bound all $\langle f_{pq},f_l\rangle$, hence $f_{pq}$.
The lemma follows easily from the explicit expression of $rhs$ and the above bounds on $H$, $A$, $u$, $f$ and their derivatives.
\end{proof}

By direct calculation, we get the following transformation law on the residues. 
\begin{lem} \label{transformation} For any $v,c\in \mathbb{R}^n$ and $R\in SO(n), S\in \so$,  we have 
 \begin{align}\label{tau1 trans}
\tau_1(f+v,c)=\tau_1(f,c), \s \tau_1(\lambda f, c)=\tau_1(f,\lambda^{-1}c), \s\tau_1(Rf, c)=\tau_1(f,R^{-1}c),
\end{align}
\begin{align}\label{tau2trans}
 \tau_2(f+v,S)=\tau_2(f,S)+\tau_1(f,Sv), \s \tau_2(\lambda f, S)=\tau_2(f,S), \s \tau_2(Rf,S)=\tau_2(f,Ad_R(S)),
 \end{align}
 where $Ad_R(S)=R^{-1}SR$ is a adjoint action of $SO(n)$ on $\so$. 
\end{lem}

\section{Gauss maps and their tension fields }
In this appendix, we collect some formulas and properties on the Grassmannian $G(2,n)$, among which is the explicit expression for Ruh and Vilms' formula \cite{Ruh-Vilms} in an isothermal  coordinate system.

Let $\Lambda^p(\R^n)$ be the $p$-th wedge product of $\R^n$. With the standard inner product of $\R^n$, we can define one for $\Lambda^p(\R^n)$ by asking
\[
e_{i_1}\wedge \cdots \wedge e_{i_p}, \qquad i_1<\cdots<i_p
\]
to be an orthonormal basis of $\Lambda^p(\R^n)$ if $(e_i)$ is an orthonormal basis of $\R^n$. In particular, the inner product of $\Lambda^2(\R^n)$ is given by the formula
$$
\lan v_1\wedge w_1,v_2\wedge w_2\ran=\lan v_1,v_2\ran\lan w_1,w_2\ran-\lan v_1,w_2\ran\lan v_2,w_1\ran.
$$

Given $P\in G(2,n)$, an oriented plane in $\R^n$, Let $(e_1,e_2)$ be an orthonormal basis of $P$ that agrees with the orientation of $P$. The map
$$
P\rightarrow e_1\wedge e_2
$$
gives an isometric embedding from $G(2,n)$ into  $\Lambda^2(\R^n)\cong \R^{\frac{n(n-1)}{2}}$. Throughout the paper, we often use $e_1\wedge e_2$ to represent a point in $G(2,n)$. 

For $P=e_1\wedge e_2\in G(2,n)$, assume that $e_1,\cdots,e_n$ is a choice of orthonormal basis of $\R^n$. Then the tangent space $T_P G(2,n)$ is spanned by
$$
\{e_i\wedge e_\alpha |\, i=1,2,\s \alpha=3,\cdots,n\}.
$$
To see this, we know the dimension $\dim G(2,n)=2(n-2)$ on one hand, while on the other hand, $\gamma(t)=e_1\wedge(e_2\cos t+e_\alpha \sin t)$ is a curve in $G(2,n)$ with
$$
\gamma(0)=e_1\wedge e_2,\s \dot\gamma(0)=e_1\wedge e_\alpha,
$$
which implies that $e_1\wedge e_\alpha$ for any $\alpha=3,\cdots,n$ is in  the tangent space.
Using the formula of inner product, we derive that the normal space at $P$ is spanned by
$$
\{e_1\wedge e_2, e_\alpha\wedge e_\beta |\, \alpha,\beta = 3, \cdots, n\}.
$$

Now, we let $f:D\rightarrow\R^n$ be a conformal map with $g=e^{2u}g_{euc}$, and set $\Sigma=f(D)$. % We will use the Gauss map of $f$ to calculate the second fundamental form of $f$. 
For each $x\in D$,  let $e_1=e^{-u}f_1$, $e_2=e^{-u}f_2$ and choose $\set{e_3,\cdots,e_n}$ such that $\{e_1,e_2,\cdots,e_n\}$ is an orthonormal basis of $\R^n$ that depends on $x$ smoothly. In what follows, we use the Latin letters for $1$ or $2$ and Greek letters for $3,\cdots,n$.
The Gauss map $\n$ of $f$ and the second fundamental form are given by
$$
\n=e_1\wedge e_2,\s A_{ij}=(f_{ij})^\bot=\lan f_{ij},e_\alpha \ran e_\alpha=-\lan f_i,\frac{\partial e_\alpha}{\partial x^j}\ran e_\alpha.
$$
Moreover, we have
$$
f_{ij}= e^{-2u}\lan f_{ij},f_k\ran f_k+A_{ij},\s \frac{\partial e_\alpha}{\partial x^j}=B_{j\alpha}^\beta e_\beta - e^{-2u}A_{jk}^\alpha f_k,
$$
where $B_{i\alpha}^\beta$ are the coefficients for the induced connection in the normal bundle. Then
\begin{eqnarray}\label{der.G}
\frac{\partial \n}{\partial x^i}&=&\frac{\partial e_1}{\partial x^i}\wedge e_2+e_1\wedge\frac{\partial e_2}{\partial x^i}\\\nonumber
&=&\lan \frac{\partial e_1}{\partial x^i},e_\alpha \ran e_\alpha \wedge e_2+e_1\wedge\lan\frac{\partial e_2}{\partial x^i},e_\alpha \ran e_\alpha \\\nonumber
&=&A(f_i,e_1)\wedge e_2+e_1\wedge A(f_i,e_2)\\\nonumber
&=&e^{-2u}(A_{i1}\wedge f_2+f_1\wedge A_{i2}).\nonumber
\end{eqnarray}

Let $J$ is the complex structure on $\Sigma$ (induced from that on $D$), i.e.
$$
J(f_1)=f_2,\s J(f_2)=-f_1.
$$
It follows from \eqref{der.G} that
$$
\frac{\partial \n}{\partial x^i}=A_{ip}\wedge J(f_q)g^{pg}
$$
Next, we take a normal frame of $T\Sigma$ at $x$, denoted by $(\tilde e_1, \tilde e_2)$. Namely, they are orthonormal at $T_x \Sigma$ and $\nabla_{\tilde e_i} \tilde e_j (x)=0$. We write $\tilde A_{ij}$ for $A(\tilde e_i,\tilde e_j)$. Then the above equation becomes
\[
\tilde e_i (\n) = \tilde A_{ip} \wedge J(\tilde e_p).
\]
Using $\nabla_{\tilde e_i} \tilde e_j (x)=0$, we have, at $x$,
\[
\Delta_g \n (x) = \tilde e_i ( \tilde A_{ip} \wedge J(\tilde e_p)).
\]
Here we have taken $\tilde A_{ip}$, $J(\tilde e_p)$ as $\R^n$-valued functions defined on $\Sigma$ and $\tilde e_i$ acts on them by taking directional derivatives. We use $\pi_1$ and $\pi_2$ to denote the projections to the tangent and the normal directions of the surface (at a fixed point), then for any vector $V\in \R^n$, $V=\pi_1(V)+\pi_2(V)$. For our purpose, we are interested in the projection of $\Delta_g \n(x)$ onto the tangent space of $G(2,n)$ at $\n(x)$. Keeping in mind that for any $V,W\in \R^n$, $\pi_2(V)\wedge \pi_2(W)$ and $\pi_1(V)\wedge \pi_1(W)$ are in the normal space, we have
\begin{eqnarray*}
\text{the tangent projection of}\, \Delta_g \n (x) &=& \nabla_{\tilde e_i}^\perp \tilde A_{ip} \wedge J(\tilde e_p) + \tilde A_{ip} \wedge \nabla_{\tilde e_i}(J\tilde e_p).
\end{eqnarray*}
Here $\nabla^\perp$ and $\nabla$ are the induced connections of the normal and the tangent bundle of the surface $\Sigma$ respectively. 
Using the Codazzi-Mainardi equation $\nabla^\perp_{\tilde e_i} \tilde A_{ip}(x) = \nabla^\perp_{\tilde e_p} H (x)$, the fact that $J$ is parallel and $\nabla_{\tilde e_i} \tilde{e_p}(x)=0$, we derive the formula for the tension field of $\n$
\[
\tau(\n)(x) = \text{the tangent projection of}\s \Delta_g \n (x) = \nabla^\perp_{\tilde e_p} H \wedge J(\tilde e_p).
\]
Or equivalently, using the second fundamental form of $G(2,n)$ in $\Lambda^2(\R^n)$,
\begin{equation}\label{tension.G}
\tau(\n)=\Delta \n-A_{G(2,n)}(d\n,d\n)=e^{-2u}\left( \nabla^\bot_1H\wedge f_2-\nabla^\bot_2H\wedge f_1\right).
\end{equation}

\bibliographystyle{alpha}
\bibliography{ref}

\end{document}